\documentclass{amsart}
\usepackage{amsmath,amssymb,amscd,mathrsfs,epic,wasysym,latexsym,tikz,mathrsfs,cite,hyperref}
\usepackage{pb-diagram}
\usepackage[matrix,arrow]{xy}
\usepackage[letterpaper,%
            left=1.2in,right=1.2in,top=1.2in,bottom=1in,%
            footskip=.25in]{geometry}
\usepackage{bbm}
\usepackage{tikz-cd}

\makeatletter

\numberwithin{equation}{section}

\newtheorem{Proposition}[equation]{Proposition}
\newtheorem{Lemma}[equation]{Lemma}
\newtheorem{Theorem}[equation]{Theorem}
\newtheorem{Corollary}[equation]{Corollary}
\theoremstyle{definition}  
\newtheorem{Definition}[equation]{Definition}
\newtheorem{Remark}[equation]{Remark}

\newtheorem{Example}[equation]{Example}

\newtheorem{Conjecture}[equation]{Conjecture}

\let\<\langle
\let\>\rangle

\newcommand\Comment[2][\relax]{\space\par\medskip\noindent%
   \fbox{\begin{minipage}{\textwidth}\textbf{Comment\ifx\relax#1\else---#1\fi}\newline%
        #2\end{minipage}}\medskip
}

\def\bi{\text{\boldmath$i$}}
\def\bj{\text{\boldmath$j$}}

\def\bk{\text{\boldmath$k$}}

\def\bt{\text{\boldmath$t$}}
\def\bu{\text{\boldmath$u$}}
\def\bc{\text{\boldmath$c$}}

\def\b1{\text{\boldmath$1$}}

\def\ba{\text{\boldmath$a$}}
\def\bb{\text{\boldmath$b$}}
\def\bv{\text{\boldmath$v$}}

\newcommand{\Hom}{\operatorname{Hom}}

\newcommand{\End}{\operatorname{End}}

\newcommand{\im}{\operatorname{im}}
\newcommand{\id}{\operatorname{id}}

\newcommand{\cha}{\operatorname{char}}

\newcommand{\noncusp}{\operatorname{nsc}}
\newcommand{\nbr}{\textup{nbr}}
\newcommand{\rednbr}{\underline{\textup{nbr}}}

\newcommand{\Z}{\mathbb{Z}}

\def\eps{{\varepsilon}}
\def\phi{{\varphi}}

\newcommand{\zx}{{\textsf{x}}}

\newcommand{\zc}{c}

\newcommand{\zz}{z}
\newcommand{\ze}{e}

\newcommand{\za}{a}

\newcommand{\ga}{\gamma}
\newcommand{\Ga}{\Gamma}

\newcommand{\La}{\Lambda}
\newcommand{\al}{\alpha}
\newcommand{\be}{\beta}

\def\Si{\mathfrak{S}}

\newcommand{\de}{\delta}
\newcommand{\De}{\Delta}

\newcommand{\aff}{\textup{aff}}

\newcommand{\kk}{{\mathbbm{k}}}

\def\id{\mathop{\mathrm {id}}\nolimits}

\newcommand{\Ind}{{\mathrm {Ind}}}

\newcommand{\tr}{{\mathrm {tr}}}

\newcommand{\Res}{{\mathrm {Res}}}

\newcommand{\ZZ}{{\mathbb Z}}

\renewcommand{\mod}{\bmod \,}

\newcommand{\Zig}{\textsf{Z}}

\def\h{{\mathfrak h}}

\def\aff{{\operatorname{aff}}}

\def\Par{{\mathscr P}}
\def\ula{{\underline{\lambda}}}
\def\umu{{\underline{\mu}}}

\def\b{\mathfrak{b}}
\def\k{{\mathbbm{k}}}

\def\height{{\operatorname{ht}}}

\def\op{{\mathrm{op}}}
\def\re{{\mathrm{re}}}
\def\im{{\mathrm{im}\,}}

\def\into{{\hookrightarrow}}

\def\mod#1{#1\!\operatorname{-mod}}

\def\iso{\stackrel{\sim}{\longrightarrow}}

\def\HOM{\operatorname{Hom}}

\def\CH{{\operatorname{ch}_q\,}}
\def\DIM{{\operatorname{dim}_q\,}}

\def\words{I}

\def\Car{{\tt C}}
\def\cc{{\tt c}}

\newcommand{\bC}{\mathbf{C}}

{\catcode`\|=\active
  \gdef\set#1{\mathinner{\lbrace\,{\mathcode`\|"8000%
  \let|\midvert #1}\,\rbrace}}
}
\def\midvert{\egroup\mid\bgroup}

\colorlet{darkgreen}{green!50!black}
\tikzset{dots/.style={very thick,loosely dotted},
         greendot/.style={fill,circle,color=darkgreen,inner sep=1.5pt,outer sep=0},
         blackdot/.style={fill,circle,color=black,inner sep=1.5pt,outer sep=0},
         graydot/.style={fill,circle,color=gray,inner sep=1.1pt,outer sep=0}
}
\def\greendot(#1,#2){\node[greendot] at(#1,#2){}}
\def\blackdot(#1,#2){\node[blackdot] at(#1,#2){}}
\def\graydot(#1,#2){\node[graydot] at(#1,#2){}}

\newenvironment{braid}{
  \begin{tikzpicture}[baseline=6mm,black,line width=1pt, scale=0.32,
                      draw/.append style={rounded corners},
                      every node/.append style={font=\fontsize{5}{5}\selectfont}]%
  }{\end{tikzpicture}
}

\def\Grid(#1,#2){
  \draw[very thin,gray,step=2mm] (0,0)grid(#1,#2);
  \draw[very thin,darkgreen,step=10mm] (0,0)grid(#1,#2);
}

\newcommand\Tableau[2][\relax]{
  \begin{tikzpicture}[scale=0.5,draw/.append style={thick,black}]
    \ifx\relax#1\relax%
    \else 
      \foreach\box in {#1} { \filldraw[blue!30]\box+(-.5,-.5)rectangle++(.5,.5); }
    \fi
    \newcount\row\newcount\col
    \row=0
    \foreach \Row in {#2} {
       \col=1
       \foreach\k in \Row {
          \draw(\the\col,\the\row)+(-.5,-.5)rectangle++(.5,.5);
          \draw(\the\col,\the\row)node{\k};
          \global\advance\col by 1
       }
       \global\advance\row by -1
    }
  \end{tikzpicture}
}

\newcommand\YoungDiagram[2][\relax]{
  \begin{tikzpicture}[scale=0.5,draw/.append style={thick,black}]
    \ifx\relax#1\relax%
    \else 
    \foreach\box in {#1} {
      \filldraw[blue!30]\box rectangle ++(1,1);
    }
    \fi
    \newcount\row
    \row=0
    \foreach \col in {#2} {
       \draw(1,\the\row)grid ++(\col,1);
       \global\advance\row by -1
    }
  \end{tikzpicture}
}

\begin{document}

\title[Affine zigzag algebras]{{\bf Affine zigzag algebras and imaginary strata for KLR algebras}}

\author{\sc Alexander Kleshchev}
\address{Department of Mathematics\\ University of Oregon\\
Eugene\\ OR 97403, USA}
\email{klesh@uoregon.edu}

\author{\sc Robert Muth}
\address{Department of Mathematics\\ Tarleton State University
\\
Stephenville\\ TX 76402, USA}
\email{robmuth@gmail.com}


\thanks{
Supported by the NSF grant DMS-1161094, Max-Planck-Institut and Fulbright Foundation.}

\begin{abstract}
KLR algebras of affine ${\tt ADE}$ types are known to be properly stratified if the characteristic of the ground field is greater than some explicit bound. Understanding the strata of this stratification reduces to semicuspidal cases, which split into real and imaginary  subcases. Real semicuspidal strata are well-understood. We show that the smallest imaginary stratum is Morita equivalent to Huerfano-Khovanov's zigzag algebra tensored with a polynomial algebra in one variable. We introduce {\em affine zigzag algebras} and prove that these are Morita equivalent to arbitrary  imaginary strata if the characteristic of the ground field is greater than the bound mentioned above. 
\end{abstract}

\maketitle

\section{Introduction}

In this paper we work with the KLR algebras $R_\theta$ of Lie type $\Gamma$, which is assumed to be of untwisted affine ${\tt ADE}$ type, over an arbitrary field $\kk$ of characteristic $p\geq 0$. Here $\theta=\sum_{i\in I} n_i\al_i$ is an arbitrary element of the positive part $Q_+$ of the root lattice. 
McNamara \cite{McNAff} shows that these algebras are explicitly properly stratified if $p=0$. McNamara's result is generalized in \cite{KMStrat} to the case where $p>\min\{n_i\mid i\in I\}$. 

Informally, a proper stratification of $R_\theta$ yields a stratification of the category $\mod{R_\theta}$ of finitely generated graded $R_\theta$-modules by the categories $\mod{B_\xi}$ for much simpler algebras $B_\xi$, see \cite{Kdonkin} for details. 
Description of the algebras $B_\xi$ is easily reduced to the semicuspidal cases, which split into  real and imaginary subcases. In the real case we have $B_{n\al}\cong \kk[z_1,\dots,z_n]^{\Si_n}$, the algebra of symmetric polynomials in $n$ variables,  
but the imaginary case $B_{n\de}$ is not so easy to understand. 

The algebras $R_\theta$ actually have many proper stratifications. These are determined by a choice of a convex preorder on the set $\Phi_+$ of the positive roots of the corresponding affine root system. 
In this paper we always work with a balanced convex preorder as in \cite{KM}. We first prove that $B_\de\cong \kk[z]\otimes \Zig$, where $\Zig$ is the zigzag algebra of \cite{HK} corresponding to the underlying finite Dynkin diagram $\Gamma'$ obtained by deleting the affine node from $\Gamma$,  and $\kk[z]$ is the polynomial algebra. 
McNamara and Tingley \cite{MT} show that this description of $B_\de$ can be obtained for all convex preorders as an application of  their technique of face functors. 

In order to describe the higher imaginary strata, we introduce the main object of study of this paper---the rank $n$ {\em affine zigzag algebra} $\Zig_n^\aff$, which is defined for any connected graph without loops. We show that $B_{n\de}$ is (graded) Morita equivalent to the affine zigzag algebra $\Zig_n^\aff$ corresponding to $\Ga'$ if $p>\min\{n_i\mid i\in I\}$ (or $p=0$). 

To state the results more explicitly, we fix some notation. The  simple roots of our affine root system of untwisted ${\tt ADE}$ type are 
denoted $\al_i$ for $i\in I=\{0,1,\dots,l\}$. We assume that $0$ is  the affine vertex, so that $\al_1,\dots,\al_l$ are the simple roots of the underlying finite root system. Let $\de$ be the null-root. Let $n\in\Z_{>0}$. The {\em semicuspidal algebra} $C_{n\de}$ is a quotient of $R_{n\de}$ defined in such a way that the category of finitely generated semicuspidal $R_{n\de}$-modules is equivalent to the category $\mod{C_{n\de}}$ of finitely generated graded $C_{n\de}$-modules. 

We denote by $\Par_n$ the set of $l$-multipartitions of $n$. 
To every $\ula\in\Par_n$ one associates an irreducible $R_{n\de}$-module $L(\ula)$ and a standard $R_{n\de}$-module $\De(\ula)$, see \cite{KMStrat}. While $L(\ula)$ is finite dimensional, $\De(\ula)$ is always infinite dimensional. We have that $\{L(\ula)\mid \ula\in\Par_n\}$ is a complete irredundant system of irreducible $C_{n\de}$-modules up to isomorphism and degree shift, and $\De(\ula)$ is the projective cover of $L(\ula)$ in the category $\mod{C_{n\de}}$. 

We denote
$$
\De_{n\de}:=\bigoplus_{\ula\in\Par_n} \De(\ula)\quad\text{and}\quad B_{n\de}:=\End_{R_{n\de}}(\De_{n\de})^\op.
$$
Thus, $B_{n\de}$ is the basic algebra Morita equivalent to $C_{n\de}$. It turns out that the parabolically induced module $\De_\de^{\circ n}$, which can be considered as a $C_{n\de}$-module, is always projective in the category $\mod{C_{n\de}}$. However, it is a projective generator in $\mod{C_{n\de}}$ if and only if $p>n$ or $p=0$. So under these assumptions, the endomorphism algebra of $\De_\de^{\circ n}$ is Morita equivalent to $C_{n\de}$ and $B_{n\de}$. Otherwise, it is Morita equivalent to their idempotent truncations. The following result is proved under no restrictions on $p$. In fact, it holds over an arbitrary commutative unital ground ring $\kk$.

\vspace{2 mm}
\noindent
{\bf Theorem A.}
{\em
Assume that the convex preorder on $\Phi_+$ is balanced. Then we have an isomorphism of graded algebras 
$$
\End_{R_{n\de}}(\De_\de^{\circ n})^\op\cong\textup{\(\Zig\)}_n^\aff,
$$
where $\textup{\(\Zig\)}_n^\aff$ is the affine zigzag algebra of type $\Ga'$. In particular, 
$B_\de\cong \kk[z]\otimes \textup{\(\Zig\)}$.
}
\vspace{2 mm}

Theorem A appears in the body of the paper as Theorem \ref{mainthm} and Corollary \ref{maincor}. We note that Theorem A has been used in a crucial way in the recent proof of TurnerÕs conjecture on RoCK blocks of symmetric groups \cite{EK}, \cite{E}.

The affine zigzag algebra is actually a special case of a more general affinization construction which we present in \S\ref{SSAffSym}. For any graded symmetric algebra \(A\), free of finite rank over \(\kk\), we construct an associated {\em rank \(n\) affinization} \(\mathcal{H}_n(A)\) (see Definition \ref{AffDef}) and prove some fundamental results about this algebra.

\vspace{2 mm}
\noindent
{\bf Theorem B.} 
{\em 
Let \(A\) be a graded symmetric \(\kk\)-algebra, free of finite rank over \(\kk\). Let \(n \in \ZZ_{>0}\). Let \(\kk[z_1, \ldots, z_n]\) be a polynomial algebra in \(n\) generators, and \(\mathfrak{S}_n\) be the symmetric group of rank \(n\). Then
\begin{enumerate}
\item \(\mathcal{H}_n(A)\) is isomorphic to \(\kk[z_1, \ldots, z_n] \otimes A^{\otimes n} \otimes \kk\mathfrak{S}_n\) as a \(\kk\)-module.
\item \(\mathcal{H}_n(A)\) is free as a left/right \(\kk[z_1, \ldots, z_n]\)-module, free as a left/right \(A^{\otimes n}\)-module, and free as a left/right \(\kk\mathfrak{S}_n\)-module.
\item The center of \(\mathcal{H}_n(A)\) is 
\(
Z(\mathcal{H}_n(A)) = (\kk[z_1, \ldots, z_n] \otimes   Z(A)^{\otimes n})^{\mathfrak{S}_n},
\)
the subalgebra of invariants under the diagonal action of \(\mathfrak{S}_n\).
\item The wreath product \(A \wr \mathfrak{S}_n\) is a homomorphic image of \(\mathcal{H}_n(A)\).
\end{enumerate}
}
Parts (i)--(iv) of Theorem B appear in the body of the paper as Theorem \ref{AffBasis}, Corollary \ref{freeact}, and Proposition  \ref{kappa}. Our affinized symmetric algebras are related to the generalized degenerate affine Hecke algebras constructed by Costello and Grojnowski \cite{CG}, and the affine zigzag algebra is closely related to certain endomorphism algebras associated with the categorification of Heisenberg algebras by Cautis and Licata \cite{CL}, see Remarks \ref{afflit} and \ref{affziglit}.

\subsection*{Acknowledgements} We are grateful to Shunsuke Tsuchioka for alerting us to the connection between affine zigzag algebras and other algebras which have previously appeared in the mathematical literature.

\section{Preliminaries}\label{SStrat}
\subsection{Basic notation}
We will often work over a ground ring $\kk$ which is assumed to be a Noetherian commutative unital ring. When we assume that $\kk$ is a field, we write $p:=\cha \kk$. If $V$ is a free $\kk$-module with basis $\{v_1,\dots,v_n\}$ we denote by $\{v_1^*,\dots,v_n^*\}$ the dual basis of $V^*=\Hom_\kk(V,\kk)$. Our basic notation is as in \cite{KMStrat}, in particular, 
 all algebras, modules, ideals, etc., are assumed to be ($\Z$-)graded. The category of finitely generated graded left modules over a $\kk$-algebra $H$ we denote $\mod{H}$.

We will write \([1,t]:=\{1,2,\ldots, t\}\) for \(t \in \ZZ_{>0}\). The quantum integers $[n]=(q^n-q^{-n})/(q-q^{-1})$ as well as  expressions like $[n]!:=[1][2]\dots[n]$ and $1/(1-q^2)$ are always interpreted as Laurent series in $\Z((q))$. 
The morphisms in this category are all homogeneous degree zero  $H$-homomorphisms, which we denote $\hom_{H}(-,-)$. 
For  $V\in\mod{H}$, let $q^d V$ denote its grading shift by $d$,  so if $V_m$ is the degree $m$ component of $V$, then $(q^dV)_m= V_{m-d}.$ 
For $U,V\in \mod{H}$,
we set 
$\HOM_H(U, V):=\bigoplus_{d \in \Z} \HOM_H(U, V)_d,$ 
where
$
\HOM_H(U, V)_d := \hom_H(q^d U, V) .
$
If all graded components $V_m$ of a $\kk$-module $V$ are free of finite rank, we denote by $\DIM V:=\sum_{m\in\Z}(\operatorname{rk} V_m)q^m\in \Z((q))$ the graded rank of $V$. 

If $\mu$ is a usual partition of $n$, we write $n=|\mu|$. 
An {\em $l$-multipartition} of $n$ is a tuple $\umu=(\mu^{(1)},\dots,\mu^{(l)})$ of partitions such that $|\umu|:=|\mu^{(1)}|+\dots+|\mu^{(l)}|=n$. The set of the all $l$-multipartitions of $n$ is denoted by $\Par_n$, and $\Par:=\sqcup_{n\geq 0}\Par_n$. 

\subsection{Symmetric group actions}\label{SnAct}
Let \(\mathfrak{S}_n\) be the symmetric group of rank \(n\), generated by the simple transpositions \(s_1, \ldots, s_{n-1}\). For a \(\kk\)-module \(V\), we define a left action of \(\mathfrak{S}_n\) on \(V^{\otimes n}\) via place permutation:
\begin{align*}
{}^\sigma(v_1 \otimes \cdots \otimes v_n)  := v_{\sigma^{-1} 1} \otimes \cdots \otimes v_{\sigma^{-1} n},
\end{align*}
for all \(\bv = v_1 \otimes \cdots \otimes v_n \in V^{\otimes n}\) and \(\sigma \in \mathfrak{S}_n\). 

We define a left action of \(\mathfrak{S}_n\) on the polynomial algebra \(\kk[z_1, \ldots, z_n]\), via permutation of generators:
\begin{align*}
{}^\sigma\hspace{-0.3mm} z_i := z_{\sigma i}
\end{align*}
for all \(i \in [1,n]\) and \(\sigma \in \mathfrak{S}_n\), and extend this action to all \(f=f(z_1, \ldots, z_n) \in \kk[z_1, \ldots, z_n]\). 

For \(i \in [1,n-1]\), define the {\em divided difference operator} \(\nabla_i\) on \(\kk[z_1, \ldots, z_n]\) by
\begin{align*}
\nabla_i(f):=\frac{f-{}^{s_i}\hspace{-0.5mm}f}{z_i - z_{i+1}}. 
\end{align*}
The following facts about divided differences are well-known and easily checked:
\begin{Lemma}\label{DivDiff} Let \(i \in [1,n-1]\), \(j \in [1,n]\), and \(f \in\kk[z_1, \ldots, z_n]\). Then:
\begin{enumerate}
\item \(\nabla_i(f) = {}^{s_i}\hspace{-0.5mm}(\nabla_i(f)) = -\nabla_i({}^{s_i}\hspace{-0.5mm}f)\)
\item \(\nabla_i(f)=0\) if \({}^{s_i}\hspace{-0.5mm}f = f\)
\item \(\nabla_i(z_jf)-z_{s_ij}\nabla_i(f)=(\delta_{i,j}-\delta_{i+1,j})f\).
\end{enumerate}
\end{Lemma}

\subsection{Affine root system}\label{SSARS}
Let $\Car=(\cc_{ij})_{i,j\in I}$ be a {\em Cartan matrix} of  untwisted affine ${\tt ADE}$ type, see \cite[\S 4, Table Aff 1]{Kac}. So \({\tt C}\) corresponds to one of the following Dynkin diagrams:

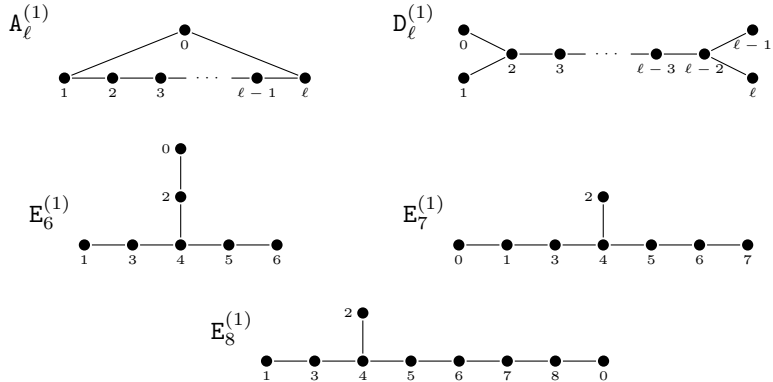
\begin{align*}
{\tt A}_\ell^{(1)}
\hspace{0mm}
\begin{braid}\tikzset{baseline=3mm}
\coordinate (0) at (5,1);
\coordinate (1) at (0,-1);
\coordinate (2) at (2,-1);
\coordinate (3) at (4,-1);
\coordinate (4) at (6,-1);
\coordinate (5) at (8,-1);
\coordinate (6) at (10,-1);
\draw [thin, black,shorten <= 0.1cm, shorten >= 0.1cm]   (0) to (1);
\draw [thin, black,shorten <= 0.1cm, shorten >= 0.1cm]   (1) to (2);
\draw [thin, black,shorten <= 0.1cm, shorten >= 0.1cm]   (2) to (3);
\draw [thin, black,shorten <= 0.1cm, shorten >= 0.3cm]   (3) to (4);
\draw [thin, black,shorten <= 0.3cm, shorten >= 0.1cm]   (4) to (5);
\draw [thin, black,shorten <= 0.1cm, shorten >= 0.1cm]   (5) to (6);
\draw [thin, black,shorten <= 0.1cm, shorten >= 0.1cm]   (0) to (6);
\draw (0) node[below]{$0$};
\draw (1) node[below]{$1$};
\draw (2) node[below]{$2$};
\draw (3) node[below]{$3$};
\draw (5) node[below]{$\ell-1$};
\draw (6) node[below]{$\ell$};
\blackdot(5,1);
\blackdot(0,-1);
\blackdot(2,-1);
\blackdot(4,-1);
\blackdot(8,-1);
\blackdot(10,-1);
\draw(6,-1) node{$\cdots$};
\end{braid}
\hspace{1cm}
{\tt D}_\ell^{(1)}
\hspace{2mm}
\begin{braid}\tikzset{baseline=3mm}
\coordinate (0) at (0,1);
\coordinate (1) at (0,-1);
\coordinate (2) at (2,0);
\coordinate (3) at (4,0);
\coordinate (4) at (6,0);
\coordinate (5) at (8,0);
\coordinate (6) at (10,0);
\coordinate (7) at (12,1);
\coordinate (8) at (12,-1);
\draw [thin, black,shorten <= 0.1cm, shorten >= 0.1cm]   (0) to (2);
\draw [thin, black,shorten <= 0.1cm, shorten >= 0.1cm]   (1) to (2);
\draw [thin, black,shorten <= 0.1cm, shorten >= 0.1cm]   (2) to (3);
\draw [thin, black,shorten <= 0.1cm, shorten >= 0.3cm]   (3) to (4);
\draw [thin, black,shorten <= 0.3cm, shorten >= 0.1cm]   (4) to (5);
\draw [thin, black,shorten <= 0.1cm, shorten >= 0.1cm]   (5) to (6);
\draw [thin, black,shorten <= 0.1cm, shorten >= 0.1cm]   (6) to (7);
\draw [thin, black,shorten <= 0.1cm, shorten >= 0.1cm]   (6) to (8);
\draw (0) node[below]{$0$};
\draw (1) node[below]{$1$};
\draw (2) node[below]{$2$};
\draw (3) node[below]{$3$};
\draw (5) node[below]{$\ell-3$};
\draw (6) node[below]{$\ell-2$};
\draw (7) node[below]{$\ell-1$};
\draw (8) node[below]{$\ell$};
\blackdot(0,1);
\blackdot(0,-1);
\blackdot(2,0);
\blackdot(4,0);
\blackdot(8,0);
\blackdot(10,0);
\blackdot(12,1);
\blackdot(12,-1);
\draw(6,0) node{$\cdots$};
\end{braid}
\end{align*}

\begin{align*}
{\tt E}_6^{(1)}
\hspace{0mm}
\begin{braid}\tikzset{baseline=3mm}
\coordinate (0) at (0,4);
\coordinate (2) at (0,2);
\coordinate (4) at (0,0);
\coordinate (3) at (-2,0);
\coordinate (1) at (-4,0);
\coordinate (5) at (2,0);
\coordinate (6) at (4,0);
\draw [thin, black,shorten <= 0.1cm, shorten >= 0.1cm]   (0) to (2);
\draw [thin, black,shorten <= 0.1cm, shorten >= 0.1cm]   (1) to (3);
\draw [thin, black,shorten <= 0.1cm, shorten >= 0.1cm]   (2) to (4);
\draw [thin, black,shorten <= 0.1cm, shorten >= 0.1cm]   (3) to (4);
\draw [thin, black,shorten <= 0.1cm, shorten >= 0.1cm]   (4) to (5);
\draw [thin, black,shorten <= 0.1cm, shorten >= 0.1cm]   (5) to (6);
\draw (0) node[left]{$0$};
\draw (1) node[below]{$1$};
\draw (2) node[left]{$2$};
\draw (3) node[below]{$3$};
\draw (4) node[below]{$4$};
\draw (5) node[below]{$5$};
\draw (6) node[below]{$6$};
\blackdot(0,4);
\blackdot(0,2);
\blackdot(0,0);
\blackdot(2,0);
\blackdot(4,0);
\blackdot(-2,0);
\blackdot(-4,0);
\end{braid}
\hspace{1.5cm}
{\tt E}_7^{(1)}
\hspace{0mm}
\begin{braid}\tikzset{baseline=3mm}
\coordinate (0) at (-6,0);
\coordinate (2) at (0,2);
\coordinate (4) at (0,0);
\coordinate (3) at (-2,0);
\coordinate (1) at (-4,0);
\coordinate (5) at (2,0);
\coordinate (6) at (4,0);
\coordinate (7) at (6,0);
\draw [thin, black,shorten <= 0.1cm, shorten >= 0.1cm]   (0) to (1);
\draw [thin, black,shorten <= 0.1cm, shorten >= 0.1cm]   (1) to (3);
\draw [thin, black,shorten <= 0.1cm, shorten >= 0.1cm]   (2) to (4);
\draw [thin, black,shorten <= 0.1cm, shorten >= 0.1cm]   (3) to (4);
\draw [thin, black,shorten <= 0.1cm, shorten >= 0.1cm]   (4) to (5);
\draw [thin, black,shorten <= 0.1cm, shorten >= 0.1cm]   (5) to (6);
\draw [thin, black,shorten <= 0.1cm, shorten >= 0.1cm]   (6) to (7);
\draw (0) node[below]{$0$};
\draw (1) node[below]{$1$};
\draw (2) node[left]{$2$};
\draw (3) node[below]{$3$};
\draw (4) node[below]{$4$};
\draw (5) node[below]{$5$};
\draw (6) node[below]{$6$};
\draw (7) node[below]{$7$};
\blackdot(-6,0);
\blackdot(0,2);
\blackdot(0,0);
\blackdot(2,0);
\blackdot(4,0);
\blackdot(6,0);
\blackdot(-2,0);
\blackdot(-4,0);
\end{braid}
\end{align*}

\begin{align*}
\hspace{5mm}
{\tt E}_8^{(1)}
\hspace{0mm}
\begin{braid}\tikzset{baseline=3mm}
\coordinate (0) at (10,0);
\coordinate (2) at (0,2);
\coordinate (4) at (0,0);
\coordinate (3) at (-2,0);
\coordinate (1) at (-4,0);
\coordinate (5) at (2,0);
\coordinate (6) at (4,0);
\coordinate (7) at (6,0);
\coordinate (8) at (8,0);
\draw [thin, black,shorten <= 0.1cm, shorten >= 0.1cm]   (0) to (8);
\draw [thin, black,shorten <= 0.1cm, shorten >= 0.1cm]   (1) to (3);
\draw [thin, black,shorten <= 0.1cm, shorten >= 0.1cm]   (2) to (4);
\draw [thin, black,shorten <= 0.1cm, shorten >= 0.1cm]   (3) to (4);
\draw [thin, black,shorten <= 0.1cm, shorten >= 0.1cm]   (4) to (5);
\draw [thin, black,shorten <= 0.1cm, shorten >= 0.1cm]   (5) to (6);
\draw [thin, black,shorten <= 0.1cm, shorten >= 0.1cm]   (6) to (7);
\draw [thin, black,shorten <= 0.1cm, shorten >= 0.1cm]   (7) to (8);
\draw (0) node[below]{$0$};
\draw (1) node[below]{$1$};
\draw (2) node[left]{$2$};
\draw (3) node[below]{$3$};
\draw (4) node[below]{$4$};
\draw (5) node[below]{$5$};
\draw (6) node[below]{$6$};
\draw (7) node[below]{$7$};
\draw (8) node[below]{$8$};
\blackdot(10,0);
\blackdot(0,2);
\blackdot(0,0);
\blackdot(2,0);
\blackdot(4,0);
\blackdot(6,0);
\blackdot(8,0);
\blackdot(-2,0);
\blackdot(-4,0);
\end{braid}
\end{align*}
We have $I=\{0,1,\dots,l\},$ 
where $0$ is the affine vertex, and set
$
I':=\{1,\dots,\l\}=I\setminus\{0\}. 
$ 
Let $\Car'$ be the {\em finite type} Cartan matrix corresponding to the subset $I'\subset I$. 
 
Let $(\h,\Pi,\Pi^\vee)$ be a realization of $\Car$, with simple roots $\{\al_i\mid i\in I\}$ 
standard bilinear form $(\cdot,\cdot)$ on $\h^*$, and $Q_+ := \bigoplus_{i \in I} \Z_{\geq 0} \cdot \al_i$. For $\theta \in Q_+$, we write $\height(\theta)$ for the sum of its 
coefficients when expanded in terms of the $\al_i$'s. 
Let $\Phi$ and $\Phi'$ be the root systems corresponding to \(\Car\) and \(\Car'\) respectively, with $\Phi_+$ and \(\Phi_+'\) being  the corresponding sets of {\em positive} roots. Let 
$\de\in \Phi_+$ 
be the {\em null root}. 
We have 
$\Phi_+=\Phi_+^\im\sqcup \Phi_+^\re
$, where
$\Phi_+^\im=\{n\de\mid n\in\Z_{>0}\}$
and 
\begin{align*}
\Phi_+^\re=\{\be+n\de\mid \be\in  \Phi'_+,\ n\in\Z_{\geq 0}\}\sqcup \{-\be+n\de\mid \be\in  \Phi'_+,\ n\in\Z_{> 0}\}.
\end{align*}

A {\em convex preorder} on $\Phi_+$ is a total preorder $\preceq$ such that for all $\be,\ga\in\Phi_+$ we have:
\begin{enumerate}
\item[$\bullet$]
If $\be\preceq \ga$ and $\be+\ga\in\Phi_+$, then $\be\preceq\be+\ga\preceq\ga$;
\item[$\bullet$]
$\be\preceq\ga$ and $\ga\preceq\be$ if and only if $\be$ and $\ga$ are imaginary or \(\be = \ga\).
\end{enumerate}
A convex preorder is called {\em balanced} if all finite simple roots $\al_i$ with $i\in I'$ satisfy $\al_i\succeq \de$.

\subsection{KLR algebras}\label{SKLR} 
Define the polynomials $\{Q_{ij}(u,v)\in \kk[u,v]\mid i,j\in I\}$  
as follows. If $\Car\neq {\tt A}_1^{(1)}$, 
choose signs $\eps_{ij}$ for all $i,j \in I$ with $\cc_{ij}
< 0$  so that $\eps_{ij}\eps_{ji} = -1$ and set 
\begin{equation*}\label{EArun}
Q_{ij}(u,v):=
\left\{
\begin{array}{ll}
0 &\hbox{if $i=j$;}\\
1 &\hbox{if $\cc_{ij}=0$;}\\
\eps_{ij}(u^{-\cc_{ij}}-v^{-\cc_{ji}}) &\hbox{if $\cc_{ij}<0$.}
\end{array}
\right.
\end{equation*}
For type ${\tt A}_1^{(1)}$ we set
\begin{equation*}\label{EArun1}
Q_{ij}(u,v):=
\left\{
\begin{array}{ll}
0 &\hbox{if $i=j$;}\\
(u-v)(v-u) &\hbox{if $i\neq j$.}
\end{array}
\right.
\end{equation*}
We point out that we have just made a so-called generic or geometric choice of parameters for KLR algebras. The main results of the paper {\em do not} hold for non-generic choices of parameters, and the imaginary semicuspidal algebra is not isomorphic to the affine zigzag algebra in the non-generic setting.

Fix 
$\theta\in Q_+$ of height $n$. Let
$ 
I^\theta=\{\bi=(i_1, \dots, i_n)\in I^n\mid \al_{i_1}+\dots+\al_{i_n}=\theta\}. 
$ 
For $\bi\in I^\theta$ and $\bj\in I^\eta$, we denote by $\bi\bj\in I^{\theta+\eta}$ the concatenation of $\bi$ and $\bj$. 
The symmetric group $\Si_n$ acts on  $I^\theta$ by place permutations.

The {\em KLR-algebra} $R_\theta$ is an associative graded unital $\kk$-algebra, given by the generators
$
\{1_{\bi}\mid \bi\in I^\theta\}\cup\{y_1,\dots,y_{n}\}\cup\{\psi_1, \dots,\psi_{n-1}\}
$ 
and the following relations for all $\bi,\bj\in I^\theta$ and all admissible $r,t$:
\begin{align}
1_{\bi}  1_{\bj} = \de_{\bi,\bj} 1_{\bi} ,
\quad{\textstyle\sum_{\bi \in I^\theta}} 1_{\bi}  = 1;\label{KLRidem}
\end{align}
\begin{align}
y_r 1_{\bi}  = 1_{\bi}  y_r;\quad y_r y_t = y_t y_r;\label{KLRy}
\end{align}
\begin{align}
\psi_r 1_{\bi}  = 1_{s_r\bi} \psi_r;\label{KLRpsiidem}
\end{align}
\begin{align}
(y_t\psi_r-\psi_r y_{s_r(t)})1_{\bi}  
= \de_{i_r,i_{r+1}}(\de_{t,r+1}-\de_{t,r})1_{\bi};
\label{KLRypsi}
\end{align}
\begin{align}
\psi_r^21_{\bi}  = Q_{i_r,i_{r+1}}(y_r,y_{r+1})1_{\bi}; 
\label{KLRpsi2}
\end{align}
\begin{align} 
\psi_r \psi_t = \psi_t \psi_r\qquad (|r-t|>1);\label{KLRpsi}
\end{align}
\begin{align}
(\psi_{r+1}\psi_{r} \psi_{r+1}-\psi_{r} \psi_{r+1} \psi_{r}) 1_{\bi}  
=
\de_{i_r,i_{r+2}}\frac{Q_{i_r,i_{r+1}}(y_{r+2},y_{r+1})-Q_{i_r,i_{r+1}}(y_r,y_{r+1})}{y_{r+2}-y_r}1_{\bi}.
\label{KLRbraid}
\end{align}
The {\em grading} on $R_\theta$ is defined by setting 
$
\deg(1_{\bi} )=0$, $\deg(y_r1_{\bi} )=2$, and $\deg(\psi_r 1_{\bi} )=-{\tt c}_{i_r,i_{r+1}}.
$

For any $V\in\mod{R_\theta}$, its {\em formal character} is $\CH V:=\sum_{\bi\in \words^\theta}(\DIM 1_{\bi} V)\cdot\bi\in\bigoplus_{\bi\in\words^\theta}\Z((q))\cdot \bi$. 
We refer to $1_{\bi} V$ as the {\em $\bi$-word space} of $V$ and to its vectors as {\em vectors of word}~$\bi$.

For $\theta_1,\dots,\theta_m\in Q^+$ and $\theta=\theta_1+\dots+\theta_m$, we have a parabolic subalgebra $R_{\theta_1,\dots,\theta_m}\subseteq R_\theta,$ and the corresponding (exact)
induction functor 
$$\Ind_{\theta_1,\dots,\theta_m}:=R_\theta 1_{\theta_1, \ldots, \theta_m} \otimes_{R_{\theta_1, \ldots, \theta_m}} -:
\mod{R_{\theta_1,\dots,\theta_m}}\to\mod{R_{\theta}}.
$$ 
For $V_1\in\mod{R_{\theta_1}}, \dots, V_m\in\mod{R_{\theta_m}}$, we denote 
$$V_1\circ\dots\circ V_m:=\Ind_{\theta_1,\dots,\theta_m} V_1\boxtimes \dots\boxtimes V_m.$$ 
Given also $W_r\in\mod{R_{\theta_r}}$ and $f_r\in\Hom_{R_{\theta_r}}(V_r,W_r)$ for $r=1,\dots,m$, we denote 
$$
f_1\circ \dots \circ f_m:=\Ind_{\theta_1,\dots,\theta_m}(f_1\otimes \dots\otimes f_m): V_1\circ\dots\circ V_m\to W_1\circ\dots\circ W_m.
$$

We also have the restriction functors: 
$$
\Res_{\theta_1,\dots,\theta_m}:= 1_{\theta_1, \ldots, \theta_m} R_{\theta}
\otimes_{R_{\theta}} -:\mod{R_{\theta}}\rightarrow \mod{R_{\theta_1, \ldots, \theta_m}}.
$$

\subsection{Diagrammatics for KLR algebras} \label{KLRdiag}
It is often useful in computations to work with the diagrammatic presentation of the KLR algebra as provided in \cite{KL1}; see that paper for a fuller explanation of the diagrammatic presentation. We will make extensive use of KLR diagrammatics in \S\ref{SStratum} and \S\ref{diagproofs}. The diagrammatic treatment for types \({\tt C} \neq {\tt A}^{(1)}_1\) is given below; in this paper we will always treat the idiosyncratic type \({\tt A}_1^{(1)}\) calculations symbolically, so we do not provide those diagrammatics here.

We depict the (idempotented) generators of \(R_\theta\) as the following diagrams:
\begin{align*}
1_{\bi} = 
\begin{braid}\tikzset{baseline=0mm}
\draw(1,1) node[above]{$i_1$}--(1,-1);
\draw(2,1) node[above]{$i_2$}--(2,-1);
\draw(3,1) node[above]{$\cdots$};
\draw(4,1) node[above]{$i_n$}--(4,-1);
\end{braid}
\hspace{15mm}
y_{r}1_{\bi} = 
\begin{braid}\tikzset{baseline=0mm}
\draw(1,1) node[above]{$i_1$}--(1,-1);
\draw(2,1) node[above]{$\cdots$};
\draw(3,1) node[above]{$i_r$}--(3,-1);
\draw(4,1) node[above]{$\cdots$};
\draw(5,1) node[above]{$i_n$}--(5,-1);
\blackdot(3,0);
\end{braid}
\hspace{15mm}
\psi_r1_{\bi} = 
\begin{braid}\tikzset{baseline=0mm}
\draw(1,1) node[above]{$i_1$}--(1,-1);
\draw(2,1) node[above]{$\cdots$};
\draw(3,1) node[above]{$i_r$}--(4.5,-1);
\draw(4.5,1) node[above]{$i_{r+1}$}--(3,-1);
\draw(5.5,1) node[above]{$\cdots$};
\draw(6.5,1) node[above]{$i_n$}--(6.5,-1);
\end{braid}
\end{align*}
Note that `right-to-left' in the symbolic presentation is to be read as `top-to-bottom' in the diagrammatic presentation. Then \(R_\theta\) is spanned by planar diagrams that look locally like these generators, equivalent up to the usual isotopies (described in \cite{KL1}). In particular, dots can be freely isotoped along strands, provided they don't pass through crossings. Multiplication of diagrams is given by stacking vertically, and products are zero unless labels for strands match. The defining local relations for \(R_\theta\) are drawn as follows:
\begin{align*}
\hspace{-0.5cm}
\begin{braid}\tikzset{scale=0.8, baseline=0mm}
\draw(1,1) node[above]{$i$}--(2,0)--(1,-1);
\draw(2,1) node[above]{$j$}--(1,0)--(2,-1);
\end{braid}
=
\begin{cases}
\varepsilon_{ij}\left(\begin{braid}\tikzset{scale=0.8,baseline=0mm}
\draw(1,1) node[above]{$i$}--(1,-1);
\draw(2,1) node[above]{$j$}--(2,-1);
\blackdot(1,0);
\end{braid}
\hspace{-0.8mm}
-
\hspace{-0.8mm} 
\begin{braid}\tikzset{scale=0.8,baseline=0mm}
\draw(1,1) node[above]{$i$}--(1,-1);
\draw(2,1) node[above]{$j$}--(2,-1);
\blackdot(2,0);
\end{braid}
\right)
&
{\tt c}_{i,j}=-1;\\
\\
0
&
i=j;\\
\\
\begin{braid}\tikzset{scale=0.8,baseline=0mm}
\draw(1,1) node[above]{$i$}--(1,-1);
\draw(2,1) node[above]{$j$}--(2,-1);
\end{braid}
&
\textup{otherwise},
\end{cases}
\hspace{1cm}
\begin{braid}\tikzset{scale=0.8,baseline=0mm}
\draw(0,1) node[above]{$i$}--(2,-1);
\draw(1,1) node[above]{$j$}--(2,0)--(1,-1);
\draw(2,1) node[above]{$k$}--(0,-1);
\end{braid}
\hspace{-1.2mm} 
-
\hspace{-1.2mm} 
\begin{braid}\tikzset{scale=0.8,baseline=0mm}
\draw(0,1) node[above]{$i$}--(2,-1);
\draw(1,1) node[above]{$j$}--(0,0)--(1,-1);
\draw(2,1) node[above]{$k$}--(0,-1);
\end{braid}
=
\begin{cases}
\varepsilon_{ij}\begin{braid}\tikzset{scale=0.8,baseline=0mm}
\draw(0,1) node[above]{$i$}--(0,-1);
\draw(1,1) node[above]{$j$}--(1,-1);
\draw(2,1) node[above]{$i$}--(2,-1);
\end{braid}&
i= k, {\tt c}_{i,j}=-1;\\
\\
0
&
\textup{otherwise},
\end{cases}
\end{align*}
\begin{align*}
\begin{braid}\tikzset{baseline=0mm}
\draw(1,1) node[above]{$i$}--(2,-1);
\draw(2,1) node[above]{$j$}--(1,-1);
\blackdot(1.75,-0.5);
\end{braid}
-
\begin{braid}\tikzset{baseline=0mm}
\draw(1,1) node[above]{$i$}--(2,-1);
\draw(2,1) node[above]{$j$}--(1,-1);
\blackdot(1.25,0.5);
\end{braid}
=
\delta_{i,j}
\begin{braid}\tikzset{baseline=0mm}
\draw(1,1) node[above]{$i$}--(1,-1);
\draw(2,1) node[above]{$i$}--(2,-1);
\end{braid}
=
\begin{braid}\tikzset{baseline=0mm}
\draw(1,1) node[above]{$i$}--(2,-1);
\draw(2,1) node[above]{$j$}--(1,-1);
\blackdot(1.75,0.5);
\end{braid}
-
\begin{braid}\tikzset{baseline=0mm}
\draw(1,1) node[above]{$i$}--(2,-1);
\draw(2,1) node[above]{$j$}--(1,-1);
\blackdot(1.25,-0.5);
\end{braid}.
\end{align*}

\subsection{Semicuspidal modules}\label{SSSemiCusp}

We fix a convex preorder $\preceq$ on $\Phi_+$ and $n\in\Z_{>0}$. In this paper we will only deal with imaginary semicuspidal modules.  An $R_{n\de}$-module $V$ is called {\em (imaginary)  semicuspidal}\, if $\theta,\eta\in Q_+$, $\theta+\eta=n\de$, and $\Res_{\theta,\eta}V\neq 0$ imply that $\theta$ is a sum of positive roots  $\preceq\de$ and $\eta$ is a sum of positive roots $\succeq\de$. 

Words $\bi\in I^{n\de}$ which appear in some semicuspidal $R_{n\de}$-module are called {\em semicuspidal words}. We denote by $I^{n\de}_{\noncusp}$ the set of non-semicuspidal words, and let
$
1_{\noncusp}
:=\sum_{\bi\in I^{n\de}_{\noncusp}}1_{\bi}.
$
Following \cite{McNAff}, define the {\em semicuspidal algebra} 
\begin{equation}\label{ESCA}
C_{n\de}=C_{n\de,\kk}:=R_{n\al}/R_{n\al} 1_{\noncusp}R_{n\al}.
\end{equation}
Then the category of finitely generated semicuspidal $R_{n\al}$-modules is equivalent to the category $\mod{C_{n\al}}$.

From now on until the end of this subsection we assume that $\kk$ is a field. The irreducible $C_{n\de}$-modules are parametrized canonically by the \(l\)-multipartitions $\ula\in\Par_n$, see \cite{KM,McNAff,TW,KMStrat}. The irreducible corresponding to $\ula$ is denoted by $L(\ula)$, and its projective cover in $\mod{C_{n\de}}$ is denoted $\De(\ula)$.

For the case $n=1$, we use a special notation. To every $i\in I'$ we associate the multipartition $\mu(i)\in\Par_1$ with the only non-trivial partition in the $i$th component. 
This gives a bijection $I'\to \Par_1$. We denote 
$$
L_{\de,i}:=L(\mu(i)),\quad \De_{\de,i}:=\De(\mu(i)) \qquad(i\in I').
$$
Then $\De_\de:=\bigoplus_{i\in I'}\De_{\de,i}$ is a projective generator in $\mod{C_\de}$. 
In \S\ref{SStratum} we give more information on these modules and construct their forms over $\kk$ which is not necessarily a field.

\section{Affinizations of symmetric algebras}\label{SAff}

\subsection{Symmetric algebras}\label{symsec}
Let \(\kk\) be a commutative Noetherian ring, and let \(A\) be a \(\ZZ\)-graded, unital, associative \(\kk\)-algebra, free of finite rank over \(\kk\). We consider \(A \otimes A\) as an \((A,A)\)-bimodule via the action \(a_1  \cdot (b_1 \otimes b_2)  \cdot a_2= a_1b_1 \otimes b_2a_2\). Note then that the multiplication map \(m: A \otimes A \to A\) is a homogeneous degree zero \((A,A)\)-bimodule homomorphism. We consider \(A^*=\bigoplus_{t\in \ZZ} (A_t)^*\) as a graded \((A,A)\)-bimodule via the action \((a_1 \cdot f \cdot a_2)(b) = f(a_2ba_1)\), where the grading is given by considering elements of \((A_t)^*\) to have degree \(-t\).

We say that \(A\) is {\em graded symmetric} if it is equipped with an \((A,A)\)-bimodule isomomorphism \(\varphi: A \xrightarrow{\sim} A^*\) which is homogeneous of degree \(-d\), for some \(d \in \ZZ\). For the rest of this section we assume that \(A\) is graded symmetric; the trivial grading \(A=A_0\) is of course permitted. 

We may then define an \((A,A)\)-bimodule homomorphism 
\begin{align*}
\Delta:= (\varphi^{-1} \otimes \varphi^{-1})\circ m^*\circ\varphi : A \to A \otimes A,
\end{align*}
which is homogeneous of degree \(d\). We call \(\Delta(1)\) the {\em distinguished element} of \(A \otimes A\). The distinguished element is homogeneous of degree \(d\), and is symmetric and intertwines elements of \(A \otimes A\) in the following sense:

\begin{Lemma}\label{tauDelta} For a \(\kk\)-module \(V\), let \(\tau_{V,V}: V \otimes V \to V \otimes V\) be the transposition map given by \(\tau_{V,V}(v \otimes w) = w \otimes v\). 
\begin{enumerate}
\item \(\tau_{A,A}(\Delta(1)) = \Delta(1)\), and
\item \(\ba \Delta(1) = \Delta(1) \tau_{A,A}(\ba)\), for all \(\ba \in A \otimes A\).
\end{enumerate}
\end{Lemma}

\begin{proof} 
For \(x,y \in A\) we have
\begin{align*}
(m^* \circ \varphi(1))(x \otimes y) &= \varphi(1)(xy) = (y \cdot \varphi(1))(x) = \varphi(y)(x) = (\varphi(1) \cdot y)(x)\\
&= \varphi(1)(yx) = (m^* \circ \varphi(1))(y \otimes x) = (m^* \circ \varphi(1))(\tau_{A,A}(x \otimes y))\\
&= (\tau_{A,A}^* \circ m^*  \circ \varphi(1))(x \otimes y) = (\tau_{A^*, A^*} \circ m^* \circ \varphi(1))(x \otimes y),
\end{align*}
Thus \(m^* \circ \varphi(1) = \tau_{A^*, A^*} \circ m^* \circ  \varphi(1)\), and, since \(\tau_{A,A} \circ (\varphi^{-1} \otimes \varphi^{-1}) = (\varphi^{-1} \otimes \varphi^{-1}) \circ \tau_{A^*, A^*}\), result (i) follows.

Now assume \(\Delta(1) = \sum_{i} x^{(i)}_1 \otimes x^{(i)}_2\), and let \(a \in A\). Then, using (i) and the fact that \(\Delta: A \to A \otimes A\) is a bimodule homomorphism, we have
\begin{align*}
(a \otimes 1)\Delta(1) = a \cdot \Delta(1) = \Delta(a) = \Delta(1)\cdot a = \Delta(1)(1 \otimes a),
\end{align*}
and
\begin{align*}
(1 \otimes a) \Delta(1) &= \sum_{i} x^{(i)}_1 \otimes ax^{(i)}_2
= \tau_{A,A}\left(\sum_{i} ax^{(i)}_2 \otimes x^{(i)}_1\right)
= \tau_{A,A}\left( a \cdot \tau_{A,A}(\Delta(1))\right)\\
&=\tau_{A,A}\left( a \cdot \Delta(1)\right)
=\tau_{A,A}\left( \Delta(a) \right)
=\tau_{A,A}\left(\Delta(1) \cdot a \right)\\
&=\tau_{A,A}\left(\sum_{i} x^{(i)}_1 \otimes x^{(i)}_2a\right)
= \sum_{i} x^{(i)}_2a \otimes x^{(i)}_1
=\left( \sum_{i} x^{(i)}_2 \otimes x^{(i)}_1\right)(a \otimes 1)\\
&= \tau_{A,A}(\Delta(1))(a \otimes 1)
= \Delta(1)(a \otimes 1),
\end{align*}
completing the proof of (ii).
\end{proof}

\subsection{Affinization}\label{SSAffSym}
Let \(n \in \ZZ_{> 0}\). The grading on \(A\) induces a grading on the algebra \(A^{\otimes n}\). For \(1\leq t<u \leq n\), let \(\iota_{t,u}: A^{\otimes 2} \to A^{\otimes n}\) be the algebra homomorphism given by
\begin{align*}
\iota_{t,u}(a_1 \otimes a_2) &= 1 \otimes \cdots \otimes 1 \otimes a_1 \otimes 1 \otimes \cdots \otimes 1 \otimes a_2 \otimes 1 \otimes \cdots \otimes 1,
\end{align*}
where \(a_1\) appears in the \(t\)th slot, and \(a_2\) appears in the \(u\)th slot. Then we define \(\Delta_{t,u}:=\iota_{t,u}\circ\Delta(1) \in A^{\otimes n}\).

Let \(\kk[z_1, \ldots, z_n]\) be the graded polynomial algebra with generators \(z_1, \ldots, z_n\) in degree \(d = \deg(\Delta(1))\). Let \(\kk\mathfrak{S}_n\) be the symmetric group algebra over \(\kk\), concentrated in degree zero.

\begin{Definition}\label{AffDef} 
We define \(\mathcal{H}_n(A)\), the {\em rank \(n\) affinization of \(A\)}, to be the free product of \(\kk\)-algebras 
\begin{align*}
\kk[z_1, \ldots, z_n] \star A^{\otimes n} \star \kk\mathfrak{S}_n,
\end{align*}
subject to the following commutation relations:
\begin{align}
 \ba z_j &= z_j \ba &\textup{ for all for all \(j \in [1,n]\),  \(\ba \in A^{\otimes n}\)}; \label{AZ}\\
 s_i \ba &= {}^{s_i}\hspace{-0.3mm}\ba s_i & \textup{for all \(i \in [1,n-1]\), \(\ba \in A^{\otimes n}\)};\label{AS}\\
s_iz_j - z_{s_ij}s_i &= (\delta_{i,j} - \delta_{i+1,j})\Delta_{i,i+1} &\textup{for all \(i \in [1,n-1]\), \(j \in [1,n]\)}.\label{SZ}
\end{align}
Note that the relations are homogeneous, so \(\mathcal{H}_n(A)\) inherits a graded structure from the algebras \(A^{\otimes n}\), \(\kk[z_1, \ldots, z_n]\) and \(\kk\mathfrak{S}_n\).
\end{Definition}

There are algebra homomorphisms
\begin{align*}
\iota^{(1)}: \kk[z_1, \ldots, z_n] \to \mathcal{H}_n(A),
\hspace{10mm}
\iota^{(2)}: A^{\otimes n} \to \mathcal{H}_n(A),
\hspace{10mm}
\iota^{(3)}: \kk\mathfrak{S}_n \to \mathcal{H}_n(A).
\end{align*}
Abusing notation, we use the same labels for elements of the domain of these maps as for their images in \(\mathcal{H}_n(A)\); however, Proposition \ref{AffBasis} will assert that this abuse should result in no significant confusion, as \(\iota^{(1)}, \iota^{(2)}, \iota^{(3)}\) are in fact embeddings.

\begin{Remark}\label{afflit}
If one takes \(A=\kk\), then \(\mathcal{H}_n(A)\) yields the {\em degenerate affine Hecke algebra}, so Definition \ref{AffDef} can be viewed as a generalization of this construction; see \cite{KLin}. Relatedly,  Costello and Grojnowski \cite{CG} construct a {\it Cherednik algebra} (or {\it degenerate double affine Hecke algebra}) \(\overline{\mathcal{H}}_n\) associated to a  commutative Frobenius algebra \(H\). Here we have extended their construction to the case of non-commutative symmetric algebras by making a few simplifying modifications to the  last two paragraphs of \cite[\S4.2]{CG}. Explicitly, we take \(H_\Gamma = A\), \(u=1\), and replace \cite[Definition 4.2.1]{CG} with the trivial action \(y_l(\Theta) = 0\). Related generalizations of degenerate affine Hecke algebras in the noncommutative case have also been studied by Tsuchioka \cite{Tsu}.
\end{Remark}

\subsection{Bases for affinized symmetric algebras} In this section we prove freeness properties of \(\mathcal{H}_n(A)\).

\begin{Lemma}\label{Vdef}
Let \(V\) be the graded \(\kk\)-module \(V:=\kk[z_1, \ldots, z_n] \otimes A^{\otimes n} \otimes \kk\mathfrak{S}_n\). Defining an action of \(\mathcal{H}_n(A)\) on \(V\) via
\begin{align*}
z_i \cdot (f \otimes \bb \otimes w) &= z_i f \otimes \bb \otimes w,\\
\ba \cdot (f \otimes \bb \otimes w) &= f \otimes \ba\bb \otimes w,\\
s_j \cdot (f \otimes \bb \otimes w) &= {}^{s_j}\hspace{-0.5mm}f \otimes {}^{s_j}\hspace{-0.3mm}\bb \otimes s_jw + \nabla_j(f) \otimes \Delta_{j,j+1} \bb \otimes w,
\end{align*}
for all \(i \in [1,n]\), \(j \in [1,n-1]\), \(f \in \kk[z_1, \ldots, z_n]\), \(\ba, \bb \in A^{\otimes n}\), and \(w \in \kk \mathfrak{S}_n\), gives \(V\) the structure of a graded \(\mathcal{H}_n(A)\)-module.
\end{Lemma}

\begin{proof}
First note that the defining relations of \(A^{\otimes n}\) and \(\kk[z_1, \ldots, z_n]\) are clearly satisfied in this action, as is the relation (\ref{AZ}).

For any \(i \in [1,n-1]\), \(\ba \in A^{\otimes n}\), we have
\begin{align*}
 s_i \cdot (\ba \cdot (f \otimes \bb \otimes w)) &= s_i \cdot (f \otimes \ba \bb \otimes w)\\
 &= {}^{s_i}\hspace{-0.5mm}f \otimes {}^{s_i}\hspace{-0.3mm}(\ba \bb) \otimes s_iw + \nabla_i(f) \otimes \Delta_{i,i+1}\ba \bb \otimes w,
 \end{align*}
 and
 \begin{align*}
{}^{s_i}\hspace{-0.3mm}\ba \cdot ( s_i \cdot (f \otimes \bb \otimes w)) &= {}^{s_i}\hspace{-0.3mm}\ba \cdot ({}^{s_i}\hspace{-0.5mm}f \otimes  {}^{s_i}\hspace{-0.3mm}\bb \otimes s_i w) + {}^{s_i}\hspace{-0.3mm}\ba \cdot(\nabla_i(f) \otimes \Delta_{i,i+1}\bb \otimes w)\\
&= {}^{s_i}\hspace{-0.5mm}f \otimes ({}^{s_i}\hspace{-0.3mm}\ba)({}^{s_i}\hspace{-0.3mm}\bb) \otimes s_i w + \nabla_i(f) \otimes {}^{s_i}\hspace{-0.3mm}\ba\Delta_{i,i+1} \bb \otimes w\\
&={}^{s_i}\hspace{-0.5mm}f \otimes {}^{s_i}\hspace{-0.3mm}(\ba \bb) \otimes s_i w + \nabla_i(f) \otimes \Delta_{i,i+1} \ba\bb \otimes w,
\end{align*}
applying Lemma \ref{tauDelta}(ii) in the last step. Thus \(s_i \ba = {}^{s_i}\hspace{-0.3mm}\ba s_i\) as operators on \(V\), so the action satisfies relation (\ref{AS}).

For \(i \in [1,n-1]\) and \(j \in [1,n]\) we have
\begin{align*}
s_i \cdot (z_j \cdot (f \otimes \bb \otimes w)) &= s_i \cdot (z_j f \otimes \bb \otimes w)\\
&= z_{s_ij}({}^{s_i}\hspace{-0.5mm}f) \otimes {}^{s_i}\hspace{-0.3mm}\bb \otimes s_i w + \nabla_i(z_j f) \otimes \Delta_{i,i+1}\bb \otimes w,
\end{align*}
and
\begin{align*}
z_{s_ij} \cdot ( s_i \cdot (f \otimes \bb \otimes w)) &= z_{s_ij} \cdot ({}^{s_i}\hspace{-0.5mm}f \otimes {}^{s_i}\hspace{-0.3mm}\bb \otimes s_iw) + z_{s_i j} \cdot (\nabla_i(f) \otimes \Delta_{i,i+1} \bb \otimes w)\\
&= z_{s_ij}({}^{s_i}\hspace{-0.5mm}f) \otimes {}^{s_i}\hspace{-0.3mm}\bb \otimes s_iw + z_{s_i j} \nabla_i(f) \otimes \Delta_{i,i+1} \bb \otimes w.
\end{align*}
Then by Lemma \ref{DivDiff}(iii), \((s_iz_j - z_{s_ij}s_i) = (\delta_{i,j}-\delta_{i+1,j})\Delta_{i,i+1}\) as operators on \(V\), so the action satisfies relation (\ref{SZ}).

It remains to prove that the action satisfies the defining Coxeter relations of \(\kk\mathfrak{S}_n\). For \(i,j \in [1,n-1]\) with \(|i-j|>1\), we have
\begin{align*}
s_i \cdot (s_j \cdot (f \otimes \bb \otimes w)) &= s_i \cdot ({}^{s_j}\hspace{-0.5mm}f \otimes {}^{s_j}\hspace{-0.3mm}\bb \otimes s_j w) + s_i \cdot (\nabla_j(f) \otimes \Delta_{j,j+1}\bb \otimes w) \\
&= {}^{s_is_j}\hspace{-0.5mm}f \otimes {}^{s_is_j}\hspace{-0.3mm}\bb \otimes s_is_jw + \nabla_i({}^{s_j}\hspace{-0.5mm}f) \otimes \Delta_{i,i+1}({}^{s_j}\hspace{-0.3mm}\bb)\otimes s_jw \\
&\hspace{10mm}+
{}^{s_i}\hspace{-0.3mm}(\nabla_j(f)) \otimes {}^{s_i}\hspace{-0.3mm}(\Delta_{j,j+1}\bb) \otimes s_iw\\
&\hspace{20mm}+
\nabla_i(\nabla_j(f)) \otimes \Delta_{i,i+1}\Delta_{j,j+1}\bb \otimes w,
\end{align*}
and similarly
\begin{align*}
s_j\cdot (s_i \cdot (f \otimes \bb \otimes w)) 
&= {}^{s_js_i}\hspace{-0.5mm}f  \otimes {}^{s_js_i}\hspace{-0.3mm}\bb \otimes s_js_iw + \nabla_j({}^{s_i}\hspace{-0.5mm}f) \otimes \Delta_{j,j+1}({}^{s_i}\hspace{-0.3mm}\bb) \otimes s_iw \\
&\hspace{10mm}+
{}^{s_j}\hspace{-0.5mm}(\nabla_i(f)) \otimes {}^{s_j}\hspace{-0.5mm}(\Delta_{i,i+1}\bb) \otimes s_jw\\
&\hspace{20mm}+
\nabla_j(\nabla_i(f)) \otimes \Delta_{j,j+1}\Delta_{i,i+1}\bb \otimes w.
\end{align*}
But, since \(\nabla_i({}^{s_j}\hspace{-0.5mm}f) = {}^{s_j}\hspace{-0.3mm}(\nabla_i(f))\), \(\nabla_j({}^{s_i}\hspace{-0.3mm}f)={}^{s_i}\hspace{-0.5mm}(\nabla_j(f))\), and \(\nabla_i(\nabla_j(f)) = \nabla_j(\nabla_i(f))\), it follows that the relation
\(
s_is_j = s_j s_i, \textup{ for all } i,j \in [1, n-1] \textup{ such that } |i-j|>1,
\)
holds as operators on \(V\).

 Next, for \(i \in [1,n-1]\), we have
\begin{align*}
s_i \cdot (s_i \cdot (f \otimes \bb \otimes w)) &= s_i \cdot ({}^{s_i}\hspace{-0.5mm}f \otimes {}^{s_i}\hspace{-0.3mm}\bb\otimes s_iw) + s_i \cdot (\nabla_i(f) \otimes \Delta_{i,i+1} \bb \otimes w)\\
&=f \otimes \bb \otimes w + \nabla_i({}^{s_i}\hspace{-0.5mm}f)\otimes \Delta_{i,i+1}({}^{s_i}\hspace{-0.3mm}\bb) \otimes s_iw\\
&\hspace{10mm}+
{}^{s_i}\hspace{-0.3mm}(\nabla_i(f)) \otimes {}^{s_i}\hspace{-0.3mm}(\Delta_{i,i+1}\bb) \otimes s_iw + \nabla_i^2(f) \otimes \Delta_{i+1}^2 \bb \otimes w\\
&= f \otimes \bb \otimes w,
\end{align*}
applying Lemma \ref{DivDiff}(i),(ii), and Lemma \ref{tauDelta}(i) in the last step. Thus the relation
\(
s_i^2 = 1, \textup{ for all } i \in [1,n-1],
\)
 holds as operators on \(V\).

Now fix \(i \in [1,n-2]\), and \(j \in [1,n]\). By the previously proved properties, we have that, as operators on \(V\):
\begin{align*}
s_{i+1}s_is_{i+1}z_j &= s_{i+1}s_i(z_{s_{i+1}j}s_{i+1} + (\delta_{i+1,j} - \delta_{i+2,j})\Delta_{i+1,i+2})\\
&= s_{i+1}s_iz_{s_{i+1}j}s_{i+1} + (\delta_{i+1,j}-\delta_{i+2,j})\Delta_{i,i+1}s_{i+1}s_i\\
&=s_{i+1}(z_{s_is_{i+1}j}s_i + (\delta_{i,s_{i+1}j}-\delta_{i+1,s_{i+1}j})\Delta_{i,i+1})s_{i+1}\\
&\hspace{10mm}+ (\delta_{i+1,j}-\delta_{i+2,j})\Delta_{i,i+1}s_{i+1}s_i\\
&=s_{i+1}z_{s_is_{i+1}j}s_is_{i+1} + (\delta_{i,j}-\delta_{i+2,j})\Delta_{i,i+2}\\
&\hspace{10mm}+ (\delta_{i+1,j}-\delta_{i+2,j})\Delta_{i,i+1}s_{i+1}s_i\\
&=(z_{s_{i+1}s_is_{i+1}j}s_{i+1} + (\delta_{i+1,s_is_{i+1}j} - \delta_{i+2,s_is_{i+1}j})\Delta_{i+1,i+2})s_is_{i+1}\\
&\hspace{10mm} + (\delta_{i,j}-\delta_{i+2,j})\Delta_{i,i+2}    + (\delta_{i+1,j}-\delta_{i+2,j})\Delta_{i,i+1}s_{i+1}s_i\\
&=z_{s_{i+1}s_is_{i+1}j}s_{i+1}s_is_{i+1} + (\delta_{i,j}-\delta_{i+1,j})\Delta_{i+1,i+2}s_is_{i+1} \\
&\hspace{10mm} + (\delta_{i,j}-\delta_{i+2,j})\Delta_{i,i+2}    + (\delta_{i+1,j}-\delta_{i+2,j})\Delta_{i,i+1}s_{i+1}s_i
\end{align*}
and similarly 
\begin{align*}
s_is_{i+1}s_iz_j &=z_{s_is_{i+1}s_ij}s_is_{i+1}s_i + (\delta_{i+1,j}-\delta_{i+2,j})\Delta_{i,i+1}s_{i+1}s_i\\
&\hspace{10mm}+(\delta_{i,j} - \delta_{i+2,j})\Delta_{i,i+2}+(\delta_{i,j}-\delta_{i+1,j})\Delta_{i+1,i+2}s_is_{i+1}.
\end{align*}
Thus \((s_is_{i+1}s_i - s_{i+1}s_is_{i+1})z_j = z_{s_is_{i+1}s_ij}(s_is_{i+1}s_i - s_{i+1}s_is_{i+1})\) as operators on \(V\). Now we prove that \(s_is_{i+1}s_i = s_{i+1}s_i s_{i+1}\) as operators on \(V\), via induction on the degree of \(f\) in the term \(f \otimes \bb \otimes w \in V\). The base case \(\deg(f) = 0\) is obvious. If the claim holds for \(f\), then
\begin{align*}
(s_is_{i+1}s_i - s_{i+1}s_is_{i+1})\cdot(z_jf \otimes \bb \otimes w)&= (s_is_{i+1}s_i - s_{i+1}s_is_{i+1}) \cdot (z_j\cdot(f \otimes \bb \otimes w))
\\
&=((s_is_{i+1}s_i - s_{i+1}s_is_{i+1})z_j) \cdot (f \otimes \bb \otimes w)\\
&=(z_{s_is_{i+1}s_ij}(s_is_{i+1}s_i - s_{i+1}s_is_{i+1}))\cdot (f \otimes \bb \otimes w)\\
&= z_{s_is_{i+1}s_ij} \cdot ((s_is_{i+1}s_i - s_{i+1}s_is_{i+1})\cdot (f \otimes \bb \otimes w))\\
&=z_{s_is_{i+1}s_ij} \cdot 0=0,
\end{align*}
proving the claim. Thus the Coxeter relations hold as operators on \(V\), and \(V\) is an \(\mathcal{H}_n(A)\)-module. 
\end{proof}

\begin{Theorem}\label{AffBasis}\(\)
\begin{enumerate}
\item The map \(V =  \kk[z_1, \ldots, z_n] \otimes A^{\otimes n} \otimes \mathfrak{S}_n\to \mathcal{H}_n(A)\) defined by
\begin{align*}
f \otimes \ba \otimes w \mapsto f\ba w
\end{align*}
is an an isomorphism of graded \(\mathcal{H}_n(A)\)-modules.
\item \(\mathcal{H}_n(A)\) is free as a \(\kk\)-module, with graded dimension 
\begin{align*}
\dim_q \mathcal{H}_n(A) = n! \left(\frac{\dim_q A}{1-q^d}\right)^n.
\end{align*}
\end{enumerate}
\end{Theorem}

\begin{proof}
Let \(B_1\) be a basis for \(\kk[z_1, \ldots, z_n]\), \(B_2\) be a basis for \(A^{\otimes n}\), and let \(B_3\) be a basis for \(\kk\mathfrak{S}_n\). Define the sets
\begin{align*}
B&=\{f \otimes \ba \otimes w \mid  f \in B_1, \ba \in B_2, w \in B_3\} \subset V,\\
\mathcal{B}&=\{f \ba w \mid f \in B_1, \ba \in B_2, w \in B_3\} \subset \mathcal{H}_n(A).
\end{align*}
Then \(B\) is a \(\kk\)-basis for \(V\). It is straightforward to see that inductive application of the commutation relations (\ref{AZ})--(\ref{SZ}) allows one to write any element in \(\mathcal{H}_n(A)\) as a \(\kk\)-linear combination of elements of \(\mathcal{B}\), so \(\mathcal{B}\) is a spanning set for \(\mathcal{H}_n(A)\). Moreover, for every \(f\ba w \in \mathcal{B}\), \(f\ba w \cdot (1 \otimes 1 \otimes 1) = f \otimes \ba \otimes w\), so the elements of \(\mathcal{B}\) are linearly independent as operators on \(V\), and thus \(\mathcal{B}\) constitutes a \(\kk\)-basis for \(\mathcal{H}_n(A)\).

Since \(V\) is a cyclic \(\mathcal{H}_n(A)\)-module, generated by \(1 \otimes 1 \otimes 1\), we have an \(\mathcal{H}_n(A)\)-module homomorphism \(\mathcal{H}_n(A) \to V\) given by \(1 \mapsto 1 \otimes 1 \otimes 1\), which sends \(f \ba w \in \mathcal{B}\) to \(f \otimes \ba \otimes w \in B\). Since the map is a bijection on \(\kk\)-bases, it is an isomorphism, proving (i). Part (ii) follows from (i).
 \end{proof}

\begin{Corollary}\label{sfa}
Let \(f \in \kk[z_1, \ldots, z_n]\), and \(\ba \in A^{\otimes n}\). 
\begin{enumerate}
\item For all \(i \in [1, n-1]\),  we have
\begin{align*}
s_i f \ba = ( {}^{s_i}\hspace{-0.5mm}f)(  {}^{s_i}\hspace{-0.5mm}\ba) s_i + \nabla_i(f) \Delta_{i,i+1} \ba.
\end{align*}
\item For all \(w \in \mathfrak{S}_n\), we have
\begin{align*}
 w  f \ba =  ({}^w \hspace{-0.5mm} f) ( {}^w \hspace{-0.5mm} \ba)w + (*),
\end{align*}
where \((*)\) is a \(\kk\)-linear combination of terms of the form \(f'\ba'w'\), where \(f' \in \kk[z_1, \ldots, z_n]\), \(\ba' \in A^{\otimes n}\), and \(w' \in \mathfrak{S}_n\) with \(\ell(w')<\ell(w)\).
\end{enumerate}
\end{Corollary}
\begin{proof}
Part (i) follows from Theorem \ref{AffBasis}(i) and the action of \(s_i\) on \(V\) defined in Lemma \ref{Vdef}. Part (ii) follows from inductive application of (i).
\end{proof}

\begin{Corollary}\label{freeact} Let $B_1$ be a $\k$-basis of $\k[z_1,\dots,z_n]$ and 
\(B_2 \) be a basis for \(A^{\otimes n}\). Then:
\begin{enumerate}
\item The sets $\{  f \ba w  \mid f \in B_1, \ba \in B_2, w \in \mathfrak{S}_n\}$ and $\{  w f \ba  \mid f \in B_1, \ba \in B_2, w \in \mathfrak{S}_n\}$ are $\k$-bases of $\mathcal{H}_n(A)$.
\item The set \(\{ \ba w \mid \ba \in B_2, w \in \mathfrak{S}_n\}\) is a basis for \(\mathcal{H}_n(A)\) as a left \(\kk[z_1, \ldots, z_n]\)-module, and \(\{ w \ba \mid \ba \in B_2, w \in \mathfrak{S}_n\}\) is a basis for \(\mathcal{H}_n(A)\) as a right \(\kk[z_1, \ldots, z_n]\)-module.
\item The set \(\{ f w \mid f \in B_1, w \in \mathfrak{S}_n\}\) is a basis for \(\mathcal{H}_n(A)\) as a left \(A^{\otimes n}\)-module, and \(\{ wf \mid f \in B_1, w \in \mathfrak{S}_n\}\) is a basis for \(\mathcal{H}_n(A)\) as a right \(A^{\otimes n}\)-module.
\item The set \(\{  f \ba  \mid f \in B_1,\ba \in B_2\}\) is a basis for \(\mathcal{H}_n(A)\) as both a left and right \(\kk\mathfrak{S}_n\)-module.
\end{enumerate}
\end{Corollary}
\begin{proof}
By Theorem \ref{AffBasis}(i), the first set in (i) is a basis for \(\mathcal{H}_n(A)\).  Applying Corollary \ref{sfa}(ii), one may use induction on the length of \(w \in \mathfrak{S}_n\) to see that the second set in (i) 
is also a basis for \(\mathcal{H}_n(A)\), completing the proof of part (i). Parts (ii)-(iv) follow from part (i) and the fact that \(\ba f = f \ba\) for all \( f \in \kk[z_1, \ldots, z_n]\) and \(\ba \in A^{\otimes n}\).
\end{proof}

\subsection{Antiautomorphisms of affinized symmetric algebras}
In this section we show that an antiautomorphism of the symmetric algebra \(A\) extends to an antiautomorphism of the affinization \(\mathcal{H}_n(A)\).

\begin{Lemma}\label{affop}
Suppose that \(\nu:A \to A^\op\) is an isomorphism of graded \(\kk\)-algebras. Then the map \(\widehat{\nu}:\mathcal{H}_n(A) \to \mathcal{H}_n(A)^\op\) defined by
\begin{align*}
\widehat{\nu}(z_i) = z_i, \hspace{10mm}
\widehat{\nu}(\ba) = (\nu \otimes \cdots \otimes \nu)(\ba), \hspace{10mm}
\widehat{\nu}(s_j) = s_j,
\end{align*}
for all \(i \in [1,n]\), \(j \in [1,n-1]\) and \(\ba \in A^{\otimes n}\), is an isomorphism of graded \(\kk\)-algebras.
\end{Lemma}

\begin{proof}
It is clear that the \(\widehat{\nu}\) is a homomorphism upon restriction to the subalgebras \(\kk[z_1, \ldots, z_n]\), \(A^{\otimes n}\), and \(\kk \mathfrak{S}_n\). It is likewise straightforward to check that \(\widehat{\nu}\) preserves the commutation relations (\ref{AZ}) and (\ref{AS}).

It remains to verify that \(\widehat{\nu}\) preserves relation (\ref{SZ}). We have, for all \(x,y \in A\),
\begin{align*}
m \circ (\nu \otimes \nu) \circ \tau_{A,A}(x \otimes y) = \nu(y)\nu(x)
= \nu(xy)
= \nu \circ m(x \otimes y).
\end{align*}
Thus \(m \circ (\nu \otimes \nu) \circ \tau_{A,A} = \nu \circ m\), so 
\begin{align*}
\tau_{A^*,A^*} \circ (\nu \otimes \nu)^* \circ m^* = \tau_{A,A}^* \circ (\nu \otimes \nu)^* \circ m^*=m^* \circ \nu^*.
\end{align*}
Therefore
\begin{align*}
\Delta \circ \nu &= (\varphi^{-1} \otimes \varphi^{-1}) \circ m^* \circ \varphi \circ \nu
=(\varphi^{-1} \otimes \varphi^{-1}) \circ m^* \circ \nu^* \circ \varphi\\
&= (\varphi^{-1} \otimes \varphi^{-1}) \circ \tau_{A^*,A^*} \circ (\nu \otimes \nu)^* \circ m^* \circ \varphi
= \tau_{A,A} \circ (\varphi^{-1} \otimes \varphi^{-1}) \circ (\nu \otimes \nu)^* \circ m^* \circ \varphi\\
&= \tau_{A,A} \circ (\nu \otimes \nu) \circ  (\varphi^{-1} \otimes \varphi^{-1}) \circ m^* \circ \varphi =
(\nu \otimes \nu) \circ \tau_{A,A} \circ  (\varphi^{-1} \otimes \varphi^{-1}) \circ m^* \circ \varphi\\
&= (\nu \otimes \nu) \circ \tau_{A, A} \circ \Delta.
\end{align*}
Thus by Lemma \ref{tauDelta}(i) we have
\begin{align*}
\Delta(1) = \Delta \circ \nu(1) = (\nu \otimes \nu) \circ \tau_{A,A} \circ \Delta(1) = (\nu \otimes \nu) \circ \Delta(1).
\end{align*}
Thus for all \(i \in [1,n-1]\) we have 
\begin{align*}
\widehat{\nu} ( \Delta_{i,i+1}) = (\nu \otimes \cdots \otimes \nu) \circ \iota_{i,i+1}\circ \Delta(1)
= \iota_{i,i+1} \circ (\nu \otimes \nu)\circ \Delta(1)
= \iota_{i,i+1}\circ \Delta(1)
= \Delta_{i,i+1}.
\end{align*}
Therefore, for all \(i \in [1,n-1]\) and \(j \in [1,n]\), we have
\begin{align*}
\widehat{\nu}(s_iz_j - z_{s_ij}s_i) &=
z_j s_i - s_i z_{s_i j} =
-(s_iz_{s_ij} - z_j s_i) =
-(\delta_{i,s_i j} - \delta_{i+1,s_i j})\Delta_{i,i+1}\\
& =
(\delta_{i,j} - \delta_{i+1,j})\Delta_{i,i+1} =
\widehat{\nu}((\delta_{i,j} - \delta_{i+1,j})\Delta_{i,i+1}).
\end{align*}
Thus \(\widehat{\nu}\) preserves relation (\ref{SZ}), so \(\widehat{\nu}\) is a graded homomorphism of \(\kk\)-algebras. 
Now by Corollary~\ref{freeact}(i), $\widehat{\nu}$ is an isomorphism. 
\end{proof}

\subsection{Centers of affinized symmetric algebras}
Let $$X_n:=\kk[z_1, \ldots, z_n] \otimes Z(A)^{\otimes n}$$
considered as a subalgebra of $\kk[z_1, \ldots, z_n] \otimes A^{\otimes n}$ which in turn is a subalgebra of $\mathcal{H}_n(A)$ in a natural way. The symmetric group \(\mathfrak{S}_n\) acts on  $X_n$ with algebra automorphisms as follows:
\begin{align*}
w \cdot (f \otimes \ba) = {}^w\hspace{-0.5mm}f \otimes {}^w\hspace{-0.3mm}\ba \qquad\qquad(f\in \k[z_1, \ldots, z_n],\ \ba\in Z(A)^{\otimes n}).
\end{align*} 

\begin{Proposition}\label{AffCenter}
The center of \(\mathcal{H}_n(A)\) is the subalgebra of invariants $X_n^{\mathfrak{S}_n}$.
\end{Proposition}

\begin{proof}
Let $x\in X_n^{\mathfrak{S}_n}$. Write \(x=\sum_j f_j \ba_j\) for some \(f \in \kk[z_1, \ldots, z_n]\) and \(\ba_j \in Z(A)^{\otimes n}\). Clearly \(x\) commutes with elements of the subalgebras \(A^{\otimes n}\) and \(\kk[z_1, \ldots, z_n]\) of \(\mathcal{H}_n(A)\). Now
\begin{align*}
s_ix - xs_i &= s_i \sum_j f_j \ba_j - xs_i 
=\sum_j ({}^{s_i}\hspace{-0.5mm}f_j)({}^{s_i}\hspace{-0.3mm}\ba_j) s_i + \sum_j \nabla_i(f_j) \Delta_{i,i+1}\ba_j -  xs_i \\
&={}^{s_i}\hspace{-0.3mm}xs_i + \sum_j \nabla_i(f_j) \Delta_{i,i+1}\ba_j -  xs_i
=  \sum_j \nabla_i(f_j) \Delta_{i,i+1}\ba_j,
\end{align*}
applying Corollary \ref{sfa}(i) for the second equality. Since \(\kk[z_1,\ldots, z_n]\) acts freely on the left of \(\mathcal{H}_n(A)\) by Corollary \ref{freeact}, we may show that the last term is zero by instead showing that \(z_i - z_{i+1}\) acts on this term as zero:
\begin{align*}
(z_i-z_{i+1}) \sum_j \nabla_i(f_j) \Delta_{i,i+1}\ba_j &= \sum_j (f_j - f_j^{s_i})\Delta_{i,i+1} \ba_j\\
&= \sum_j f_j \ba_j\Delta_{i,i+1}  - \sum_j ({}^{s_i}\hspace{-0.5mm}f_j)({}^{s_i}\hspace{-0.3mm}\ba_j)\Delta_{i,i+1} \\
&= (x- {}^{s_i}\hspace{-0.3mm}x) \Delta_{i,i+1} = 0.
\end{align*}
In the second equality we have applied centrality of \(\ba_j\) in \(A^{\otimes n}\) for the first sum, and Lemma \ref{tauDelta}(ii) for the second sum. Therefore \(s_ix = xs_i\), and \(X_n^{\mathfrak{S}_n} \subseteq Z(\mathcal{H}_n(A))\).

Now we show that \(Z(\mathcal{H}_n(A)) \subseteq X_n^{\mathfrak{S}_n}\). Let \(0 \neq x \in Z(\mathcal{H}_n(A))\). By Theorem \ref{AffBasis} we may write \(x = \sum_{w \in \mathfrak{S}_n} y_w w\) for some \(y_w \in \kk[z_1, \ldots, z_n] \otimes A^{\otimes n}\). Let \(l\) be maximal such that \(y_u \neq 0\) and \(\ell(u) = l\) for some \(u \in \mathfrak{S}_n\). Then, using centrality of \(x\) and Corollary \ref{sfa}(ii), we have 
\begin{align*}
\sum_{w \in \mathfrak{S}_n}(z_1^1 z_2^2\cdots z_n^n) y_ww =
\sum_{w \in \mathfrak{S}_n} y_w w(z_1z_2^2\cdots z_n^n) =
\sum_{\substack{w \in \mathfrak{S}_n\\\ell(w)=l}}(z_{w1}^1z_{w2}^2\cdots z_{wn}^n)y_w w + (*),
\end{align*}
where \((*)\) is a linear combination of basis elements of the form \(y_{w'}'w'\), where \(y'_{w'} \in \kk[z_1, \ldots, z_n] \otimes A^{\otimes n}\) and \(\ell(w')<l\). Then, again by Theorem \ref{AffBasis}, 
\begin{align*}
(z_1z_2^2 \cdots z_n^n- z_{u1}z_{u2}^2\cdots z_{un}^n)y_u = 0,
\end{align*}
but since \(\kk[z_1, \ldots, z_n]\) acts freely on the left of \(\mathcal{H}_n(A)\) and \(y_u \neq 0\), we have that \(z_1z_2^2 \cdots z_n^n = z_{u1}^1z_{u2}^2\cdots z_{un}^n\), and hence \(u=1\). Thus \(x \in \kk[z_1, \ldots, z_n] \otimes A^{\otimes n}\). 

For \(t \in \ZZ_{\geq 0}\), let \(B_1^{t}\) be a basis for the homogeneous degree-\(t\) polynomials in \(\kk[z_1, \ldots, z_n]\). Let \(B_1 = \bigcup_{t=0}^{\infty} B_1^{t}\). Then \(x=\sum_{f \in B_1} f\ba_f\) for some \(\ba_f \in A^{\otimes n}\). For all \(\bb \in A^{\otimes n}\), we have
\begin{align*}
\sum_{f \in B_1} f\ba_f \bb = \sum_{f \in B_1}\bb f \ba_f = \sum_{f \in B_1}f \bb \ba_f,
\end{align*}
so by Theorem \ref{AffBasis}, \(\ba_f\bb = \bb\ba_f\) for all \(f\), and thus \(\ba_f \in Z(A^{\otimes n}) = Z(A)^{\otimes n}\) for all \(f\), so \(x \in X_n\).

We write \(\deg_z(x)=m\) if \(m\) is maximal such that \(\ba_f \neq 0\) for some \(f \in B^m_1\). We argue by induction on \(\deg_z(x)\) that \(x \in X_n^{\mathfrak{S}_n}\). If \(\deg_z(x) = 0\), we have \(x=\ba\) for some \(\ba \in Z(A)^{\otimes n}\). Let \(i \in [1, n-1]\). Then we have
\begin{align*}
\ba s_i = s_i \ba ={}^{s_i}\hspace{-0.3mm} \ba s_i,
\end{align*}
so Theorem \ref{AffBasis} implies that \(\ba = {}^{s_i}\hspace{-0.3mm}\ba\), and thus \(x \in X_n^{\mathfrak{S}_n}\). Now for the induction step, assume \(\deg_z(x)=m\), and \(x' \in X_n^{\mathfrak{S}_n}\) for all \(x' \in Z(\mathcal{H}_n(A))\) with \(\deg_z(x')<m\). Let \(i \in [1, n-1]\), and note that if \(f \in B_1^t\), then \(\nabla_i(f)\) is in the span of \(\bigcup_{r=0}^{t-1}B_1^r\). Then, applying Corollary \ref{sfa}(i) for the second equality, we have
\begin{align*}
\sum_{f \in B_1}f\ba_fs_i &= \sum_{f \in B_1}s_i f \ba_f 
= \sum_{f\in B_1} ({}^{s_i}\hspace{-0.5mm}f)({}^{s_i}\hspace{-0.3mm}\ba_f) s_i + \sum_{f \in B_1} \nabla_i(f)\Delta_{i,i+1}\ba_f\\
&={}^{s_i}\hspace{-0.5mm}\left(\sum_{f \in B^m_1} f \ba_f\right) s_i + (*),
\end{align*}
where \((*)\) is a linear combination of terms of the form \(fy\), where \(f \in B_1^t\) for \(t <m\), and \(y \in A^{\otimes n} \otimes \kk\mathfrak{S}_n\). Thus, writing \(x_m := \sum_{f \in B^m_1}f\ba_f\), it follows from Theorem \ref{AffBasis}(i) that \((x_m -{}^{s_i}\hspace{-0.3mm}x_m)s_i=0\), and thus \(x_m = {}^{s_i}\hspace{-0.3mm}x_m\), for all \(i\in[1,n-1]\) by Corollary  \ref{freeact}. Then \(x_m \in X_n^{\mathfrak{S}_n}\), so \(x_m \in Z(\mathcal{H}_n(A))\). Thus \(x-x_m \in Z(\mathcal{H}_n(A))\) and \(\deg_z(x-x_m)<m\), so \(x-x_m \in X_n^{\mathfrak{S}_n}\) by the induction assumption. Therefore  \(x = x_m + (x-x_m) \in X_n^{\mathfrak{S}_n}\), as desired.
\end{proof}

We have the immediate corollary:
\begin{Corollary}\label{ZAcent}
\(
(Z(A)^{\otimes n})^{\mathfrak{S}_n} = Z(\mathcal{H}_n(A)) \cap A^{\otimes n}.
\)
\end{Corollary}

\subsection{Cyclotomic quotients} 

By Theorem \ref{AffBasis} and relation (\ref{AS}), the subalgebra of \(\mathcal{H}_n(A)\) generated by \(A^{\otimes n}\) and \(\kk \mathfrak{S}_n\) may be identified with the wreath product \(A \wr \mathfrak{S}_n\). For \(r \in [1,n]\), define the {\em Jucys-Murphy elements} of \(A \wr \mathfrak{S}_n\) as follows:
\begin{align*}
l_r := -\sum_{t=1}^{r-1} \Delta_{t,r} (t,r),
\end{align*}
where \((t,r) \in \mathfrak{S}_n\) is the transposition of \(t\) and \(r\).

\begin{Lemma}\label{JMcent}
The Jucys-Murphy elements centralize the subalgebra \(A^{\otimes n}\) of \(\mathcal{H}_n(A)\).
\end{Lemma}
\begin{proof}
Let \(t<r \in [1,n]\). Let \(\ba =a_1 \otimes \cdots \otimes a_n\in A^{\otimes n}\), and set
\begin{align*}
\widehat{\ba}:= a_1 \otimes \cdots \otimes a_{t-1} \otimes 1 \otimes a_{t+1} \otimes \cdots \otimes a_{r-1} \otimes 1 \otimes a_{r+1} \otimes \cdots \otimes a_n.
\end{align*}
Then \(\ba = \widehat{\ba}\,\iota_{t,r}(a_t \otimes a_r)\), and 
\begin{align*}
\Delta_{t,r}(t,r)\ba &= \Delta_{t,r}({}^{(t,r)}\hspace{-.3mm}\ba)(t,r)
= \Delta_{t,r} \widehat{\ba} \iota_{t,r}(a_r \otimes a_t)(t,r)
= \widehat{\ba}\Delta_{t,r} \iota_{t,r}(a_r \otimes a_t)(t,r)\\
&=\widehat{\ba}   \iota_{t,r}( \Delta(1)(a_r \otimes a_t))        (t,r)
= \widehat{\ba}  \iota_{t,r}( (a_t \otimes a_r)\Delta(1))        (t,r)
=\ba  \Delta_{t,r}        (t,r),
\end{align*}
using Lemma \ref{tauDelta}(ii) for the fifth equality. Thus the Jucys-Murphy elements are linear combinations of elements which centralize the subalgebra \(A^{\otimes n}\).
\end{proof}

\begin{Lemma}\label{JMcentS}
For \(r \in [1,n]\), the Jucys-Murphy element \(l_r\) centralizes the subalgebra \(\kk \mathfrak{S}_{r-1}\otimes A^{\otimes r-1}\) of \(\mathcal{H}_n(A)\) generated by \(s_1, \ldots, s_{r-2}\) and $A^{\otimes r-1}\otimes 1^{\otimes n-r+1}$. In particular, $l_1,\dots,l_n$ commute.
\end{Lemma}
\begin{proof}
Let \(i \in [1,r-2]\). Then
\begin{align*}
s_i l_r = -\sum_{t=1}^{r-1} s_i \Delta_{t,r}(t,r) =  -\sum_{t=1}^{r-1}  \Delta_{s_it,r}s_i(t,r)
= -\sum_{t=1}^{r-1}  \Delta_{s_it,r}(s_it,r)s_i = l_r s_i,
\end{align*}
as required.
\end{proof}

\begin{Corollary} \label{C060617} 
The Jucys-Murphy elements $l_1,\dots,l_n$ commute. 
\end{Corollary}
\begin{proof}
Let $t<r$. Since $l_t$ lies in the subalgebra generated by \(A^{\otimes n}\) and \(s_1, \ldots, s_{t-1}\), the result follows from Lemmas~\ref{JMcent} and \ref{JMcentS}.
\end{proof}

The next proposition involves a choice of a degree \(d\) element \(\bc \in (Z(A)^{\otimes n})^{\mathfrak{S}_n}\). Though we work with this general choice of \(\bc\), we note that there is at least one natural choice of such an element; we may take \(c = m(\Delta(1)) \in A\), and define 
\begin{align*}
\bc := c \otimes1 \otimes \cdots \otimes 1 + 1 \otimes c \otimes 1 \otimes \cdots \otimes 1 + \cdots + 1 \otimes \cdots \otimes 1 \otimes c,
\end{align*}
noting that \(c \in Z(A)\) since
\begin{align*}
x\, m(\Delta(1)) = m(x\cdot \Delta(1)) = m(\Delta(x)) = m(\Delta(1) \cdot x) = m(\Delta(1))\,x,
\end{align*}
for all \(x \in A\).

\begin{Proposition}\label{kappa}
Let \(\bc \in (Z(A)^{\otimes n})^{\mathfrak{S}_n}_d\). Let \(\beta_\bc: \mathcal{H}_n(A) \to A\wr \mathfrak{S}_n\) be the map which is the identity on the subalgebra \(A \wr \mathfrak{S}_n\) of \(\mathcal{H}_n(A)\), and sends
\begin{align*}
z_r \mapsto l_r + \bc
\end{align*}
for all \(r \in [1, \ldots, n]\). Then \(\beta_\bc\) is a surjective homomorphism of graded \(\kk\)-algebras. The kernel of \(\beta_\bc\) is the 2-sided ideal generated by \(z_1 - \bc\).
\end{Proposition}

\begin{proof}
It is clear that \(\beta_\bc\) is surjective. 
By Corollaries \ref{ZAcent} and \ref{C060617}, we have \(\beta_\bc(z_i)\beta_\bc( z_j) = \beta_{\bc}(z_j)\beta_\bc( z_i)\) for all \(i,j \in [1,n]\). 
By Corollary \ref{ZAcent} and Lemma \ref{JMcent}, we have $\beta_\bc(z_i)\beta_\bc(\ba)=\beta_\bc(\ba)\beta_\bc(z_i)$ for all \(i \in [1,n]\) and \(\ba \in A^{\otimes n}\). 
So to see that \(\beta_\bc\) is a homomorphism, it suffices to verify that \(\beta_\bc\) preserves relation (\ref{SZ}) of Definition \ref{AffDef}. 

Note that, for all \(i \in [1,n-1]\) and \(j \in [1,n]\) we have
\begin{align*}
\beta_\bc(s_i z_j) = s_i(l_j + \bc)
= -\sum_{t =1}^{j-1} s_i\Delta_{t,j}  (t,j) + s_i \bc
= - \sum_{t=1}^{j-1} \Delta_{s_it, s_ij} (s_it,s_ij)s_i + \bc s_i
\end{align*}
and
\begin{align*}
\beta_\bc(z_{s_ij}s_i) = -\sum_{t=1}^{s_ij-1} \Delta_{t,s_ij} (t,s_ij)s_i + \bc s_i.
\end{align*}
If \(j \notin \{i,i+1\}\), these terms are equal. If \(j=i\), then \(s_i t = t\) for all \(t \in [1,j-1]\), thus 
\begin{align*}
\beta_\bc(s_iz_i - z_{i+1}s_i) = \Delta_{i,i+1}  (i,i+1)s_i = \Delta_{i,i+1} = \beta_\bc(\Delta_{i,i+1}),
\end{align*}
as desired. If \(j = i+1\), then \(s_i t = t\) for all \(t \in [1,j-2]\), thus
\begin{align*}
\beta_\bc(s_iz_{i+1} - z_{i}s_i) = -\Delta_{i+1, i}  (i+1,i)s_i = -\Delta_{i,i+1} = \beta_\bc(-\Delta_{i,i+1}),
\end{align*}
as desired. Thus \(\beta_\bc\) is a homomorphism.

Since $l_1=0$, we have \(z_1- \bc \in \ker \beta_\bc\). Let \(\overline{\mathcal{H}_n(A)} := \mathcal{H}_n(A) / \mathcal{H}_n(A)(z_1- \bc)\mathcal{H}_n(A)\), and let \(\pi: \mathcal{H}_n(A) \to \overline{\mathcal{H}_n(A)}\) be the natural projection.  Then \(\beta_\bc\) factors through to a surjection \(\overline{\beta}_\bc: \overline{\mathcal{H}_n(A)}\to A \wr \mathfrak{S}_n\). Let \(\iota: A \wr \mathfrak{S}_n \to \mathcal{H}_n(A)\) be the inclusion map. We have the commuting diagram:
$$
\begin{tikzpicture}[scale=0.55, line join=bevel]
\node at (0,2) {$\mathcal{H}_n(A)$};
\node at (4,2) {$A \wr \mathfrak{S}_n$};
\node at (2,5) {$\overline{\mathcal{H}_n(A)}$};
\draw [->] (1,2.1) -- (3.05,2.1);
\draw [<-] (1,1.9) -- (3.05,1.9);
\draw [<-] (1.75,4.55) -- (0.1,2.45);
\draw [->] (2.2,4.55) -- (3.9,2.45);
\node at (0.6,3.7) {${\scriptstyle \pi}$};
\node at (3.6,3.7) {${\scriptstyle \overline{\beta}_\bc}$};
\node at (2,2.5) {${\scriptstyle \beta_\bc}$};
\node at (2,1.55) {${\scriptstyle \iota}$};
\end{tikzpicture}
$$
Note that \(\overline{\beta}_\bc \circ \pi \circ \iota = \beta_\bc \circ \iota = \textup{id}_{A \wr \mathfrak{S}_n}\). Then 
\begin{align}
\pi \circ \iota \circ \overline{\beta}_\bc \circ \pi \circ \iota = 
\pi \circ \iota \circ \textup{id}_{A \wr \mathfrak{S}_n} = \textup{id}_{\overline{\mathcal{H}_n(A)}} \circ \pi \circ \iota.\label{betamaps}
\end{align}
From the defining relations of \(\mathcal{H}_n(A)\), we have that
\(
z_{i+1}=s_iz_i s_i - \Delta_{i,i+1}s_i
\)
for all \(i \in [1, \ldots,n ]\). Thus \(\mathcal{H}_n(A)\) is generated by \(z_1-\bc\) together with the subalgebra \(A \wr \mathfrak{S}_n\). Therefore \(\pi \circ \iota\) is a surjection, so (\ref{betamaps}) implies that \(\pi \circ \iota \circ \overline{\beta}_\bc = \textup{id}_{\overline{\mathcal{H}_n(A)}}\). Thus \(\pi \circ \iota\) and \(\overline{\beta}_\bc\) are mutual inverses, proving the second statement of the lemma.
\end{proof}

Let \(l \in \ZZ_{>0}\), and let \(\bC = (\bc^{(1)}, \ldots, \bc^{(l)})\) be a sequence of elements of \((Z(A)^{\otimes n})^{\mathfrak{S}_n}_d\). We define the corresponding {\em level \(l\) cyclotomic quotient algebra} \(\mathcal{H}^{\bC}_n(A)\) to be \(\mathcal{H}_n(A)\) modulo the two-sided ideal generated by the element
\begin{align*}
\prod_{j=1}^l(z_1 -  \bc^{(j)}).
\end{align*}
By Proposition \ref{kappa}, \(\mathcal{H}^{\bC}_n(A) \cong A \wr \mathfrak{S}_n\) when \(l=1\).

\begin{Proposition}\label{cycspan}
Let $B$ be a $\k$-basis of $A^{\otimes n}$. The level \(l\) cyclotomic quotient \(\mathcal{H}^{\bC}_n(A)\) is spanned by the elements
\begin{align*}
\{z_1^{t_1} \cdots z_n^{t_n}\ba w \mid 0 \leq t_1, \ldots, t_n < l, \, \ba \in B,\, w \in \mathfrak{S}_n\}.
\end{align*}
In particular, \(\mathcal{H}^{\bC}_n(A)\) is finitely generated as a \(\kk\)-module.
\end{Proposition}
\begin{proof}
For any \(\bu = (u_1, \ldots, u_n) \in \ZZ_{\geq 0}^n\) and \(i \in [0,n]\), we define the sets
\begin{align*}
X_i^\bu &:= \{ z_1^{t_1} \cdots  z_n^{t_n} \mid
0 \leq t_1, \ldots, t_i < l,\, 0 \leq t_k \leq u_k \textup{ for } k>i \}\subseteq \kk[z_1, \ldots, z_n],
\\
Y_i^\bu &:= \textup{span}\{ fy \mid f \in X_i^\bu,\, y \in A^{\otimes n} \otimes \kk \mathfrak{S}_n\} \subseteq \mathcal{H}^{\bC}_n(A).
\end{align*}
Note that \(Y_n^\bu\) is the span of the elements in the statement. Moreover, by Theorem \ref{AffBasis}, every element of \(\mathcal{H}^{\bC}_n(A)\) belongs to some \(Y_0^\bu\). So the result follows from the following

{\em Claim.} \(Y_i^\bu \subseteq Y_{i+1}^\bu\) for all \(\bu \in \ZZ_{\geq 0}^n\) and \(i \in [0,n-1]\). 

We prove the claim by induction on \(i\). For the base case \(i=0\), let \(f = z_1^{t_1} \cdots z_n^{t_n} \in X_0^\bu\) and \(y \in A^{\otimes n} \otimes \kk \mathfrak{S}_n\). Note that, by the definition of the cyclotomic quotient and Corollary \ref{ZAcent}, we have
\(
z_1^{t_1} = \sum_{k=0}^{l-1} z_1^k \bb_k
\)
in \(\mathcal{H}^{\bC}_n(A)\), for some \(\bb_0, \ldots, \bb_k \in A^{\otimes n}\). Thus we have
\begin{align*}
fy = z_1^{t_1} \cdots z_n^{t_n} y = \sum_{k=0}^{l-1}z_1^k z_2^{t_2} \cdots z_n^{t_n} \bb_k y \in Y_1^\bu,
\end{align*}
so \(Y_0^{\bu} \subseteq Y_1^{\bu}\), as desired.

For the inductive step, let \(i \in [1,n-1]\) and suppose that  \(Y_0^\bu \subseteq \cdots \subseteq Y_i^\bu\) for all \(\bu \in \ZZ_{\geq 0}^n\). Let \(f = z_1^{t_1} \cdots z_n^{t_n} \in X_i^\bu\) for some $(t_1,\dots,t_n)=:\bt \in \ZZ_{\geq 0}^n$, \(\ba \in A^{\otimes n}\) and \(w \in \mathfrak{S}_n\). In order to show that \(Y_i^\bu \subseteq Y_{i+1}^\bu\) it suffices to show that \(f \ba w \in Y_{i+1}^\bu\). By Lemma \ref{sfa}(i), we have
\begin{align}\label{fbsig}
s_i({}^{s_i}f)({}^{s_i}\ba)s_i w= f \ba w + \nabla_i({}^{s_i}f)\Delta_{i,i+1}({}^{s_i}\ba)s_i w.
\end{align}
Note that 
\({}^{s_i}f \in X^{s_i \bt}_{i-1} \). So \(\nabla_i({}^{s_i}f)\) is in the \(\kk\)-span of \(X_{i-1}^{s_i \bt}\). Therefore 
$$\nabla_i({}^{s_i}f) \Delta_{i,i+1} ({}^{s_i} \ba)s_i w \in Y_{i-1}^{s_i \bt} \subseteq Y_i^{s_i \bt} \subseteq Y_{i+1}^\bu,
$$ 
where  the first containment holds by the induction assumption, and the second containment follows since \((s_i \bt)_{i+1} = t_i < l\), and \((s_i \bt)_k = t_k \leq u_k\) for \(k>i+1\).
Similarly 
$$({}^{s_i}f)({}^{s_i}\ba)s_i w \in Y_{i-1}^{s_i \bt} \subseteq Y_i^{s_i \bt} \subseteq Y_{i+1}^\bu.$$ 
So, to complete the proof  that \(f \ba w \in Y_{i+1}^\bu\), it suffices to show that  \(s_i Y_{i+1}^\bu  \subseteq Y_{i+1}^\bu\). For this, let \(g \in X_{i+1}^\bu\) and \(x \in A^{\otimes n} \otimes \kk \mathfrak{S}_n\). By Lemma \ref{sfa}(i), we have  
$
s_igx = {}^{s_i}g x' + \nabla_i({}^{s_i}g)x''
$ 
for some \(x', x'' \in A^{\otimes n} \otimes \kk \mathfrak{S}_n\). But \({}^{s_i}g \in X_{i+1}^\bu\), and \(\nabla_i({}^{s_i}g)\) is in the \(\kk\)-span of \(X_{i+1}^\bu\), so  \(s_igx \in Y_{i+1}^{\bu}\).  
\end{proof}

We complete this section with three conjectures. 

\begin{Conjecture}\label{C3}
The spanning set of Proposition \ref{cycspan} constitutes a \(\kk\)-basis for \(\mathcal{H}^{\bC}_n(A)\).
\end{Conjecture}

\begin{Conjecture}\label{C1}
The algebra $\mathcal{H}^{\bC}_n(A)$ is graded symmetric. 
\end{Conjecture}

\begin{Conjecture}\label{C2}
If $A$ is cyclic cellular, then so is $\mathcal{H}^{\bC}_n(A)$. 
\end{Conjecture}

Note that in level $1$ the conjectures hold. 
Indeed, Conjecture~\ref{C3} in level $1$  follows from Proposition~\ref{kappa}. 
For Conjecture~\ref{C1} we can use a bimodule isomorphism $A\wr\Si_n\to (A\wr\Si_n)^*$ given by 
$a_1 \otimes \cdots \otimes a_n \otimes \sigma \mapsto \varphi(a_{\sigma 1}) \otimes \cdots \otimes \varphi(a_{\sigma n}) \otimes (\sigma^{-1})^*$. Conjecture~\ref{C2} in level $1$ is the main result of \cite{GG}. The conjectures are also known to hold  for any level when $A=\k$. Indeed, $\mathcal{H}^{\bC}_n(\k)$ is a degenerate cyclotomic Hecke algebra. Now, for Conjecture~\ref{C3} see for example \cite[Theorem 7.5.6]{KLin},  Conjecture~\ref{C1} can be deduced for example from \cite[Corollary 6.18]{HM} and \cite{BKcyc}, and Conjecture~\ref{C2} can be seen from \cite{GL},\cite{DJM},\cite{AMR}.

\section{Zigzag algebras}

Let \(\Gamma=(\Gamma_0,\Gamma_1)\) be a connected graph without loops or multiple edges. Eventually, we will need only the case where $\Gamma$ is of finite ${\tt ADE}$ type, but we do not need to assume that in this section. 
We maintain our assumption that \(\kk\) is a commutative Noetherian ring. If $i,j\in \Ga_0$
are such that \(\{i,j\} \in \Gamma_1\), we say that $i$ and $j$ are {\em neighbors}. 

\subsection{Huerfano-Khovanov zigzag algebras}
The {\em  zigzag algebra \(\textup{\(\Zig\)}:= \textup{\(\Zig\)}(\Gamma)\) of type \(\Gamma\)} is defined in \cite{HK} as follows:

\begin{Definition}\label{zigdef}
First assume that \(|\Gamma_0|>1\). Let \(\overline{\Gamma}\) be the quiver obtained by doubling all edges between connected vertices and then orienting the edges so that if \(i\) and \(j\) are neighboring vertices in \(\Gamma\), then there is an arrow \(\za^{i,j}\) from \(j\) to \(i\) and an arrow \(\za^{j,i}\) from \(i\) to \(j\). For example, \(\overline{{\tt A}}_{\ell}\) is the quiver
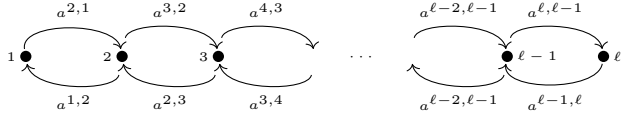
\begin{align*}
\begin{braid}\tikzset{baseline=3mm}
\coordinate (1) at (0,0);
\coordinate (2) at (4,0);
\coordinate (3) at (8,0);
\coordinate (4) at (12,0);
\coordinate (6) at (16,0);
\coordinate (L1) at (20,0);
\coordinate (L) at (24,0);
\draw [thin, black,->,shorten <= 0.1cm, shorten >= 0.1cm]   (1) to[distance=1.5cm,out=100, in=100] (2);
\draw [thin,black,->,shorten <= 0.25cm, shorten >= 0.1cm]   (2) to[distance=1.5cm,out=-100, in=-80] (1);
\draw [thin,black,->,shorten <= 0.25cm, shorten >= 0.1cm]   (2) to[distance=1.5cm,out=80, in=100] (3);
\draw [thin,black,->,shorten <= 0.25cm, shorten >= 0.1cm]   (3) to[distance=1.5cm,out=-100, in=-80] (2);
\draw [thin,black,->,shorten <= 0.25cm, shorten >= 0.1cm]   (3) to[distance=1.5cm,out=80, in=100] (4);
\draw [thin,black,->,shorten <= 0.25cm, shorten >= 0.1cm]   (4) to[distance=1.5cm,out=-100, in=-80] (3);
\draw [thin,black,->,shorten <= 0.25cm, shorten >= 0.1cm]   (6) to[distance=1.5cm,out=80, in=100] (L1);
\draw [thin,black,->,shorten <= 0.25cm, shorten >= 0.1cm]   (L1) to[distance=1.5cm,out=-100, in=-80] (6);
\draw [thin,black,->,shorten <= 0.25cm, shorten >= 0.1cm]   (L1) to[distance=1.5cm,out=80, in=100] (L);
\draw [thin,black,->,shorten <= 0.1cm, shorten >= 0.1cm]   (L) to[distance=1.5cm,out=-100, in=-100] (L1);
\blackdot(0,0);
\blackdot(4,0);
\blackdot(8,0);
\blackdot(20,0);
\blackdot(24,0);
\draw(0,0) node[left]{$1$};
\draw(4,0) node[left]{$2$};
\draw(8,0) node[left]{$3$};
\draw(14,0) node {$\cdots$};
\draw(20,0) node[right]{$\ell-1$};
\draw(24,0) node[right]{$\ell$};
\draw(2,1.2) node[above]{$\za^{2,1}$};
\draw(6,1.2) node[above]{$\za^{3,2}$};
\draw(10,1.2) node[above]{$\za^{4,3}$};
\draw(18,1.2) node[above]{$\za^{\ell-2,\ell-1}$};
\draw(22,1.2) node[above]{$\za^{\ell,\ell-1}$};
\draw(2,-1.2) node[below]{$\za^{1,2}$};
\draw(6,-1.2) node[below]{$\za^{2,3}$};
\draw(10,-1.2) node[below]{$\za^{3,4}$};
\draw(18,-1.2) node[below]{$\za^{\ell-2,\ell-1}$};
\draw(22,-1.2) node[below]{$\za^{\ell-1,\ell}$};
\end{braid}
\end{align*}
Then \(\Zig(\Gamma)\) is the path algebra  \(\kk\overline{\Gamma}\), generated by length-0 paths \(\ze_i\) for $i\in\Ga_0$, and length-$1$ paths \(\za^{i,j}\), modulo the following relations:
\begin{enumerate}
\item All paths of length three or greater are zero.
\item All paths of length two that are not cycles are zero.
\item All length-two cycles based at the same vertex are equal.
\end{enumerate}
The algebra \(\Zig(\Gamma)\) is graded by path length. If \(|\Gamma_0| =1\), i.e. \(\Gamma= {\tt A}_1\), we merely decree that \(\Zig(\Gamma):= \kk[\zc]/(\zc^2)\), where \(\zc\) is in degree 2. So that we may consider this algebra among the wider family of zigzag algebras, we will write \(e_1:=1\).
\end{Definition}

For type \(\Gamma \neq {\tt A}_1\), for every vertex \(i\), let \(j\) be any neighbor of \(i\), and write \(\zc^i\) for the cycle \(\za^{i,j}\za^{j,i}\).  The relations in \(\Zig\) imply that \(\zc^i\) is independent of choice of \(j\). Define \(\zc:= \sum_{i \in \Gamma_0} \zc^i\). 
Note that $\zc_i=\zc e_i=e_i \zc_i$.  
The following results are easily verified:

\begin{Lemma}\label{zigfacts}\(\)
\begin{enumerate}
\item The zigzag algebra \(\textup{\(\Zig\)}(\Gamma)\) is free of finite rank over \(\kk\), with \(\kk\)-basis:
\begin{align*}
\{\za^{i,j} \mid  \{i,j\} \in \Gamma_1\} \cup \{\zc^m\ze_i \mid i \in \Gamma_0,\ m \in \{0,1\}\}.
\end{align*}
\item The graded dimension of \(\textup{\(\Zig\)}\) is 
$
\DIM \textup{\(\Zig\)}=|\Gamma_0|(1+q^2)+2 |\Gamma_1|q.
$
\item The center of \(\textup{\(\Zig\)}\) is the \(\kk\)-span of the elements \(\{1\} \cup \{\zc \ze_i \mid i \in \Gamma_0\}\).
\item There is a \(\kk\)-algebra isomorphism \(\nu:\textup{\(\Zig\)} \to \textup{\(\Zig\)}^\textup{op}\) such that
$\nu(\textup{\(\ze\)}_i )= \textup{\(\ze\)}_i$, $\nu(\textup{\(\za\)}^{i,j})= \textup{\(\za\)}^{j,i}$, 
for all \(i,j \in \Gamma_0\), and $\nu(c)=c$.
\item The linear function \(\tr: \textup{\(\Zig\)} \to \kk\) given on basis elements by
\begin{align*}
 \tr(\ze_i) = 0, \hspace{10mm} \tr(\za^{i,j}) = 0,  \hspace{10mm} \tr(\zc \ze_i) = 1,
\end{align*}
for all \(i,j \in \Gamma_0\), satisfies \(\tr(xy) = \tr(yx)\) for all \(x,y \in \textup{\(\Zig\)}\).
\item The bilinear form \(\langle \cdot, \cdot \rangle: \textup{\(\Zig\)} \otimes \textup{\(\Zig\)} \to \kk\) given by \(\langle x, y\rangle := \tr(xy)\) is nondegenerate, symmetric and associative.
\item The map \(\varphi: \textup{\(\Zig\)} \to \textup{\(\Zig\)}^*\) given by \(\varphi(a) = \langle a, - \rangle \) is a  \((\textup{\(\Zig\)},\textup{\(\Zig\)})\)-bimodule isomorphism of degree \(-2\), with 
$\varphi(\ze_i) = (\zc \ze_i)^*$, $\varphi(\za^{i,j}) = (\za^{j,i})^*$,  $\varphi(\zc \ze_i) = \ze_i^*$. 
\end{enumerate}
\end{Lemma}

Lemma \ref{zigfacts} implies that \(\Zig\) is a graded symmetric algebra. Following \S\ref{symsec}, we have a \((\Zig, \Zig)\)-bimodule homomorphism \(\Delta: \Zig \to \Zig \otimes \Zig\), with distinguished degree 2 element
\begin{align}\label{zigDel}
\Delta(1) = \sum_{i \in \Gamma_0} \Big(\ze_i \otimes \zc \ze_i + \zc\ze_i \otimes \ze_i + \sum_{j\ \text{with}\ \{i,j\} \in \Gamma_1}
\za^{j,i} \otimes \za^{i,j}\Big) \in \Zig \otimes \Zig.
\end{align}

\subsection{Affine zigzag algebras}\label{SSAffZig}

The major focus of this paper will be the affine zigzag algebra, constructed via the affinization process  presented in Definition \ref{AffDef} for $A=\Zig=\Zig(\Gamma)$.

\begin{Definition}\label{Defaffzig}
For \(n \in \ZZ_{>0}\), we refer to the affinization \(\Zig^\aff_n(\Gamma) := \mathcal{H}_n(\Zig(\Gamma))\) of the zigzag algebra \(\Zig(\Gamma)\) as the {\em affine zigzag algebra of rank \(n\) and type $\Gamma$}.
\end{Definition}


The algebra \(\Zig^{\otimes n}\) is generated by the elements
\begin{align*}
\ze_{\bi}& := \ze_{i_1} \otimes \ze_{i_2} \otimes \cdots \otimes \ze_{i_n} &\textup{ for  }\bi = (i_1, i_2, \ldots, i_n) \in \Gamma_0^n,\\
\za_{r}^{i,j} &:= 1 \otimes \cdots \otimes \za^{i,j} \otimes1 \otimes  \cdots \otimes 1 \;\;(r\textup{th slot})&\textup{ for } r \in [1,n], \ \{i,j\} \in \Gamma_1,\\
\zc_r &:= 1 \otimes \cdots \otimes \zc \otimes1 \otimes  \cdots \otimes 1 \;\;(r\textup{th slot})&\textup{ for }r \in [1,n],
\end{align*}
subject only to the relations
\begin{equation}
\label{E1}
\textstyle{\sum_{\bi \in \Gamma_0^n} e_\bi} = 1,
\hspace{10mm}
e_\bi e_\bj = \delta_{\bi, \bj} e_\bi,
\hspace{10mm}
 c_r e_\bi = e_\bi c_r,
 \end{equation}
 \begin{equation}
 \label{E2}
a_r^{i,j}a_t^{k,l} = a_t^{k,l}a_r^{i,j},
\hspace{10mm}
a_r^{i,j}c_t = c_ta_r^{i,j},
\hspace{10mm}
c_r c_t = c_t c_r
\hspace{10mm}
 (\textup{for \(t \neq r\)})\end{equation}
 \begin{equation}
 \label{E3}
 a_r^{i,j}e_{\bi} = \delta_{j,i_r}e_{i_1,\ldots, i_{r-1},i,i_{r+1},\ldots, i_n} a_r^{i,j},
 \hspace{10mm}
 e_{\bi} a_r^{i,j} = \delta_{i,i_r}a_r^{i,j}e_{i_1, \ldots, i_{r-1},j,i_{r+1},\ldots,i_n},\end{equation}
 \begin{equation}
 \label{E4}
a_r^{i,j}a_r^{k,l} e_{\bi}= \delta_{j,k}\delta_{i,l} \delta_{i_r,l}c_r e_{\bi},
\hspace{10mm}
c_r^2 = 0, 
\hspace{10mm}
c_r a_r^{i,j} = a_r^{i,j} c_r = 0
\end{equation}
for all admissible \(r,t \in [1,n]\), \(\bi,\bj \in \Gamma_0^n\), and \(i,j,k,l \in \Gamma_0\).

Taking into account Definitions \ref{AffDef} and \ref{zigdef} and the description of \(\Delta(1)\) in (\ref{zigDel}), we may provide a more direct presentation of \(\Zig^\aff_n(\Gamma)\):

\begin{Lemma}\label{zigpresdef} The 
algebra \(\textup{\(\Zig\)}^\aff_n(\Gamma)\) is the graded \(\kk\)-algebra generated by the elements
\begin{align*}
\{e_\bi \mid \bi \in \Gamma_0^n\}\ \cup\ 
\{c_r, z_r,a^{i,j}_r \mid r \in [1,n],\, i,j \in \Gamma_0 \textup{ with }\{i,j\} \in \Gamma_1\}\ \cup\ 
\{s_t \mid t \in [1,n-1]\},
\end{align*}
with \(\deg(e_\bi)=\deg(s_t)=0$, $\deg(c_r)=\deg(z_r)= 2$, $\deg(a^{i,j}_r)=1\), subject only to the relations (\ref{E1}), (\ref{E2}), (\ref{E3}), (\ref{E4}) together with 
\begin{equation*}
s_r e_\bi = e_{s_r \bi} s_r,
\hspace{10mm}
s_r a_t^{i,j} = a_{s_r t}^{i,j}s_r,
\hspace{10mm}
s_r c_t = c_{s_r t} s_r,\end{equation*}
\begin{equation*}
z_r z_t = z_tz_r,
\hspace{10mm}
z_r a_t^{i,j} = a_t^{i,j}z_r,
\hspace{10mm}
z_r c_t = c_t z_r, 
\hspace{10mm}
z_r e_\bi = e_\bi z_r,\end{equation*}\begin{equation*}
s_r s_t = s_t s_r \;\;\textup{(for \( |t-r|>1\))},
\hspace{10mm}
s_r^2 = 1,
\hspace{10mm}
s_rs_{r+1}s_r = s_{r+1}s_r s_{r+1},\end{equation*}\begin{equation*}
(s_r z_t - z_{s_r t} s_r)e_\bi=
\begin{cases}
(\delta_{r,t}- \delta_{r+1,t})(c_r + c_{r+1})e_\bi
&
i_r = i_{r+1};\\
(\delta_{r,t}- \delta_{r+1,t})a_r^{i_{r+1},i_r} a_{r+1}^{i_r,i_{r+1}} e_\bi
&
\{i_r, i_{r+1}\} \in \Gamma_1;\\
0
& 
\textup{otherwise},
\end{cases}
\end{equation*}
for all admissible \(r,t \in [1,n]\), \(\bi \in \Gamma_0^n\), and \(i,j \in \Gamma_0\).
\end{Lemma}

We finish this subsection with three properties of affine zigzag algebras which follows easily from the general theory of affinization developed in section~\ref{SAff}.

\begin{Lemma}\label{affzigdim}
The affine zigzag algebra \(\textup{\(\Zig\)}^\aff_n(\Gamma)\) is free as a \(\kk\)-module, with graded dimension 
\begin{align*}
\dim_q \textup{\(\Zig\)}^\aff_n(\Gamma) =n!\left(\frac{(1+q^2)|\Gamma_0| + 2q|\Gamma_1| }{1-q^2}\right)^n.
\end{align*}
\end{Lemma}
\begin{proof}
This follows from Theorem \ref{AffBasis} and Lemma \ref{zigfacts}(ii).
\end{proof}

\begin{Lemma}\label{affzigop}
There is an isomorphism of graded \(\kk\)-algebras \(\widehat{\nu}:\textup{\(\Zig\)}^\aff_n(\Gamma) \to \textup{\(\Zig\)}^\aff_n(\Gamma)^\op\), given on generators by
$
\widehat{\nu}(z_t) = z_t$,  $
\widehat{\nu}(\ze_\bi) = \ze_\bi$, $\widehat{\nu}(\za^{i,j}_t) = \za^{j,i}_t$, $\widehat{\nu}(s_u) = s_u$. 
\end{Lemma}
\begin{proof}
This follows from Lemmas \ref{affop} and \ref{zigfacts}(iv).
\end{proof}

\begin{Lemma}
If the ground ring $\kk$ is indecomposable, so is the affine zigzag algebra \(\textup{\(\Zig\)}^\aff_n(\Gamma)\).
\end{Lemma}
\begin{proof}
Note that \(\Zig^\aff_n(\Gamma)\) is non-negatively graded, and by Proposition \ref{AffCenter} and Lemma \ref{zigfacts}(iii), the center of \(\Zig^\aff_n(\Gamma)\) has rank one in degree zero. Thus the only primitive central idempotent in \(\Zig^\aff_n(\Gamma)\) is \(1\), so the result follows.
\end{proof}

\subsection{Diagrammatics for the affine zigzag algebra}\label{zigdiagrammatics}
We provide a diagrammatic description of the algebra \(\Zig^\aff_n(\Gamma)\), which renders the relations described in Lemma \ref{zigpresdef} with more clarity.
We depict the (idempotented) generators as the following diagrams:
\begin{align*}
\ze_{\bi} = 
\begin{braid}\tikzset{baseline=0mm}
\draw(1,1) node[above]{$i_1$}--(1,-1);
\draw(2,1) node[above]{$i_2$}--(2,-1);
\draw(3,1) node[above]{$\cdots$};
\draw(4,1) node[above]{$i_n$}--(4,-1);
\end{braid}
\hspace{15mm}
z_{r}\ze_{\bi} = 
\begin{braid}\tikzset{baseline=0mm}
\draw(1,1) node[above]{$i_1$}--(1,-1);
\draw(2,1) node[above]{$\cdots$};
\draw(3,1) node[above]{$i_r$}--(3,-1);
\draw(4,1) node[above]{$\cdots$};
\draw(5,1) node[above]{$i_n$}--(5,-1);
\blackdot(3,0);
\end{braid}
\hspace{15mm}
\zc_{r}\ze_{\bi} = 
\begin{braid}\tikzset{baseline=0mm}
\draw(1,1) node[above]{$i_1$}--(1,-1);
\draw(2,1) node[above]{$\cdots$};
\draw(3,1) node[above]{$i_r$}--(3,-1);
\draw(4,1) node[above]{$\cdots$};
\draw(5,1) node[above]{$i_n$}--(5,-1);
\draw(2.8,0.2)--(3.2,-0.2);
\draw(3.2,0.2)--(2.8,-0.2);
\end{braid}
\end{align*}
\begin{align*}s_r\ze_{\bi} = 
\begin{braid}\tikzset{baseline=0mm}
\draw(1,1) node[above]{$i_1$}--(1,-1);
\draw(2,1) node[above]{$\cdots$};
\draw(3,1) node[above]{$i_r$}--(4.5,-1);
\draw(4.5,1) node[above]{$i_{r+1}$}--(3,-1);
\draw(5.5,1) node[above]{$\cdots$};
\draw(6.5,1) node[above]{$i_n$}--(6.5,-1);
\draw(1,-1) node[below]{$i_1$};
\draw(2,-1) node[below]{$\cdots$};
\draw(3,-1) node[below]{$i_{r+1}$};
\draw(4.5,-1) node[below]{$i_r$};
\draw(5.5,-1) node[below]{$\cdots$};
\draw(6.5,-1) node[below]{$i_n$};
\end{braid}
\hspace{15mm}
\za^{j_,i_r}_r\ze_{\bi} = 
\begin{braid}\tikzset{baseline=0mm}
\draw(1,1) node[above]{$i_1$}--(1,-1);
\draw(2,1) node[above]{$\cdots$};
\draw(3,1) node[above]{$i_r$}--(3,-1);
\draw(4,1) node[above]{$\cdots$};
\draw(5,1) node[above]{$i_m$}--(5,-1);
\draw(1,-1) node[below]{$i_1$};
\draw(2,-1) node[below]{$\cdots$};
\draw(3,-1) node[below]{$j$};
\draw(4,-1) node[below]{$\cdots$};
\draw(5,-1) node[below]{$i_n$};
\draw[red](3,0)--(3,-1);
\draw(2.7,0.3)--(3,0);
\draw(3.3,0.3)--(3,0);
\end{braid}
\textup{ for } \{i_r,j\} \in \Gamma_1.
\end{align*}
The red color is just intended to highlight that the label for the \(r\)th strand has changed. Then \(\Zig^{\textup{aff}}_n(\Gamma)\) is spanned by planar diagrams that look locally like these generators, equivalent up to the usual isotopies (cf. \cite{KL1}). In particular, dots, arrows, and \(\zx\)'s can be freely moved along strands, provided they don't pass through crossings. Multiplication of diagrams is given by stacking vertically, and products are zero unless labels for strands match. 

Then the defining local relations can be drawn as follows:

\begin{enumerate}
\item \(\Zig^{\otimes n}\) relations:
\begin{align*}
\begin{braid}\tikzset{baseline=0mm}
\draw(0,1) node[above]{$i$}--(0,-1);
\draw[red](0,0.5)--(0,-0.5);
\draw[green](0,-0.5)--(0,-1);
\draw(-0.3,0.8)--(0,0.5);
\draw(0.3,0.8)--(0,0.5);
\draw(-0.3,-0.2)--(0,-0.5);
\draw(0.3,-0.2)--(0,-0.5);
\draw(0,-1) node[below]{$k$};
\draw(0,0.2) node[left]{$j$};
\end{braid}
=0 \;\;(i,j,k \textup{ distinct)}
\hspace{10mm}
\begin{braid}\tikzset{baseline=0mm}
\draw(0,1) node[above]{$i$}--(0,-1);
\draw[red](0,0.5)--(0,-0.5);
\draw(-0.3,0.8)--(0,0.5);
\draw(0.3,0.8)--(0,0.5);
\draw(-0.3,-0.2)--(0,-0.5);
\draw(0.3,-0.2)--(0,-0.5);
\draw(0,-1) node[below]{$i$};
\draw(0,0.2) node[left]{$j$};
\end{braid}
=
\begin{braid}\tikzset{baseline=0mm}
\draw(0,1) node[above]{$i$}--(0,-1);
\draw(-0.2,0.2)--(0.2,-0.2);
\draw(0.2,0.2)--(-0.2,-0.2);
\draw(0,-1) node[below]{$i$};
\end{braid}
\;\;(\forall j)
\hspace{10mm}
\begin{braid}\tikzset{baseline=0mm}
\draw(0,1) node[above]{$i$}--(0,-1);
\draw[red](0,-0.5)--(0,-1);
\draw(-0.2,0.3)--(0.2,0.7);
\draw(0.2,0.3)--(-0.2,0.7);
\draw(-0.3,-0.2)--(0,-0.5);
\draw(0.3,-0.2)--(0,-0.5);
\draw(0,-1) node[below]{$j$};
\end{braid}
=
\begin{braid}\tikzset{baseline=0mm}
\draw(0,1) node[above]{$i$}--(0,-1);
\draw[red](0,0.5)--(0,-1);
\draw(-0.3,0.8)--(0,0.5);
\draw(0.3,0.8)--(0,0.5);
\draw(0,-1) node[below]{$j$};
\draw(-0.2,-0.3)--(0.2,-0.7);
\draw(0.2,-0.3)--(-0.2,-0.7);
\end{braid}
=
\begin{braid}\tikzset{baseline=0mm}
\draw(0,1) node[above]{$i$}--(0,-1);
\draw(-0.2,0.7)--(0,0.5);
\draw(0.2,0.7)--(0,0.5);
\draw(0,-1) node[below]{$i$};
\draw(-0.2,-0.3)--(0.2,-0.7);
\draw(0.2,-0.3)--(-0.2,-0.7);
\draw(-0.2,0.3)--(0.2,0.7);
\draw(0.2,0.3)--(-0.2,0.7);
\end{braid}
=0
\end{align*}

\item \(\kk \mathfrak{S}_n\) relations:
\begin{align*}
\begin{braid}\tikzset{baseline=0mm}
\draw(-0.5,1) node[above]{$i$};
\draw(0.5,1) node[above]{$j$};
\draw(0.5,1)--(-0.5,0)--(0.5,-1);
\draw(-0.5,1)--(0.5,0)--(-0.5,-1);
\end{braid}
=
\begin{braid}\tikzset{baseline=0mm}
\draw(-0.5,1) node[above]{$i$};
\draw(0.5,1) node[above]{$j$};
\draw(0.5,1)--(0.5,-1);
\draw(-0.5,1)--(-0.5,-1);
\end{braid}
\hspace{10mm}
\begin{braid}\tikzset{baseline=0mm}
\draw(-1,1) node[above]{$i$};
\draw(0,1) node[above]{$j$};
\draw(1,1) node[above]{$k$};
\draw(1,1)--(-1,-1);
\draw(0,1)--(1,0)--(0,-1);
\draw(-1,1)--(1,-1);
\end{braid}
=
\begin{braid}\tikzset{baseline=0mm}
\draw(-1,1) node[above]{$i$};
\draw(0,1) node[above]{$j$};
\draw(1,1) node[above]{$k$};
\draw(1,1)--(-1,-1);
\draw(0,1)--(-1,0)--(0,-1);
\draw(-1,1)--(1,-1);
\end{braid}
\hspace{10mm}
(\forall i,j,k)
\end{align*}

\item \((\Zig^{\otimes n}, \kk[z_1, \ldots, z_n])\) commutation relations:
\begin{align*}
\begin{braid}\tikzset{baseline=0mm}
\draw(0,1) node[above]{$i$}--(0,-1);
\draw[red](0,0.5)--(0,-1);
\draw(-0.3,0.8)--(0,0.5);
\draw(0.3,0.8)--(0,0.5);
\blackdot(0,-0.5);
\draw(0,-1) node[below]{$j$};
\end{braid}
=
\begin{braid}\tikzset{baseline=0mm}
\draw(0,1) node[above]{$i$}--(0,-1);
\draw[red](0,-0.5)--(0,-1);
\blackdot(0,0.5);
\draw(0,-1) node[below]{$j$};
\draw(-0.3,-0.2)--(0,-0.5);
\draw(0.3,-0.2)--(0,-0.5);
\end{braid}
\hspace{10mm}
\begin{braid}\tikzset{baseline=0mm}
\draw(0,1) node[above]{$i$}--(0,-1);
\draw(-0.2,0.3)--(0.2,0.7);
\draw(0.2,0.3)--(-0.2,0.7);
\blackdot(0,-0.5);
\draw(0,-1) node[below]{$i$};
\end{braid}
=
\begin{braid}\tikzset{baseline=0mm}
\draw(0,1) node[above]{$i$}--(0,-1);
\blackdot(0,0.5);
\draw(0,-1) node[below]{$i$};
\draw(-0.2,-0.3)--(0.2,-0.7);
\draw(0.2,-0.3)--(-0.2,-0.7);
\end{braid}
\end{align*}

\item \((\kk \mathfrak{S}_n, \Zig^{\otimes n})\) commutation relations:
\begin{align*}
\begin{braid}\tikzset{baseline=0mm}
\draw(-1,1) node[above]{$i$};
\draw(1,1) node[above]{$j$};
\draw(1,1)--(-1,-1);
\draw(-1,1)--(1,-1);
\draw[red](-0.5,0.5)--(1,-1);
\draw(-0.5,0.9)--(-0.5,0.5);
\draw(-0.9,0.5)--(-0.5,0.5);
\draw(-1,-1) node[below]{$j$};
\draw(1,-1) node[below]{$k$};
\end{braid}
=
\begin{braid}\tikzset{baseline=0mm}
\draw(-1,1) node[above]{$i$};
\draw(1,1) node[above]{$j$};
\draw(1,1)--(-1,-1);
\draw(-1,1)--(1,-1);
\draw[red](0.5,-0.5)--(1,-1);
\draw(0.5,-0.1)--(0.5,-0.5);
\draw(0.1,-0.5)--(0.5,-0.5);
\draw(-1,-1) node[below]{$j$};
\draw(1,-1) node[below]{$k$};
\end{braid}
\hspace{10mm}
\begin{braid}\tikzset{baseline=0mm}
\draw(-1,1) node[above]{$i$};
\draw(1,1) node[above]{$j$};
\draw(1,1)--(-1,-1);
\draw(-1,1)--(1,-1);
\draw[red](-0.5,-0.5)--(-1,-1);
\draw(-0.5,-0.1)--(-0.5,-0.5);
\draw(-0.1,-0.5)--(-0.5,-0.5);
\draw(-1,-1) node[below]{$k$};
\draw(1,-1) node[below]{$i$};
\end{braid}
=
\begin{braid}\tikzset{baseline=0mm}
\draw(-1,1) node[above]{$i$};
\draw(1,1) node[above]{$j$};
\draw(1,1)--(-1,-1);
\draw(-1,1)--(1,-1);
\draw[red](0.5,0.5)--(-1,-1);
\draw(0.5,0.9)--(0.5,0.5);
\draw(0.9,0.5)--(0.5,0.5);
\draw(-1,-1) node[below]{$k$};
\draw(1,-1) node[below]{$i$};
\end{braid}
\hspace{10mm}
(\forall i,j,k)
\end{align*}
\begin{align*}
\begin{braid}\tikzset{baseline=0mm}
\draw(-1,1) node[above]{$i$};
\draw(1,1) node[above]{$j$};
\draw(1,1)--(-1,-1);
\draw(-1,1)--(1,-1);
\draw(-0.5,0.8)--(-0.5,0.2);
\draw(-0.8,0.5)--(-0.2,0.5);
\end{braid}
=
\begin{braid}\tikzset{baseline=0mm}
\draw(-1,1) node[above]{$i$};
\draw(1,1) node[above]{$j$};
\draw(1,1)--(-1,-1);
\draw(-1,1)--(1,-1);
\draw(0.5,-0.2)--(0.5,-0.8);
\draw(0.2,-0.5)--(0.8,-0.5);
\end{braid}
\hspace{10mm}
\begin{braid}\tikzset{baseline=0mm}
\draw(-1,1) node[above]{$i$};
\draw(1,1) node[above]{$j$};
\draw(1,1)--(-1,-1);
\draw(-1,1)--(1,-1);
\draw(-0.5,-0.8)--(-0.5,-0.2);
\draw(-0.8,-0.5)--(-0.2,-0.5);
\end{braid}
=
\begin{braid}\tikzset{baseline=0mm}
\draw(-1,1) node[above]{$i$};
\draw(1,1) node[above]{$j$};
\draw(1,1)--(-1,-1);
\draw(-1,1)--(1,-1);
\draw(0.5,0.2)--(0.5,0.8);
\draw(0.2,0.5)--(0.8,0.5);
\end{braid}
\hspace{10mm}
(\forall i,j)
\end{align*}

\item \((\kk[z_1, \ldots, z_n],\kk\mathfrak{S}_n)\) commutation relations:

\begin{align*}
\begin{braid}\tikzset{baseline=0mm}
\draw(1,1) node[above]{$j$}--(-1,-1);
\draw(-1,1) node[above]{$i$}--(1,-1);
\blackdot(-0.5,0.5);
\end{braid}
-
\begin{braid}\tikzset{baseline=0mm}
\draw(1,1) node[above]{$j$}--(-1,-1);
\draw(-1,1) node[above]{$i$}--(1,-1);
\blackdot(0.5,-0.5);
\end{braid}
=
\begin{braid}\tikzset{baseline=0mm}
\draw(1,1) node[above]{$j$}--(-1,-1);
\draw(-1,1) node[above]{$i$}--(1,-1);
\blackdot(-0.5,-0.5);
\end{braid}
-
\begin{braid}\tikzset{baseline=0mm}
\draw(1,1) node[above]{$j$}--(-1,-1);
\draw(-1,1) node[above]{$i$}--(1,-1);
\blackdot(0.5,0.5);
\end{braid}
=
\begin{cases}
\begin{braid}\tikzset{baseline=0mm}
\draw(1,1) node[above]{$i$}--(1,-1);
\draw(-1,1) node[above]{$i$}--(-1,-1);
\draw(-0.8,0.2)--(-1.2,-0.2);
\draw(-1.2,0.2)--(-0.8,-0.2);
\end{braid}
+
\begin{braid}\tikzset{baseline=0mm}
\draw(1,1) node[above]{$i$}--(1,-1);
\draw(-1,1) node[above]{$i$}--(-1,-1);
\draw(0.8,0.2)--(1.2,-0.2);
\draw(1.2,0.2)--(0.8,-0.2);
\end{braid}
&
\textup{if }i=j;\\
\\
\begin{braid}\tikzset{baseline=0mm}
\draw(1,1) node[above]{$j$}--(1,-1);
\draw(-1,1) node[above]{$i$}--(-1,-1);
\draw[red](-1,0)--(-1,-1);
\draw[red](1,1)--(1,0);
\draw(-0.7,0.3)--(-1,0);
\draw(-1.3,0.3)--(-1,0);
\draw(0.7,0.3)--(1,0);
\draw(1.3,0.3)--(1,0);
\draw(-1,-1) node[below]{$j$};
\draw(1,-1) node[below]{$i$};
\end{braid}
&
\textup{if }\{i,j\} \in \Gamma_1;\\
\\
0
&
\textup{otherwise}.
\end{cases}
\end{align*}
\end{enumerate}

\begin{Remark}\label{affziglit}
In their work on the categorification of the Heisenberg algebra \(\mathfrak{h}_\Gamma\) for \(\Gamma\) of affine ${\tt ADE}$ type, Cautis and Licata \cite{CL} introduce a certain 2-category \(\mathcal{H}^\Gamma\). The 1-morphisms in this category are generated by objects \(P_i\) and \(Q_i\), for each \(i \in \Gamma_0\). Comparing the local relations between 2-morphisms \cite[\S6.1, \S10.3]{CL} with the diagrammatic above, it can be seen that \(\End_{\mathcal{H}^\Gamma}\left(P^n\right)\) satisfies the defining relations of \(\Zig^\aff_n(\Gamma)\), up to some signs. More generally, Rosso and Savage \cite{RS} introduce a monoidal category \(\mathcal{H}_B\) associated to any Frobenius superalgebra \(B\), recover the above category as a special case, and study in particular the endomorphism algebra of \(P^n\) \cite[\S8.4]{RS}. 
\end{Remark}

\section{The minuscule imaginary stratum category}\label{SStratum}
For the remainder of the paper we assume  \(\preceq\) is a {\it balanced order} on $\Phi_+$, and denote $d:=\height(\de)$, see \S\ref{SSARS}. We also assume that the graph $\Gamma$ is the Dynkin diagram corresponding to the finite type Cartan matrix $\Car'$, and write $\Zig$ for $\Zig(\Gamma)$, $\Zig_n^\aff$ for $\Zig_n^\aff(\Gamma)$. 
We do not assume that $\kk$ is a field unless otherwise stated. 

\subsection{Irreducible semicuspidal modules}
\label{SSBOSW}

Recall from \S\ref{SSSemiCusp} that, when $\kk$ is a field, the irreducible semicuspidal \(R_\delta\)-modules may be  canonically labeled \(L_{\delta,i}\), for \(i \in I'\). The following theorem, proved in  \cite[Lemma 5.1, Corollary 5.3]{Kcusp}, gives a characterization of the these important modules via their words.

\begin{Lemma}\label{Kcuspminuscule} Let $\kk$ be a field. 
For each \(i \in I'\), \(L_{\delta,i}\) can be characterized up to isomorphism and grading shift as the unique irreducible \(R_\delta\)-module such that \(i_1 = 0\) and \(i_d=i\) for all words \(\bi\) of \(L_{\delta,i}\).
\end{Lemma}

In this section we will use this lemma to recognize the irreducible semicuspidal \(R_\delta\)-modules as certain homogeneous modules which are concentrated in degree zero.

Following \cite{KRhomog}, for \(1 \leq r < d\) and \(\bi \in I^\delta\), we say \(s_r \in \mathfrak{S}_d\) is  \(\bi\)-{\em admissible} if \({\tt c}_{i_r,i_{r+1}} = 0\). More generally, if \(s_{r_1} \cdots s_{r_t}\) is a reduced expression for \(w \in \mathfrak{S}_d\) and each \(s_{r_k}\) is \((s_{r_{k+1}} \cdots s_{r_t}\bi)\)-admissible, then we say \(w\) is \(\bi\)-admissible. This property is independent of reduced expression for \(w\). In addition, admissibility is preserved by products in the sense that if \(w\) is \(\bi\)-admissible and \(w'\) is \((w\bi)\)-admissible, then \(w'w\) is \(\bi\)-admissible. The {\it connected component} of \(\bi\) is 
$$\textup{Con}(\bi):=\{w\bi \mid \bi\textup{-admissible }w \in \mathfrak{S}_d\}.
$$ 
Clearly \(\textup{Con}(\bi) = \textup{Con}(\bj)\) if and only if \(\bi \in \textup{Con}(\bj)\). We say that \(\bi\) is {\em homogeneous} provided that  \(i_r = i_s\) for some \(r <s\) implies there exist \(t,u\) with \(r<t<u<s\) such that \({\tt c}_{i_r,i_t}={\tt c}_{i_r,i_u}=-1\). 

\begin{Lemma}\label{homogconst}
If \(\theta \in Q_+\) and \(\bi \in I^\theta\) is a homogeneous word,  then there exists an \(R_\theta\)-module \(M\) with character \(\sum_{\bj \in \textup{Con}(\bi)} \bj\). If \(\kk\) is a field, this module is irreducible. 
\end{Lemma}
\begin{proof}
If \(\kk\) is a field, this is \cite[Theorem 3.4]{KRhomog}. The part of the proof verifying the relations of $R_\theta$ on the $\k$-module $\bigoplus_{\bj \in \textup{Con}(\bi)} \k\cdot v_\bj$ works for an arbitrary commutative ground ring $\kk$.  
\end{proof}

With the intention of applying this lemma, we associate to each \(i \in I'\) a special homogeneous word \(\bb^i \in I^\delta\). 

\begin{flushleft}
\begin{tabular}{lll}
\(\textup{Type }{\tt A}_\ell^{(1)}:\)&\(\bb^i := 012\cdots(i-1)\ell(\ell-1)(\ell-2) \cdots (i+1)i\)
\\
\end{tabular}
\begin{tabular}{lll}
\(\textup{Type }{\tt D}_\ell^{(1)}:\)& \(\bb^i :=\begin{cases}
 0234\cdots\ell(\ell-2)(\ell-3) \cdots (i+1)123\cdots i & \textup{if } 1 \leq i \leq \ell-2;\\
  0234\cdots(\ell-2)\ell123\cdots(\ell-1)& \textup{if } i=\ell-1;\\
   0234\cdots(\ell-1)123\cdots(\ell-2)\ell& \textup{if } i=\ell
 \end{cases}\)
 \\
\end{tabular}
\begin{tabular}{lll}
\(\textup{Type }{\tt E}_6^{(1)}:\) &\( \bb^i :=\begin{cases}
024354 265431 & \textup{if }i=1;\\
024354 136542 & \textup{if }i=2;\\
024354 126543 & \textup{if }i=3;\\
024354 123654 & \textup{if }i=4;\\
024354 123465 & \textup{if }i=5;\\
024354 123456 & \textup{if }i=6
\end{cases}\)
\\
\end{tabular}

\begin{tabular}{lll}
\(\textup{Type }{\tt E}_7^{(1)}:\) & \(\bb^i :=\begin{cases}
01342546354 2765431 & \textup{if }i=1;\\
01342546354 1376542 & \textup{if }i=2;\\
01342546354 1276543 & \textup{if }i=3;\\
01342546354 1237654 & \textup{if }i=4;\\
01342546354 1234765 & \textup{if }i=5;\\
01342546354 1234576 & \textup{if }i=6;\\
01342546354 1234567 & \textup{if }i=7
\end{cases}\)
\\
\end{tabular}

\begin{tabular}{lll}
\(\textup{Type }{\tt E}_8^{(1)}:\) &\(\bb^i :=\begin{cases}
0876542314356425764354 28765431 & \textup{if }i=1;\\
0876542314356425764354 13876542 & \textup{if }i=2;\\
0876542314356425764354 12876543 & \textup{if }i=3;\\
0876542314356425764354 12387654 & \textup{if }i=4;\\
0876542314356425764354 12348765 & \textup{if }i=5;\\
0876542314356425764354 12345876 & \textup{if }i=6;\\
0876542314356425764354 12345687 & \textup{if }i=7;\\
0876542314356425764354 12345678 & \textup{if }i=8.
\end{cases}\)
\\
\end{tabular}
\end{flushleft}
We will write \(G^i:=\textup{Con}(\bb^i)\) and 
$$G^\delta:= \bigcup_{i \in I'}G^i.$$ We need one more combinatorial notion for words:

\begin{Definition}
Let \(\bi \in I^\delta\). For \(t \in \{1,\ldots, d\}\), define the \(t\)-{\it neighbor sequence} of \(\bi\) to be \(\nbr_t(\bi):=(n_1, \ldots, n_t) \in \{0,N,S\}^t\), where
\begin{align*}
n_r = \begin{cases}
S, & \textup{if }i_r = i_t;\\
N, & \textup{if }{\tt c}_{i_r,i_t}<0;\\
0, &\textup{otherwise.}
\end{cases}
\end{align*}
Then \(\rednbr_t(\bi)\), the {\it reduced \(t\)-neighbor sequence of \(\bi\)}, is achieved by deleting all \(0\)'s from \(\nbr_t(\bi)\).
\end{Definition}

\begin{Example} Take \({\tt C} = {\tt A}_{7}^{(1)}\). Then \(\bi = 01726354 \in G^4\), \(\nbr_6(\bi) = 000N0S\), and \(\rednbr_6(\bi)=NS\).
\end{Example}

The following lemma is clear:

\begin{Lemma}\label{Admis}
If \(s_r\) is \(\bi\)-admissible, then \(\rednbr_{s_r(t)}(s_r\bi)= \rednbr_{t}(\bi)\).
\end{Lemma}

Now we prove numerous useful facts about the special words \(\bb^i\):

\begin{Lemma}\label{wordfacts}
Let \(i,j \in I'\) such that \({\tt c}_{i,j}=-1\).
\begin{enumerate}
\item If \(\bi \in G^\delta\), then \(\bi\) is homogeneous.
\item For all \(\bi \in G^i\), we have \(i_1=0\), \(i_d=i\), \(i_1\) is a neighbor of \(i_2\), and \(i_{d-1}\) is a neighbor of \(i_d\).
\item If \({\tt C} \neq {\tt A}^{(1)}_1\) and \(\bi \in G^\delta\), then
\begin{align*}
\rednbr_t(\bi) = \begin{cases}
(NSN)^aNS, & \textup{if } 1<t<d;\\
(NSN)^aNNS, & \textup{if } t=d,
\end{cases}
\end{align*}
for some \(a \geq 0\).
\item If \(\bi \in G^\delta\) and  \(r<d-1\), then \(s_r\bi \in G^\delta\) if and only if \(s_r\) is \(\bi\)-admissible.
\item For any \(\bi, \bi' \in G^i\), there exists a unique \(w_{\bi',\bi} \in \mathfrak{S}_d\) such that \(w_{\bi',\bi}\bi=\bi'\) and \(w_{\bi',\bi}\) is \(\bi\)-admissible.
\item There exists a unique \(w_{i,j} \in \mathfrak{S}_d\) such that \(w_{i,j} \bb^j = \bb^i\), and \(w_{i,j} = w_1 s_{d-1} w_2\), where \(w_2\) is \(\bb^j\)-admissible and \(w_1\) is \(s_{d-1} w_2 \bb^j\)-admissible.
\item For any \(\bi \in G^i\) and \(\bj \in G^j\) such that \({\tt c}_{i,j}=-1\), there exists a unique \(w_{\bi,\bj} \in \mathfrak{S}_d\) such that \(w_{\bi,\bj}\bj = \bi\) and \(w_{\bi,\bj}=w_1s_{d-1}w_2\), where \(w_2\) is \(\bi\)-admissible and \(w_1\) is \(s_{d-1}w_2\bi\)-admissible.
\item If \({\tt C} \neq {\tt A}^{(1)}_1\) and \(\bi \in G^\delta\), then \(s_{d-1}\bi \in G^{i_{d-1}}\).
\end{enumerate}
\end{Lemma}

\begin{proof}
(i) It is straightforward to check that \(\bb^i\) satisfies the homogeneity condition. Thus by \cite[Lemma 3.3]{KRhomog}, every \(\bi \in G^i\) satisfies this condition. 

(ii) If \(1<r<d\), then \((\bb^i)_r\) has a neighbor somewhere to the left and right in \(\bb^i\), so no \(\bb^i\)-admissible element \(w\) may send \(r\) to \(1\) or \(d\), so \(i_1=(\bb^i)_1 = 0\), and \(i_d = (\bb^i)_d=i\) for every \(\bi \in G^j\). Moreover it cannot be that \(i_{d-1}=i_d\) by (i), and if it were the case that \({\tt c}_{i_d,i_{d-1}}=0\), then we would have \(s_{d-1}\bi \in G^i\), but \((s_{d-1}\bi)_d \neq i\), a contradiction. Thus \(i_{d-1}\) and \(i_d\) are neighbors, and a similar argument proves the same for \(i_1\) and \(i_2\). 

(iii) We have by part (ii) that \(s_1\) and \(s_{d-1}\) are never admissible transpositions for \(\bi \in G^\delta\). Therefore, by Lemma \ref{Admis}, it is enough to check that that statement (iii) holds for the special words \(\bb^i\), which may be readily done. 

(iv) The statement holds for \(r=1\) by part (ii), since \(s_1\) is never \(\bi\)-admissible, and \(i_1=0\) for every \(\bi \in G^\delta\). Let \(1<r<d-1\). If \(s_r\) is not \(\bi\)-admissible, then \({\tt c}_{i_r,i_{r+1}}=-1\) by (i). By part (iii), \(\rednbr_{r+1}(\bi) = (NSN)^aNS\) for some \(a \geq 0\). Then \(\rednbr_r(s_r\bi)=(NSN)^aS\). But then again by part (iii), \(s_r\bi \notin G^\delta\). 

(v) The existence of \(w_{\bi',\bi}\) is guaranteed by the definition of \(G^i = \textup{Con}(\bi)\), and we note that \(\bi\)-admissible elements are in bijection with \(G^i\) since \(\bi\)-admissible elements cannot transpose similar letters. This proves uniqueness.

(vi) In types \({\tt A}^{(1)}_\ell\) and \({\tt E}^{(1)}_\ell\), it is straightforward to verify that if \(j<i \in I'\), then we have \(s_{d-1}\bb^j \in G^i\). Thus there exists \(s_{d-1}\bb^j\)-admissible \(u \in \mathfrak{S}_d\) such that \(u s_{d-1} \bb^j = \bb^i\), so taking \(w_{i,j}:= u s_{d-1}\) satisfies the claim. On the other hand, if \(j>i\), then \(w_{i,j}:=w_{j,i}^{-1}\) also satisfies the claim.
In type \({\tt D}^{(1)}_\ell\), if \(i<j \in I'\), we have \(s_{d-1}\bb^j \in G^i\). Thus there exists \(s_{d-1}\bb^j\)-admissible \(u \in \mathfrak{S}_d\) such that \(u s_{d-1} \bb^j = \bb^i\), so taking \(w_{i,j}:= u s_{d-1}\) satisfies the claim. On the other hand, if \(j<i\), then \(w_{i,j}:=w_{j,i}^{-1}\) must also satisfy the claim.
Uniqueness follows as in the proof of (v), from consideration of the fact that no similar letters are transposed in this product.

(vii) We may take \(w_{\bi,\bj}=w_{\bi,\bb^i}w_{i,j}w_{\bb^j,\bj}\) to show existence. Uniqueness follows as in the proof of (v).

(viii) Let \(\bi \in G^i\). Then \(j:=i_{d-1}\) is a neighbor of \(i=i_d\) by part (ii). By part (vii), \(\bi = w_{\bi, \bb^j} \bb^j = w_1 s_{d-1} w_2 \bb^j\), where \(w_2\) is \(\bb^j\)-admissible and \(w_1\) is \(s_{d-1}w_2 \bb^j\)-admissible. But \((s_{d-1}w_2 \bb_j)_{d-1} = j\) and \(i_{d-1} = j\), so it follows from admissibility of \(w_1\) that \(w_1\) fixes the \((d-1)\)th and \(d\)th positions. Thus we have that \(w_1 s_{d-1} = s_{d-1} w_1\) and \(w_1\) is \(w_2\bb^j\)-admissible. So \(s_{d-1}\bi = w_1 w_2 \bb^j \in G^j\).
\end{proof}

\begin{Lemma}\label{Ldel}
Let \(\kk\) be a field. For each \(i \in I'\), \(\CH L_{\delta,i}= \sum_{\bi \in G^i}\bi\).
\end{Lemma}
\begin{proof}
By Lemmas  \ref{homogconst} and \ref{wordfacts}(i), there exists a homogeneous irreducible \(R_\delta\)-module with character \(\sum_{\bi \in G^i}\bi\). By Lemmas \ref{Kcuspminuscule} and \ref{wordfacts}(ii), this module must be \(L_{\delta,i}\).
\end{proof}

\begin{Corollary} \label{CSCW} 
We have that \(G^\delta\) is a complete set of semicuspidal words in \(I^\delta\), and so \(C_\delta = R_\delta/R_\delta 1_{\noncusp} R_\delta\), where \(1_{\noncusp}=\sum_{\bi \in I^\delta \backslash G^\delta} 1_{\bi}\). 
\end{Corollary}


\subsection{\boldmath A spanning set for $C_\delta$}\label{spansec}

For each \(w \in \mathfrak{S}_d\), we choose a distinguished reduced expression \(w=s_{r_1} \cdots s_{r_t}\). Based on this set of choices, we define, for every \(w \in \mathfrak{S}_n\), an element \(\psi_w = \psi_{r_1} \cdots \psi_{r_t} \in R_\delta\). We warn the reader that \(\psi_w\) is dependent on the choice of distinguished reduced expression for \(w\), as \(\psi_r\)'s do not in general satisfy braid relations in \(R_\delta\), see (\ref{KLRbraid}). We will see however, that the images of the elements $\psi_w$ in \(C_\delta\) are well defined. 

Recalling the elements of \(\mathfrak{S}_d\) defined in Lemma \ref{wordfacts}(v)--(vii), we will write  \(\psi_{\bi,\bj}\) (resp. \(\psi_{i,j}\)) for \(\psi_{w_{\bi,\bj}}\) (resp. \(\psi_{w_{i,j}}\)).

\begin{Lemma}\label{Cdelfacts} The algebra \(C_\delta\) is non-negatively graded. Moreover, in $C_\delta$ we have:
\begin{enumerate}
\item The elements \(\psi_w\) are independent of reduced expression for \(w\) for all $w\in\Si_d$.
\item \(\psi_r y_t= y_{s_r(t)}\psi_r \), for all admissible \(r,t\).
\end{enumerate}
\end{Lemma}
\begin{proof}
All of these follow from Corollary~\ref{CSCW} and Lemma \ref{wordfacts}(i). 
We have \(1_{\bi}=0\) in \(C_\delta\) if $i_r=i_{r+1}$ for some $1\leq r<d$. So there are no generators \(\psi_r1_{\bj}\) in negative degrees, hence (i). Part (iii) also follows from that observation, together with relation (\ref{KLRypsi}). Finally, semicuspidal words have no subwords of the form \(iji\), so, by relation (\ref{KLRbraid}), the images of \(\psi\)'s satisfy braid relations in \(C_\de\), hence (ii). 
\end{proof}

\begin{Lemma}\label{Cy}
The following facts hold in \(C_\delta\):
\begin{enumerate}
\item \(y_1 = \cdots = y_{d-1}\).
\item \((y_1-y_d)^2 = 0\).
\item \(y_1 \in Z(C_\delta)\).
 \end{enumerate}
\end{Lemma}
\begin{proof}
Assume first that \({\tt C} = {\tt A}^{(1)}_1\). Then \(d=2\), and so claim (i) is trivial. We have \(G^\delta = \{01\}\) and \(1_{10} =0\) in \(C_\delta\), hence
$ 
0 = \psi_1 1_{10} \psi_1 = \psi_1^21_{01}  = \pm(y_1-y_2)^21_{01} = \pm(y_1-y_2)^2,
$ 
proving claim (ii). For (iii), it follows from KLR relations that \(y_1\) commutes with every generator of \(R_\delta\) except \(\psi_1\). However, in $C_\de$ we have \(\psi_1 y_1=\psi_1 y_1 1_{01} = 1_{10} \psi_1y_1=0\), and similarly $y_1\psi_1 =0$.

Now let \({\tt C} \neq {\tt A}^{(1)}_1\). We will use the diagrammatic presentation for \(R_\delta\), see \S\ref{KLRdiag}. We prove (i) first. Let \(\bi=0i_2i_3 \cdots i_d \in G^\delta\). Let \(1 < r < d\). The following diagram is  zero in \(C_\delta\) since, by Corollary~\ref{CSCW} and Lemma \ref{wordfacts}(ii), all semicuspidal words start with 0, and \(i_r \neq 0\):
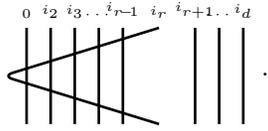
\begin{align*}
\begin{braid}\tikzset{baseline=0mm}
\draw(1,2) node[above]{$0$}--(1,-2);
\draw(2,2) node[above]{$i_2$}--(2,-2);
\draw(3,2) node[above]{$i_3$}--(3,-2);
\draw(4,2) node[above]{$\cdots$}--(4,-2);
\draw(5,2) node[above]{$i_{r\hspace{-0.5mm}-\hspace{-0.5mm}1}$}--(5,-2);
\draw(6.5,2) node[above]{$i_r$}--(0,0)--(6.5,-2);
\draw(8,2) node[above]{$i_{r+1}$}--(8,-2);
\draw(9,2) node[above]{$\cdots$}--(9,-2);
\draw(10,2) node[above]{$i_{d}$}--(10,-2);
\end{braid}.
\end{align*}
We will simplify this diagram using relations. Note that we may ignore strands to the right of \(i_r\) and strands whose colors do not neighbor \(i_r\). Omitting such strands, and recalling from Lemma \ref{wordfacts}(iii) that \(\rednbr_k(\bi) = (NSN)^aNS\) for some \(a \geq0\), we have, using the relations in $R_\de$:
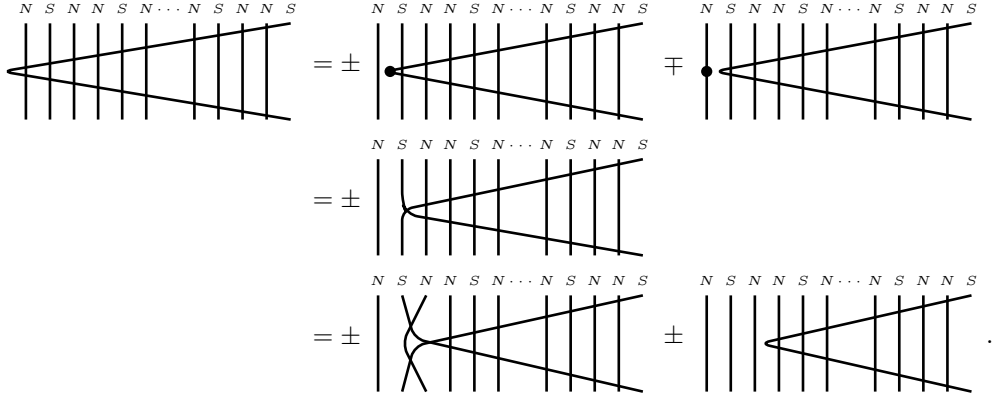
\begin{align*}
\begin{braid}\tikzset{baseline=0mm}
\draw(1,2) node[above]{$N$}--(1,-2);
\draw(2,2) node[above]{$S$}--(2,-2);
\draw(3,2) node[above]{$N$}--(3,-2);
\draw(4,2) node[above]{$N$}--(4,-2);
\draw(5,2) node[above]{$S$}--(5,-2);
\draw(6,2) node[above]{$N$}--(6,-2);
\draw(7,2) node[above]{$\cdots$};
\draw(8,2) node[above]{$N$}--(8,-2);
\draw(9,2) node[above]{$S$}--(9,-2);
\draw(10,2) node[above]{$N$}--(10,-2);
\draw(11,2) node[above]{$N$}--(11,-2);
\draw(12,2) node[above]{$S$}--(0,0)--(12,-2);
\end{braid}
&=
\pm
\begin{braid}\tikzset{baseline=0mm}
\draw(1,2) node[above]{$N$}--(1,-2);
\draw(2,2) node[above]{$S$}--(2,-2);
\draw(3,2) node[above]{$N$}--(3,-2);
\draw(4,2) node[above]{$N$}--(4,-2);
\draw(5,2) node[above]{$S$}--(5,-2);
\draw(6,2) node[above]{$N$}--(6,-2);
\draw(7,2) node[above]{$\cdots$};
\draw(8,2) node[above]{$N$}--(8,-2);
\draw(9,2) node[above]{$S$}--(9,-2);
\draw(10,2) node[above]{$N$}--(10,-2);
\draw(11,2) node[above]{$N$}--(11,-2);
\draw(12,2) node[above]{$S$}--(1.3,0)--(12,-2);
\blackdot(1.5,0);
\end{braid}
\mp
\begin{braid}\tikzset{baseline=0mm}
\draw(1,2) node[above]{$N$}--(1,-2);
\draw(2,2) node[above]{$S$}--(2,-2);
\draw(3,2) node[above]{$N$}--(3,-2);
\draw(4,2) node[above]{$N$}--(4,-2);
\draw(5,2) node[above]{$S$}--(5,-2);
\draw(6,2) node[above]{$N$}--(6,-2);
\draw(7,2) node[above]{$\cdots$};
\draw(8,2) node[above]{$N$}--(8,-2);
\draw(9,2) node[above]{$S$}--(9,-2);
\draw(10,2) node[above]{$N$}--(10,-2);
\draw(11,2) node[above]{$N$}--(11,-2);
\draw(12,2) node[above]{$S$}--(1.3,0)--(12,-2);
\blackdot(1,0);
\end{braid}
\\
&=
\pm
\begin{braid}\tikzset{baseline=0mm}
\draw(1,2) node[above]{$N$}--(1,-2);
\draw(2,2) node[above]{$S$}--(2,0.2)--(2.2,-0.3)--(12,-2);
\draw(3,2) node[above]{$N$}--(3,-2);
\draw(4,2) node[above]{$N$}--(4,-2);
\draw(5,2) node[above]{$S$}--(5,-2);
\draw(6,2) node[above]{$N$}--(6,-2);
\draw(7,2) node[above]{$\cdots$};
\draw(8,2) node[above]{$N$}--(8,-2);
\draw(9,2) node[above]{$S$}--(9,-2);
\draw(10,2) node[above]{$N$}--(10,-2);
\draw(11,2) node[above]{$N$}--(11,-2);
\draw(12,2) node[above]{$S$}--(2,-.1)--(2,-2);
\end{braid}
\\
&=
\pm
\begin{braid}\tikzset{baseline=0mm}
\draw(1,2) node[above]{$N$}--(1,-2);
\draw(2,2) node[above]{$S$}--(2.5,0.2)--(12,-2);
\draw(3,2) node[above]{$N$}--(2,0)--(3,-2);
\draw(4,2) node[above]{$N$}--(4,-2);
\draw(5,2) node[above]{$S$}--(5,-2);
\draw(6,2) node[above]{$N$}--(6,-2);
\draw(7,2) node[above]{$\cdots$};
\draw(8,2) node[above]{$N$}--(8,-2);
\draw(9,2) node[above]{$S$}--(9,-2);
\draw(10,2) node[above]{$N$}--(10,-2);
\draw(11,2) node[above]{$N$}--(11,-2);
\draw(12,2) node[above]{$S$}--(2.5,-.1)--(2,-2);
\end{braid}
\pm
\begin{braid}\tikzset{baseline=0mm}
\draw(1,2) node[above]{$N$}--(1,-2);
\draw(2,2) node[above]{$S$}--(2,-2);
\draw(3,2) node[above]{$N$}--(3,-2);
\draw(4,2) node[above]{$N$}--(4,-2);
\draw(5,2) node[above]{$S$}--(5,-2);
\draw(6,2) node[above]{$N$}--(6,-2);
\draw(7,2) node[above]{$\cdots$};
\draw(8,2) node[above]{$N$}--(8,-2);
\draw(9,2) node[above]{$S$}--(9,-2);
\draw(10,2) node[above]{$N$}--(10,-2);
\draw(11,2) node[above]{$N$}--(11,-2);
\draw(12,2) node[above]{$S$}--(3.2,0)--(12,-2);
\end{braid}.
\end{align*}

The first term in the last line involves an \((S,S)\)-crossing and hence is zero in \(C_\de\). We may continue on in this fashion, moving the \(S\) strand past \(NSN\)-triples, until we arrive at
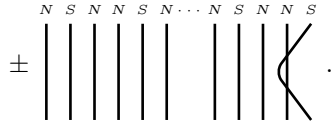
\begin{align*}
\pm
\begin{braid}\tikzset{baseline=0mm}
\draw(1,2) node[above]{$N$}--(1,-2);
\draw(2,2) node[above]{$S$}--(2,-2);
\draw(3,2) node[above]{$N$}--(3,-2);
\draw(4,2) node[above]{$N$}--(4,-2);
\draw(5,2) node[above]{$S$}--(5,-2);
\draw(6,2) node[above]{$N$}--(6,-2);
\draw(7,2) node[above]{$\cdots$};
\draw(8,2) node[above]{$N$}--(8,-2);
\draw(9,2) node[above]{$S$}--(9,-2);
\draw(10,2) node[above]{$N$}--(10,-2);
\draw(11,2) node[above]{$N$}--(11,-2);
\draw(12,2) node[above]{$S$}--(10.5,0)--(12,-2);
\end{braid}.
\end{align*}
The \((N,S)\) crossing opens, giving \(\pm(y_s-y_r)1_{\bi}\), for some \(s<r\). Recalling that the initial diagram was zero, we have \(y_s1_{\bi}=y_r1_{\bi}\). Applying induction on \(r\), for every semicuspidal word \(\bi\), it follows that \(y_1 = \cdots = y_{d-1}\) in \(C_\delta\).

Now we prove (ii). Let \(\bi=0i_2i_3 \cdots i_d \in G^\delta\). Again, this diagram is zero in \(C_\delta\):
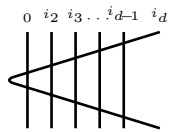
\begin{align*}
\begin{braid}\tikzset{baseline=0mm}
\draw(1,2) node[above]{$0$}--(1,-2);
\draw(2,2) node[above]{$i_2$}--(2,-2);
\draw(3,2) node[above]{$i_3$}--(3,-2);
\draw(4,2) node[above]{$\cdots$}--(4,-2);
\draw(5,2) node[above]{$i_{d\hspace{-.5mm}-\hspace{-.5mm}1}$}--(5,-2);
\draw(6.5,2) node[above]{$i_{d}$}--(0,0)--(6.5,-2);
\end{braid}
\end{align*}
As in the proof of (i), we omit non-neighbors of \(i_d\), and use the fact that \(\rednbr_d(\bi)=(NSN)^aNNS\) from Lemma \ref{wordfacts}(iii) to write
\begin{align*}
\begin{braid}\tikzset{baseline=0mm}
\draw(1,2) node[above]{$0$}--(1,-2);
\draw(2,2) node[above]{$i_2$}--(2,-2);
\draw(3,2) node[above]{$i_3$}--(3,-2);
\draw(4,2) node[above]{$\cdots$}--(4,-2);
\draw(5,2) node[above]{$i_{d\hspace{-.5mm}-\hspace{-.5mm}1}$}--(5,-2);
\draw(6.5,2) node[above]{$i_{d}$}--(0,0)--(6.5,-2);
\end{braid}
&=
\begin{braid}\tikzset{baseline=0mm}
\draw(1,2) node[above]{$N$}--(1,-2);
\draw(2,2) node[above]{$S$}--(2,-2);
\draw(3,2) node[above]{$N$}--(3,-2);
\draw(4,2) node[above]{$N$}--(4,-2);
\draw(5,2) node[above]{$S$}--(5,-2);
\draw(6,2) node[above]{$N$}--(6,-2);
\draw(7,2) node[above]{$\cdots$};
\draw(8,2) node[above]{$N$}--(8,-2);
\draw(9,2) node[above]{$S$}--(9,-2);
\draw(10,2) node[above]{$N$}--(10,-2);
\draw(11,2) node[above]{$N$}--(11,-2);
\draw(12,2) node[above]{$N$}--(12,-2);
\draw(13,2) node[above]{$S$}--(0,0)--(13,-2);
\end{braid}.
\end{align*}
We then move the \(S\)-strand past \((NSN)\)-strands as in the first part, to arrive at
\begin{align*}
\begin{braid}\tikzset{baseline=0mm}
\draw(1,2) node[above]{$N$}--(1,-2);
\draw(2,2) node[above]{$S$}--(2,-2);
\draw(3,2) node[above]{$N$}--(3,-2);
\draw(4,2) node[above]{$N$}--(4,-2);
\draw(5,2) node[above]{$S$}--(5,-2);
\draw(6,2) node[above]{$N$}--(6,-2);
\draw(7,2) node[above]{$\cdots$};
\draw(8,2) node[above]{$N$}--(8,-2);
\draw(9,2) node[above]{$S$}--(9,-2);
\draw(10,2) node[above]{$N$}--(10,-2);
\draw(11,2) node[above]{$N$}--(11,-2);
\draw(12,2) node[above]{$N$}--(12,-2);
\draw(13,2) node[above]{$S$}--(10.5,0)--(13,-2);
\end{braid}.
\end{align*}
Applying the quadratic relation twice yields \(\pm(y_t-y_d)(y_s-y_d)1_{\bi}\), for some \(t<s<d\). But  \(y_t=y_s=y_1\) by (i), so we have \((y_1-y_d)^21_{\bi}=0\) for all semicuspidal words \(\bi\), which implies that \((y_1-y_d)^2=0\) in $C_\delta$.
\end{proof}

\begin{Lemma}\label{Cpsi} Let \(u \in \mathfrak{S}_d\) and $\bi=i_1\cdots i_d\in I^\delta$. We have \(\psi_u 1_{\bi}=0\) in \(C_\delta\) unless:
\begin{enumerate}
\item \(\bi \in G^\delta\), \(u\bi \in G^{i_d}\), and \(u=w_{u\bi,\bi}\), in which case \(\deg(\psi_u1_{\bi})=0\), or;
\item \(\bi \in G^\delta\), \(u\bi \in G^j\) for some \(j \in I'\) such that \({\tt c}_{j,i_d}= -1\), and \(u=w_{u\bi,\bi}\), in which case \(\deg(\psi_u1_{\bi})=1\).
\end{enumerate} 
\end{Lemma}
\begin{proof}
The lemma is easily checked in type \({\tt A}^{(1)}_1\), since then \(G^\delta=\{01\}\) and \(\psi_11_{01} = 0\). Suppose we are not in type \({\tt A}^{(1)}_1\) and that \(\psi_u 1_{\bi} = e_{u\bi} \psi_u 1_{\bi}\neq 0\). Then \(\bi, u\bi \in G^\delta\) by Corollary~\ref{CSCW}. We may write \(u=w'w''\), where \(w'' \in \mathfrak{S}_{d-1}\) and \(w'\) is a minimal length left coset representative of \(\mathfrak{S}_{d-1}\) in $\mathfrak{S}_d$. By Lemma \ref{Cdelfacts}(i), \(\psi_u= \psi_{w'}\psi_{w''}\). By Lemma \ref{wordfacts}(iv), \(w''\) must be \(\bi\)-admissible. If \(w' = \textup{id}\), then \(u\bi_d = i_d\), \(\deg(\psi_u1_{\bi})=0\) and we are in case (i) by the uniqueness of Lemma \ref{wordfacts}(v).
Let \(w' \neq \textup{id}\). Then for some \(r\), \(w'=s_{r}s_{r+1} \cdots s_{d-1}\) is a reduced expression for \(w'\). By Lemma \ref{Cdelfacts}(i), \(\psi_u = \psi_r \psi_{r+1} \cdots \psi_{d-1} \psi_{w''}\) in \(C_\delta\). By Lemma \ref{wordfacts}(iv), \(s_r s_{r+1} \cdots s_{d-2}\) is \(s_{d-1}w''\bi\)-admissible. Further, \({\tt c}_{i_{d-1},i_d}=-1\) by Lemma \ref{wordfacts}(ii), so \(\deg(\psi_u1_{\bi})=1\), and we are in case (ii) by the uniqueness of Lemma \ref{wordfacts}(vii).
\end{proof}

Given a word \(\bi=i_1\cdots i_d \in G^\delta\), define
\begin{align*}
W_{\bi} = \{ w_{\bj,\bi} \in \mathfrak{S}_d \mid \bj \in G^j \textup{ for some } j \textup{ such that } {\tt c}_{j,i_d} \neq 0\}.
\end{align*}
Note that by Lemma \ref{wordfacts}(vii) and (viii), \(W_{\bi}\) is in bijection with \(\bigcup_{j \in I', {\tt c}_{j,i_d}\neq 0}G^j\).

\begin{Lemma}\label{Cypsi}
Let \(u \in \mathfrak{S}_d\) and $\bi\in I^\delta$. 
If \(\deg(\psi_u1_{\bi})\geq 1\), then \((y_1-y_d)\psi_u1_{\bi} = 0\) in \(C_\delta\).
\end{Lemma}
\begin{proof}
By Lemma \ref{Cpsi}, we only need consider the case where \(\bi \in G^\delta\) and \(u \in W_{\bi}\). Since \(\deg(\psi_u1_{\bi})\geq 1\), it must be that \(u\bi \in G^j\), where \({\tt c}_{j,i_d}=-1\), so \((u\bi)_d =j \neq i_d\) and \((u\bi)_1=i_1=0\).  Thus \(u(1)=1\) and \(u(d)< d\). By Lemma \ref{Cdelfacts}(ii), we have 
$
(y_1-y_d)\psi_u1_{\bi} = \psi_u (y_1-y_{u(d)})1_{\bi},
$ 
but \(y_1-y_{u(d)}=0\) in \(C_\delta\) by Lemma \ref{Cy}(i).
\end{proof}

\begin{Proposition}\label{Cspan}
The following is a spanning set for \(C_\delta\):
\begin{align*}
X:=\{y_1^b(y_1-y_d)^m \psi_w 1_{\bi} \mid \bi \in G^\delta,\, w \in W_{\bi},\, m+\deg(\psi_w1_{\bi}) \leq 1,\, b \in \ZZ_{\geq 0}\}.
\end{align*}
\end{Proposition}
\begin{proof}
By the basis theorem  \cite[Theorem 2.5]{KL1} or \cite[Theorem 3.7]{Ro}, we have that
\begin{align*}
\{y_1^{b_1} \cdots y_{d-1}^{b_{d -1}}(y_1 - y_d)^{b_d}\psi_w 1_{\bi} \mid \bi \in I^\delta, w \in \mathfrak{S}_d, b_i \in \ZZ_{\geq 0}  \}
\end{align*}
spans \(R_\delta\). We get the spanning set \(X\)  by throwing out elements of this set which are known to be zero or redundant in \(C_\delta\) via Lemmas \ref{Cy}, \ref{Cpsi} and \ref{Cypsi}.
\end{proof}

\subsection{\boldmath A basis for \(C_\delta\)}  To prove linear independence of \(X\), we construct a graded \(R_\delta\)-module which descends to a faithful \(C_\delta\)-module. 
For \(\bi,\bj \in I^\delta\), set
\begin{align*}
V_{\bi, \bj} := \begin{cases}
\kk[z,x]/(x^2) & \textup{if } \bi,\bj \in G^\delta \textup{ and }i_d = j_d\\
q\kk[z,x] /(x) & \textup{if } \bi,\bj \in G^\delta \textup{ and }{ \tt c}_{i_d,j_d}=-1\\
\kk[z,x]/(1) & \textup{otherwise},
\end{cases}
\end{align*}
where \(z,x\) are generators in degree 2, and $q$ stands for a degree shift up by $1$. 
Note that \(V_{\bi,\bj}=0\) in the `otherwise' cases above---it is presented as is for convenience in defining an action on \(V\). 
Set \(V = \bigoplus_{\bi,\bj \in I^\delta} V_{\bi,\bj}\). We will label polynomials \(f \in \kk[z,x]\) belonging to the \(\bi,\bj\)-th component of \(V\) with subscripts, a la \(f_{\bi,\bj}\). Recall the signs $\eps_{i,j}$ from \S\ref{SKLR}. 

\begin{Lemma}
The vector space \(V\) is a graded \(R_\delta\)-module, with the action of generators defined in types \({\tt C} \neq {\tt A}_1^{(1)}\) as follows:
\begin{align*}
1_{\bk} \cdot f_{\bi,\bj} &= \delta_{\bk,\bi} f_{\bi,\bj}\\
y_r \cdot f_{\bi,\bj} &= (zf-\de_{r,d}xf)_{\bi,\bj}
\\
\psi_r \cdot f_{\bi,\bj} &= \begin{cases}
f_{s_r\bi, \bj} & \textup{if }s_r \textup{ is } \bi\textup{-admissible};\\
f_{s_{d-1}\bi, \bj} & \textup{if }r=d-1 \textup{ and } i_d = j_d;\\
\varepsilon_{i_d,j_d}(xf)_{s_{d-1}\bi,\bj} & \textup{if }r=d-1 \textup{ and } i_{d-1}=j_d; \\
0 & \textup{otherwise.}
\end{cases}
\end{align*}
If \({\tt C}={\tt A}_1^{(1)}\), the action of \(1_{\bk}\), \(y_r\) are as above, but \(\psi_1v=0\) for all \(v\in V\).
\end{Lemma}

\begin{proof}
First we argue that the actions of the generators are well-defined. The only non-obvious case is the action of \(\psi_r\). Let \(r,\bi,\bj\) be such that \(\psi_r |_{V_{\bi,\bj}} \neq 0\). We show that \(\psi_r |_{V_{\bi,\bj}}:V_{\bi,\bj} \to V_{s_r \bi, \bj}\) is a well-defined \(\kk\)-linear homomorphism. Assume first that \(r <d-1\). Then \(\psi_r |_{V_{\bi,\bj}} \neq 0\) implies that \(s_r\) is \(\bi\)-admissible. Since \(i_d = (s_r\bi)_d\), we either have \(V_{\bi,\bj} = V_{s_r \bi, \bj} = \kk[z,x]/(x^2)\) or  \(V_{\bi,\bj} = V_{s_r \bi, \bj} = \kk[z,x]/(x)\), so in either case \(\psi_r |_{V_{\bi,\bj}}: f \mapsto f\) is well-defined. 
Assume next that  \(r=d-1\) and \(i_d = j_d\). In this case \(V_{\bi,\bj} = \kk[z,x]/(x^2)\), and by Lemma \ref{wordfacts}(ii), \(i_{d-1}\) is a neighbor of \(i_d\), so \((s_{d-1}\bi)_d = i_{d-1}\) is a neighbor of \(j_d\), and \(s_{d-1} \bi \in G^\delta\) by Lemma \ref{wordfacts}(viii), so \(V_{s_{d-1} \bi,\bj} = \kk[z,x]/(x)\). Thus \(\psi_{d-1} |_{V_{\bi,\bj}}: f \mapsto f\) is well-defined. Finally, assume that  \(r=d-1\) and \(i_{d-1} = j_d\). In this case \({\tt c}_{i_d,j_d}=-1\) by Lemma \ref{wordfacts}(ii), so \(V_{\bi,\bj} = \kk[z,x]/(x)\). Since \((s_{d-1} \bi)_d = j_d\), we have \(V_{s_{d-1}\bi,\bj} = \kk[z,x]/(x^2)\). Thus the map \(\psi_{d-1}|_{V_{\bi,\bj}}:f \mapsto \varepsilon_{i_d,j_d}xf\) is well-defined.

Now we check that the action satisfies the defining relations of \(R_\delta\). If \({\tt C} = {\tt A}_1^{(1)}\) then \(G^\delta = \{01\}\). Since \(\psi\)'s act as zero on \(V\), the only relation that is not clearly satisfied is (\ref{KLRpsi2}). In this case, for \(f \in V_{01,01}\), we have as desired:
\begin{align*}
Q_{01}(y_1,y_2) \cdot (1_{01} \cdot f_{01,01})=(y_1 - y_2) \cdot ((y_2 - y_1) \cdot  f_{01,01})
&= -(x^2f)_{01,01} = 0.
\end{align*}

Let \({\tt C} \neq {\tt A}_1^{(1)}\). The relations (\ref{KLRidem}), (\ref{KLRy}) and (\ref{KLRpsi}) are obvious. Since \(1_\bk\) acts as the projection \(V \to \bigoplus_{\bj \in I^\delta} V_{\bk,\bj}\) and \(\psi_r\) restricts to a map \(V_{\bi,\bj} \to V_{s_r\bi, \bj}\) for all \(\bi,\bj \in I^\delta\), relation (\ref{KLRpsiidem}) is satisfied as well.

For relation (\ref{KLRypsi}), we have
\begin{align}\label{ypsi1}
y_t \cdot (\psi_r  \cdot (1_\bi \cdot f_{\bj, \bk})) =
\begin{cases}
\delta_{\bi,\bj}(zf - \delta_{t,d}xf)_{s_r \bj, \bk} & \textup{if }s_r \textup{ is } \bj\textup{-admissible};\\
\delta_{\bi,\bj}(zf - \delta_{t,d}xf)_{s_{d-1}\bj, \bk} & \textup{if }r=d-1, j_d = k_d;\\
\delta_{\bi,\bj}\varepsilon_{j_d, k_d}(zxf - \delta_{t,d}x^2f)_{s_{d-1}\bj, \bk} & \textup{if }r=d-1, j_{d-1} = k_d;\\
0 & \textup{otherwise}.
\end{cases}
\end{align}
and 
\begin{align}\label{ypsi2}
\psi_r \cdot (y_{s_r t} \cdot (1_\bi \cdot f_{\bj, \bk})) =
\begin{cases}
\delta_{\bi,\bj}(zf - \delta_{s_rt,d}xf)_{s_r \bj, \bk} & \textup{if }s_r \textup{ is } \bj\textup{-admissible};\\
\delta_{\bi,\bj}(zf - \delta_{s_rt,d}xf)_{s_{d-1}\bj, \bk} & \textup{if }r=d-1, j_d = k_d;\\
\delta_{\bi,\bj}\varepsilon_{j_d, k_d}x(zf - \delta_{s_rt,d}xf)_{s_{d-1}\bj, \bk} & \textup{if }r=d-1, j_{d-1} = k_d;\\
0 & \textup{otherwise}.
\end{cases}
\end{align}
Since \(i_r \neq i_{r+1}\) for every \(r \in [1,d-1]\) and \(\bi \in G^\delta\) by Lemma \ref{wordfacts}(i), in order to show that relation (\ref{KLRypsi}) is satisfied, we simply must show that (\ref{ypsi1}) and (\ref{ypsi2}) are equal in every case. First, note that if \(s_r\) is \(\bj\)-admissible, then \(r<d-1\) by Lemma \ref{wordfacts}(ii), so \(\delta_{s_rt, d} = \delta_{t,d}\) in this case, and we have equality. Next, assume that \(r=d-1\), \(j_d =k_d\) and \(0 \neq f_{\bj,\bk} \in V_{\bj, \bk}\). Since \((s_{d-1}\bj)_d = j_{d-1} \neq j_d\) by Lemma \ref{wordfacts}(ii), we have that \(xf = 0 \in V_{s_{d-1}\bj, \bk} =\kk[z,x]/(x)\), so we have equality in the second case. In the third case, note that \(x^2 = 0 \in V_{\bj,\bk}\) for any \(\bj,\bk\), so equality holds in this case as well.

For relation (\ref{KLRpsi2}), note that \(s_r\) is \(\bi\)-admissible if and only if \(s_r\) is \(s_r \bi\)-admissible, so
\begin{align}\label{psiQ1}
\psi_r \cdot (\psi_r \cdot (1_\bi \cdot f_{\bj, \bk}))=
\begin{cases}
\delta_{\bi,\bj} f_{\bj,\bk} & \textup{if }s_r \textup{ is } \bj\textup{-admissible};\\
\delta_{\bi,\bj} \varepsilon_{j_{d-1},k_d} (xf)_{\bj,\bk} & \textup{if }r=d-1 \textup{ and } j_d = k_d;\\
\delta_{\bi,\bj} \varepsilon_{j_d,k_d} (xf)_{\bj,\bk} & \textup{if }r=d-1 \textup{ and } j_{d-1} = k_d;\\
0 & \textup{otherwise},
\end{cases}
\end{align}
and
\begin{align}\label{psiQ2}
Q_{i_r,i_{r+1}} \cdot (1_\bi \cdot f_{\bj, \bk})=
\begin{cases}
\delta_{\bi,\bj} f_{\bj,\bk} & \textup{if }{\tt c}_{j_r,j_{r+1}}=0;\\
\delta_{\bi,\bj} \varepsilon_{j_{d-1},j_d} (xf)_{\bj,\bk} & \textup{if }r=d-1;\\
0 & \textup{otherwise},
\end{cases}
\end{align}
If \(s_r\) is \(\bj\)-admissible (and hence \({\tt c}_{j_r, j_{r+1}} = 0\)), then (\ref{psiQ1}) and (\ref{psiQ2}) are equal. If \(r=d-1\) and \(j_d = k_d\), then the expressions clearly agree. Note that in every other case, (\ref{psiQ1}) and (\ref{psiQ2}) are both zero, since \((xf)_{\bj, \bk}=0\) whenever \(j_d \neq k_d\). 

Finally, we check relation (\ref{KLRbraid}). Fix \(r \in [1,d-2]\), and write \(\sigma = s_rs_{r+1}s_r = s_{r+1} s_r s_{r+1}\). If \(r<d-2\), then 
\begin{align*}
\psi_{r+1} \cdot (\psi_r \cdot (\psi_{r+1} \cdot (1_{\bi} \cdot f_{\bj, \bk}))) = 
\psi_{r} \cdot (\psi_{r+1} \cdot (\psi_{r} \cdot (1_{\bi} \cdot f_{\bj, \bk}))),
\end{align*}
since both terms are equal to \(\delta_{\bi,\bj}f_{\sigma\bj, \bk}\) if \({\tt c}_{j_r,j_{r+1}} = {\tt c}_{j_r,j_{r+2}} = {\tt c}_{j_{r+1},j_{r+2}} = 0\), and both are zero otherwise. Now assume \(r=d-2\), and assume that either \(\psi_{d-1} \cdot (\psi_{d-2} \cdot (\psi_{d-1} \cdot (1_{\bi} \cdot f_{\bj, \bk})))\) or \(\psi_{d-2} \cdot (\psi_{d-1} \cdot (\psi_{d-2} \cdot (1_{\bi} \cdot f_{\bj, \bk})))\) is nonzero. Then we must have \(V_{\bj,\bk}, V_{\sigma \bj, \bk} \neq 0\). Then, since \(\bj, \sigma \bj \in G^\delta\), by Lemma \ref{wordfacts}(i), we have that \(j_{d-2}, j_{d-1}, j_d\) are all distinct, and by Lemma \ref{wordfacts}(ii), \({\tt c}_{j_{d-1},j_d}={\tt c}_{j_{d-1},j_{d-2}}=-1\). 
Thus \(\psi_{d-2} \cdot (\psi_{d-1} \cdot (\psi_{d-2} \cdot (1_{\bi} \cdot f_{\bj, \bk}))) = 0\), since \(s_{d-2}\) is not \(\bj\)-admissible. Thus \(\psi_{d-1} \cdot (\psi_{d-2} \cdot (\psi_{d-1} \cdot (1_{\bi} \cdot f_{\bj, \bk})))\) must be nonzero. If \(j_d =k_d\), then \(j_{d-1} \neq k_d\), which implies \(j_{d-2}=(s_{d-2}s_{d-1}\bj)_{d-1}=k_d=j_d\), a contradiction. Then we must have \(j_{d-1}=k_d\) and \({\tt c}_{j_{d-2},j_d} = 0\), in which case \(\psi_{d-1} \cdot (\psi_{d-2} \cdot (\psi_{d-1} \cdot (1_{\bi} \cdot f_{\bj, \bk}))) = \delta_{\bi,\bj}\varepsilon_{j_{d-1},j_d}(xf)_{\sigma \bj, \bk}\). But, since \((\sigma \bj)_d = j_{d-2} \neq j_d = k_d\), we have that \(xf = 0\) in \(V_{\sigma \bj, \bk}\), another contradiction. Thus 
$
(\psi_{r+1} \psi_r \psi_{r+1} - \psi_r \psi_{r+1} \psi_r)1_{\bi} = 0
$ 
as an operator on \(V\). Moreover, \(j_r \neq j_{r+2}\) for every \(\bj \in G^\delta\) by Lemma \ref{wordfacts}(i), so the right side of relation (\ref{KLRbraid}) also acts identically as zero on \(V\). \end{proof}

\begin{Theorem}\label{Cbasis}
The set \(X\) of Proposition \ref{Cspan} is a basis for \(C_\delta\).
\end{Theorem}
\begin{proof}
 Note that 
 \begin{align*}
 \{(z^b)_{\bi, \bj} \mid \bi, \bj \in G^\delta, &{\tt c}_{i_d, j_d} = -1, b \in \ZZ_{\geq 0}\}\\
 &\cup \{(z^bx^m)_{\bi,\bj} \mid \bi, \bj \in G^\delta, i_d = j_d, m \in \{0,1\}, b \in \ZZ_{\geq 0}\}
 \end{align*}
 is a basis for \(V\). The \(R_\delta\)-module \(V\) factors through to a \(C_\delta\)-module since \(1_{\noncusp}V=0\) by Corollary~\ref{CSCW}. Letting  \(\bi, \bj \in G^\delta, w \in W_{\bi}, m+\deg(\psi_w1_{\bi}) \leq 1, b \in \ZZ_{\geq 0}\), we have, by Lemma \ref{wordfacts}(vii),
\begin{align*}
 y_1^b(y_1-y_d)^m \psi_w 1_{\bi}\cdot 1_{\bj,\bj} = \begin{cases}
\delta_{\bi,\bj} (z^b)_{w\bi,\bi} & \textup{if }\deg(\psi_w 1_{\bi})=1;\\
\delta_{\bi,\bj} (z^bx^m)_{w\bi,\bi} & \textup{if }\deg(\psi_w 1_{\bi})=0.
 \end{cases}
  \end{align*}
   For all \(w \in W_\bi\), \(w \bi \in G^\delta\), \((w \bi)_d = i_d\) when \(\deg(\psi_w 1_\bi)=0\), and \((w \bi)_{d}\) is a neighbor of \(i_d\) when \(\deg(\psi_w 1_\bi) = 1\). Moreover, \(w\bi = u\bi\) for \(w,u \in W_\bi\) if and only if \(w=u\), so the elements of \(X\) act on \(V\) as linearly independent operators. Taking into account Proposition~\ref{Cspan}, we deduce that \(X\) is a basis for \(C_\delta\).
  \end{proof}

For each \(\alpha \in Q_+\) and dominant weight \(\Lambda\) associated to \({\tt C}\), there is an important quotient \(R_\alpha^\Lambda\) of \(R_\alpha\) called the {\em cyclotomic KLR algebra} (see e.g. \cite{KL1,BKcyc}). Of relevance to the discussion at hand is the level-one case \(R_\delta^{\Lambda_0}\); it is by definition the quotient of \(R_\delta\) by the two-sided ideal generated by the elements \(\{y_1^{\delta_{i_1,0}}1_{\bi} \mid \bi \in I^\delta\}\). By \cite[Lemma 5.1]{Kcusp}, when $\kk$ is a field, \(\{L_{\delta,i} \mid i \in I'\}\) is a full set of irreducible modules for \(R_\delta^{\Lambda_0}\), so \(1_{\bi}=0\) in \(R_\delta^{\Lambda_0}\) unless \(\bi \in G^\delta\). So by Corollary~\ref{CSCW}, there is a natural surjection \(C_\delta \twoheadrightarrow R_\delta^{\La_0} \cong C_\delta/C_\delta y_1 C_\delta\).

\begin{Lemma}\label{Ccyc}
There is an isomorphism  of graded \(\kk\)-algebras \(C_\de\cong \kk[y_1] \otimes R_\delta^{\La_0}\), and \(R_\delta^{\La_0}\), considered as a subalgebra of \(C_\delta\),  has basis
\begin{align}\label{Ccycbasis}
\{(y_1-y_d)^m \psi_w 1_{\bi} \mid \bi \in G^\delta, w \in W_{\bi}, m+ \deg(\psi_w1_{\bi}) \leq 1\}.
\end{align}
\end{Lemma}
\begin{proof}
We may construct a map \(\varphi:R_\delta^{\Lambda_0} \to C_\delta\) via:
\begin{align*}
1_{\bi} \mapsto 1_{\bi}, \hspace{10mm} \psi_r \mapsto \psi_r, \hspace{10mm} y_r \mapsto y_r-y_1.
\end{align*}
All defining relations of \(R_\delta^{\Lambda_0}\) are preserved by the map---the only non-obvious relation to check is (\ref{KLRypsi}), which follows since \(y_1\) is central in \(C_\delta\) by Lemma \ref{Cy}(iii). Thus \(\varphi\) is a well-defined homomorphism of graded \(\kk\)-algebras which splits the natural surjection \(C_\delta \twoheadrightarrow R_\delta^{\Lambda_0}\). The set (\ref{Ccycbasis}) is clearly in the image of \(\varphi\), so the result follows by Lemmas \ref{Cy}, \ref{Cpsi}, \ref{Cypsi} and Theorem \ref{Cbasis}.
\end{proof}

\subsection{\boldmath Description of the algebra $B_\de$}\label{zigsec}

Recall the signs $\eps_{i,j}$ from \S\ref{SKLR}. 
Define 
\begin{align*}
\xi_1 := \begin{cases}
1
&
\textup{if }{\tt C} = {\tt A}_{1}^{(1)};\\
\varepsilon_{10} \cdots \varepsilon_{\ell,\ell-1}\varepsilon_{0,\ell}
&
\textup{if }{\tt C} = {\tt A}_{\ell > 1}^{(1)};\\
(-1)^{\ell}
&
\textup{if }{\tt C} = {\tt D}_{\ell}^{(1)};\\
-1
&
\textup{if }{\tt C}= {\tt E}_\ell^{(1)}.
\end{cases}
\end{align*}
Then for all other \(i \in I'\), define \(\xi_i\) such that \(\xi_i\xi_j=-1\) whenever \({\tt c}_{i,j}=-1\) (this is possible as \(\Gamma\) is a tree). We also define, for all \(i,j \in I'\) with \({\tt c}_{i,j} = -1\), the constants
\begin{align*}
\mu_{ji}=
\begin{cases}
\varepsilon_{ji} & \textup{if }\xi_i=1;\\
1 & \textup{if }\xi_i = -1.\\
\end{cases}
\end{align*}
\begin{Lemma}\label{ximu}
For all \(i,j \in I'\) with \({\tt c}_{i,j} = -1\), we have \(\varepsilon_{ji} \xi_i = \mu_{ij} \mu_{ji}\).
\end{Lemma}
\begin{proof}
This is a direct check, just using the fact that by definition \(\varepsilon_{ij} = -\varepsilon_{ji}\).
\end{proof}

Assume for a moment that $\k$ is a field. Then $\{L_{\de,i}\mid i\in I'\}$ is a complete set of irreducible $C_\de$-modules up to isomorphism and degree shift, Moreover, the orthogonal idempotents \(\{1_{\bi} \mid \bi \in G^\delta\}\) in \(C_\delta\) are primitive, since by Theorem~\ref{Cbasis}, the space \((1_{\bi} C_\delta 1_{\bi})_0\) is 1-dimensional. Set  
$$1_{i} := 1_{\bb^i}\quad(\textup{for }i \in I'),\qquad
1_{\Delta} := \sum_{i\in I'}1_i. 
$$
By Lemma~\ref{Ldel}, we have \(1_iL_{\delta,i} \neq 0\), and \(1_iL_{\delta, j}=0\) for every \(i \neq j \in I'\). So the projective cover \(\Delta_{\delta,i}\) of \(L_{\delta,i}\) in \(C_\delta\) is isomorphic to \(C_\delta 1_i\), and 
$$
\Delta_\delta:= \bigoplus_{i\in I'} \Delta_{\delta,i} \cong C_{\delta} 1_\Delta.
$$ 
is a projective generator in $\mod{C_\de}$. We want to compute the endomorphism algebra 
$$B_\de:=\End_{C_\de}(\De_\de)^\op\cong 1_\De C_\de 1_\De.$$
Observe that the definitions of $C_\de$, $\De_\de$, $1_\De$, $B_\de$, $L_{\de,i}$, etc. make sense over an arbitrary commutative ground ring $\k$. In fact, all our computations below will be done in this generality. 

The following lemma follows from consideration of Theorem \ref{Cbasis}:

\begin{Lemma}\label{MinDelHom}
For \(i,j \in I'\), \(\textup{Hom}_{C_\delta}(\Delta_{\delta,i}, \Delta_{\delta,j}) \cong 1_i C_{\delta} 1_j\) as \(\kk\)-modules, and \(1_i C_{\delta} 1_j\) has \(\kk\)-basis
\begin{align*}
\{y_1^b (y_1-y_d)^m 1_j \mid b \in \mathbb{Z}_{\geq 0}, m \in \{0,1\}\} \textup{ if }i=j,
\end{align*}
and
\begin{align*}
\{y_1^b \psi_{i,j}1_j\mid b \in \mathbb{Z}_{\geq 0}\} \textup{ if }{\tt c}_{i,j}=-1,
\end{align*}
and is zero otherwise.
\end{Lemma}

Recall that \(\Gamma\) is the Dynkin diagram corresponding to the finite type Cartan matrix \({\tt C}'\), so the vertices of \(\Gamma\) are identified with the set \(I'\). The following theorem establishes a Morita equivialence between the cyclotomic KLR algebra \(R_\delta^{\Lambda_0}\) and the zigzag algebra \(\Zig = \Zig(\Gamma)\).

\begin{Theorem}\label{Zigisom}
Consider \(1_\Delta R_\delta^{\Lambda_0} 1_\Delta\) as a subalgebra of \(C_\delta\) via Lemma \ref{Ccyc}.
\begin{enumerate}
\item \(1_\Delta R_\delta^{\Lambda_0} 1_\Delta\) has basis
\begin{align*}
\{(y_1-y_d)^m 1_j \mid j \in I',\, m \in \{0,1\}\} \cup \{ \psi_{i,j}1_j \mid i,j \in I',\, {\tt c}_{i, j}=-1\}.
\end{align*}
\item There is an isomorphism of graded algebras 
\begin{equation}\label{E100617} 
\varphi: 1_\Delta R_\delta^{\Lambda_0} 1_\Delta \iso \textup{\(\Zig\)},\quad 
1_i \mapsto \ze_i, \quad
(y_1-y_d) 1_i\mapsto \xi_i \zc \ze_i,\quad
\psi_{j,i} 1_i\mapsto \mu_{ji} \za^{j,i},
\end{equation}
\end{enumerate}
\end{Theorem}

\begin{proof}
Part (i) follows immediately from Lemma~\ref{Ccyc}. For part (ii), let $\phi:1_\Delta R_\delta^{\Lambda_0} 1_\Delta \iso \textup{\(\Zig\)}$ be the degree zero homogeneous linear isomorphism defined on basis elements as in (\ref{E100617}). To check that $\phi$ respects multiplication, observe that elements of the form \((y_1-y_d)1_\bi\) and \(\psi_{j,i}1_i\) have degrees 1 and 2 respectively, and both \( 1_\Delta R_\delta^{\Lambda_0} 1_\Delta\) and \(\Zig\) are concentrated in degrees 0,1,2. Thus the only non-obvious check is that \(\varphi(\psi_{i,j}1_j \cdot \psi_{k,l}1_l) = \varphi(\psi_{i,j}1_i) \varphi(\psi_{k,l}1_l)\).

Note that by (i), if \(x \in 1_\Delta R_\delta^{\Lambda_0} 1_\Delta\) is in degree 2, then \(1_i x 1_j = 0\) whenever \(i \neq j\). Thus 
\begin{align*}
\psi_{i,j}1_j \cdot \psi_{k,l}1_l = \delta_{j,k} 1_i \psi_{i,j}\psi_{j,l}1_l = \delta_{j,k} \delta_{i,l} \psi_{i,j} \psi_{j,i} 1_i.
\end{align*}
By Lemma \ref{wordfacts}(vi), \(w_{i,j} = w_1 s_{d-1} w_2\) for some \(w_2\) which is \(\bb^j\)-admissible, and some \(w_1\) which is \(s_{d-1}w_2\bb^j\)-admissible, and \(w_{j,i} = w_{i,j}^{-1}\). Then by Lemma \ref{Cdelfacts}(i), we may write
\begin{align*}
\psi_{i,j}1_j \cdot \psi_{k,l}1_l = \delta_{j,k} \delta_{i,l} \psi_{w_1} \psi_{d-1} \psi_{w_2} \psi_{w_2^{-1}} \psi_{d-1} \psi_{w_1^{-1}} 1_i.
\end{align*}
Since \(w_2^{-1}\) is \(s_{d-1}w_1^{-1}\bb^i\)-admissible, and \(w_2\) is \(w_2^{-1}s_{d-1}w_1^{-1}\bb^i\)-admissible, it follows from relation (\ref{KLRpsi2}) that \(\psi_{w_2}\psi_{w_2^{-1}}1_{s_{d-1}w_1^{-1}} =1_{s_{d-1}w_1^{-1}}\). Thus we have 
\begin{align*}
\psi_{i,j}1_j \cdot \psi_{k,l}1_l =\delta_{j,k} \delta_{i,l} \psi_{w_1} \psi_{d-1} \psi_{d-1} \psi_{w_1^{-1}} 1_i.
\end{align*}
Since \(w_2^{-1}s_{d-1}w_1^{-1}\bb^i = \bb^j\), where \(w_1^{-1}\) is \(\bb^i\) is \(\bb^i\)-admissible and \(w_2^{-1}\) is \(s_{d-1}w_1^{-1}\bb^i \)-admissible, we have by Lemma \ref{wordfacts}(ii) that \((w_1^{-1}\bb^i)_d = i\) and \((w_1^{-1}\bb^i)_{d-1} = j\). Thus by relation (\ref{KLRpsi2}) we have 
\begin{align*}
\psi_{i,j}1_j \cdot \psi_{k,l}1_l &=\delta_{j,k} \delta_{i,l} \psi_{w_1} \varepsilon_{j,i}(y_{d-1}-y_d) \psi_{w_1^{-1}} 1_i = \delta_{j,k} \delta_{i,l} \psi_{w_1} \varepsilon_{j,i}(y_1-y_d) \psi_{w_1^{-1}} 1_i\\
&= \delta_{j,k} \delta_{i,l} \varepsilon_{j,i}(y_1-y_d)\psi_{w_1}  \psi_{w_1^{-1}} 1_i
= \delta_{j,k} \delta_{i,l} \varepsilon_{j,i}(y_1-y_d) 1_i.
\end{align*}
The second equality follows from Lemma \ref{Cy}(i). The third equality follows since \(w_1\), being \(w_1^{-1}\bb^i\)-admissible, cannot involve \(s_1\) or \(s_{d-1}\) by Lemma \ref{wordfacts}(ii). The fourth equality follows by admissibility of \(w_1\) and \(w_1^{-1}\).
Thus
\begin{align*}
\varphi(\psi_{i,j}1_j \cdot \psi_{k,l}1_l) =  
\xi_i\delta_{j,k} \delta_{i,l} \varepsilon_{ji} \zc \ze_i
\end{align*}
Now
\(
\varphi(\psi_{i,j}1_i) \varphi(\psi_{k,l}1_l)
=\kappa\za^{i,j}\za^{k,l}\), for some \(\kappa \in \kk\). The zigzag relations imply that this is zero unless \(i=l\) and \(j=k\). So
\begin{align*}
\varphi(\psi_{i,j}1_i) \varphi(\psi_{k,l}1_l) &= 
\delta_{j,k} \delta_{i,l} (\mu_{ij}\za^{i,j})(\mu_{ji}\za^{j,i})
=
\delta_{j,k} \delta_{i,l}\mu_{ij} \mu_{ji} \zc \ze_i
\end{align*}
Thus it follows from Lemma \ref{ximu} that \(\varphi(\psi_{i,j}1_j \cdot \psi_{k,l}1_l) = \varphi(\psi_{i,j}1_i) \varphi(\psi_{k,l}1_l)\). 
\end{proof}

\begin{Corollary}\label{minisom}
Let \(z\) be an indeterminate in degree $2$. We have isomorphisms of graded algebras  
$$B_\delta = \textup{End}_{C_\delta}( \Delta_\delta)^\textup{op} \cong 1_\Delta C_\delta 1_\Delta \cong\kk[y_1] \otimes 1_\Delta R_\delta^{\Lambda_0} 1_\Delta\cong \kk[z] \otimes \textup{\(\Zig\)}\cong \textup{\(\Zig\)}^\aff_1 \cong \textup{End}_{C_\delta}(\Delta_\delta).$$ 
\end{Corollary}
\begin{proof}
The first isomorphism is standard. The second isomorphism follows from Lemma \ref{Ccyc}. The third isomorphism follows from Theorem \ref{Zigisom}. The fourth isomorphism follows from the definition of \(\Zig^\aff_1\). The fifth isomorphism follows from Lemma \ref{affzigop}.
\end{proof}

Now, to avoid confusion, we will write \(v_i:=1_i\) for the generating vector of word \(\bb^i\) in \(\Delta_{\delta,i} = C_\delta 1_i\). Let \(v_\delta := \sum_{i \in I'}v_i \in \Delta_\delta\). 

\begin{Corollary}\label{ExpEndZig}
The \(\kk\)-algebra \(\End_{C_\delta}(\Delta_\delta)\) is generated by the homomorphisms 
\begin{align*}
\begin{array}{c}
e_i: 
v_\delta \mapsto 1_iv_\delta,
\end{array}
\hspace{4mm}
\begin{array}{c}
z: 
v_\delta \mapsto y_1 v_\delta,
\end{array}
\hspace{4mm}
\begin{array}{c}
c_i: 
v_\delta \mapsto \xi_i(y_1-y_d) 1_iv_\delta,
\end{array}
\hspace{4mm}
\begin{array}{c}
a^{i,j}: 
v_\delta \mapsto \mu_{ji}\psi_{j,i} 1_i v_\delta,
\end{array}
\end{align*}
where $i$ runs over \(I'\) and \(j\) runs over all neighbors of $i$ in $I'$, subject only to the same relations as their namesakes in \(\kk[z] \otimes\textup{\(\Zig\)} \cong \textup{\(\Zig\)}^\aff_1\): 
\begin{align*}
&\sum_{i \in I'}e_i = 1,
\hspace{6mm}
e_i\circ e_j = \delta_{i,j}e_i, 
\hspace{19.5mm}
a^{i,j} \circ a^{k,l} = \delta_{j,k} \delta_{i,l} c_i
\hspace{11mm}
e_k\circ a^{i,j} =\de_{i,k}a^{i,j},
\\
&c_i \circ c_j = 0,
\hspace{6mm}
e_i \circ c_j =c_j \circ e_i = \delta_{i,j} c_j,
\hspace{6mm}
a^{i,j} \circ c_k = c_k \circ a^{i,j} = 0,
\hspace{6mm}
a^{i,j} \circ e_k=\de_{j,k}a^{i,j},
\end{align*}
for all admissible \(i,j,k,l \in I'\), and \(z\circ g = g\circ z\) for all generators $g$. 
\end{Corollary}
\begin{proof}
Follows directly from tracing through the isomorphisms of Corollary \ref{minisom}.
\end{proof}

\section{On the higher imaginary stratum categories}\label{HighImag}

Suppose for a moment that \(\kk\) is a field. 
It is shown in \cite{McNAff, KMStrat} that the $R_{n \delta}$-module \(\Delta_\delta^{\circ n} = \textup{Ind}_{\delta, \ldots, \delta}(\Delta_\delta^{\boxtimes n})\) factors through to a projective \(C_{n \delta}\)-module, and \(\Delta_\delta^{\circ n}\) is a projective generator for \(C_{n \delta}\) if \(\cha \kk=0\) or \(\cha \kk>n\). 
We will build on the previous section to explicitly describe for all \(n\) the algebra \(\End_{C_{n\delta}}(\Delta^{\circ n}_\delta)\)  as the rank $n$ affine zigzag algebra $\Zig_n^\aff(\Gamma)$, defined in \S\ref{SSAffZig}, where $\Gamma$ is the finite type Dynkin diagram of type $\Car'$. This gives a Morita equivalence between \(C_{n\delta}\) and \(\Zig_n^\aff\) when \(\cha \kk=0\) or \(\cha \kk>n\). In fact, our proof that $\End_{C_{n\delta}}(\Delta^{\circ n}_\delta)\cong \Zig_n^\aff(\Gamma)$ works over any commutative until ground ring $\kk$.

\subsection{\boldmath Endomorphisms of \(\Delta_\delta^{\circ n}\)}\label{EndSec} First we compute the graded dimension of $\End_{C_{n\delta}}(\Delta_\delta^{\circ n})$:

\begin{Lemma}\label{Endcircdim}
For \(n \in \mathbb{Z}_{\geq0}\) we have 
\begin{align*}
\dim_q\End_{C_{n\delta}}(\Delta_\delta^{\circ n}) = \frac{n!( \ell + 2(\ell-1)q + \ell q^2    )^n}{(1-q^2)^n}.
\end{align*}
\end{Lemma}
\begin{proof}
By the Mackey Theorem \cite[Proposition 2.18]{KL1} (see also  \cite[Theorem 4.3]{KMStrat}), the restriction \(\textup{Res}_{\delta, \ldots, \delta}\textup{Ind}_{\delta, \ldots, \delta}(\Delta_\delta^{\boxtimes n})\) has filtration with \(n!\) subquotients all of which are isomorphic to \(\Delta_\delta^{\boxtimes n}\). But \(\Delta_\delta^{\boxtimes n}\) is projective as a \(C_{\delta}^{\otimes n}\)-module, so these subquotients are in fact summands. So Frobenius Reciprocity gives
\begin{align*}
\textup{End}_{C_{n\delta}}(\Delta_\delta^{\circ n}) \cong \textup{Hom}_{C_{\delta}^{\otimes n} }\left( \Delta_\delta^{\boxtimes n}, (\Delta_\delta^{\boxtimes n})^{\oplus n!}\right)
\cong (\End_{C_\delta}(\Delta_\delta)^{\otimes n})^{\oplus n!} \cong ((\kk[\zz] \otimes \Zig)^{\otimes n})^{\oplus n!}
\end{align*}
as $\kk$-modules. The result now follows by Lemma \ref{zigfacts}(ii). 
\end{proof}

Recalling $e_i,z,c_i,a^{i,j}\in \End_{C_{\delta}}(\Delta_\delta)$ from Corollary~\ref{ExpEndZig}, we set \(c := \sum_{i \in I'} c_i\). Let 
$\bi=(i_1,\dots,i_n) \in (I')^n$, \(g \in \{z, c, a^{i,j}\}\) and \(1 \leq r \leq n\). We define endomorphisms
\begin{align*}
e_\bi:= e_{i_1} \circ \cdots \circ e_{i_n},
\ 
g_r:= \id^{\circ (r-1)} \circ \, g \circ \id^{\circ (n-r)} \, \in\, \End_{C_{n\delta}}(\Delta^{\circ n}_\delta).
\end{align*}
Writing 
$$\Delta_{\bi} := \Delta_{\delta,i_1} \circ \cdots \circ \Delta_{\delta,i_n},$$ 
we have that \(\Delta_\delta^{\circ n} = \bigoplus_{\bi \in (I')^n} \Delta_{\bi}\), and \(e_{\bi}\) is the projection to the summand \(\Delta_{\bi}\).

\subsection{Twist endomorphisms}
We describe one more family of endomorphisms of \(\Delta_\delta^{\circ n}\). Let $L_\de:=\bigoplus_{i\in I'}L_{\de,i}$. 
For \(\bi \in (I')^n\), we will write \(L_\bi := L_{\delta,i_1} \circ \cdots \circ L_{\delta,i_n}\). As explained in \cite{KM}, there exists, for every \(i,j \in I'\), a distinguished nonzero degree-zero homomorphism \(r^{i,j}: L_{\delta,i} \circ L_{\delta,j} \to L_{\delta,j} \circ L_{\delta,i}\). We will describe this map explicitly later in this section. We have \(L_\delta^{\circ 2} = \bigoplus_{i,j \in I'} L_{\delta, i} \circ L_{\delta, j}\), so we may consider \(r := \sum_{i,j \in I'} r^{i,j}\) as an endomorphism of \(L_\delta^{\circ 2}\). More generally, for \(t \in [1,n-1]\), we have an endomorphism \(r_t\) of \(L_\delta^{\circ n}\) given by 
\begin{align*}
r_t:= \id^{\circ (t-1)} \circ\,  r \circ \id^{\circ (n-t-1)}.
\end{align*}
It can be seen as in \cite[Theorem 4.2.1]{KM}, that \(r_1 \ldots, r_{n-1}\) satisfy Coxeter relations of the symmetric group \(\mathfrak{S}_n\), and, together with the projections \(L_\delta^{\circ n} \to L_\bi\), generate a subalgebra \(T\) of dimension \(\ell^nn!\) in \(\End_{C_{n\delta}}(L_\delta^{\circ n}) = \End_{C_{n\delta}}(L_\delta^{\circ n})_0\).

The projective \(C_\delta^{\otimes n}\)-module \(\Delta_\delta^{\boxtimes n}\) surjects onto \(L_\delta^{\boxtimes n}\), which induces a (degree zero) surjection \(\pi:\Delta_\delta^{\circ n} \to L_\delta^{\circ n}\). For \(t \in [1,n-1]\), the Coxeter relations imply that \(r_t^2\) is the identity function on \(L_\delta^{\circ n}\), so \(r_t\) is an isomorphism of \(L_\delta^{\circ n}\). Since \(\Delta_\delta^{\circ n}\) is a projective \(C_{n\delta}\)-module, \(r_t\) lifts to a surjection \(\tilde{r}_t: \Delta_\delta^{\circ n} \to L_\delta^{\circ n} \). Then, again by projectivity of \(\Delta_\delta^{\circ n}\), \(\tilde{r}_t\) lifts to an endomorphism \(\hat{r}_t : \Delta_\delta^{\circ n} \to \Delta_\delta^{\circ n}\), as shown in the commuting diagram below:
$$
\begin{tikzpicture}[scale=0.55, line join=bevel]
\node at (0,0) {$L^{\circ n}_\delta$};
\node at (0,3) {$\Delta^{\circ n}_\delta$};
\node at (3,0) {$L^{\circ n}_\delta$};
\node at (3,3) {$\Delta^{\circ n}_\delta$};

\draw [->,dashed] (0.8,3) -- (2.3,3);
\draw [->] (0.8,0) -- (2.3,0);
\draw [->>] (0,2.5) -- (0,0.6);
\draw [->>] (3,2.5) -- (3,0.6);
\draw [->>,dashed] (2.7,2.5) -- (0.2,0.6);
\node at (1.6,-0.4) {${\scriptstyle r_t}$};
\node at (1.6,0.2) {$\sim$};
\node at (-0.3,1.55) {${\scriptstyle \pi}$};
\node at (3.3,1.55) {${\scriptstyle \pi}$};
\node at (1.6,1.2) {${\scriptstyle \tilde{r}_t}$};
\node at (1.6,3.3) {${\scriptstyle \hat{r}_t}$};
\end{tikzpicture}
$$
Moreover, since \(r_t\) is a degree zero map, we have \(\hat{r}_t \in \End_{C_{n\delta}}(\Delta_\delta^{\circ n})_0\). We also have, for every \(i \in (I')^n\), the projection \(e_\bi : \Delta_\delta^{\circ n} \to \Delta_\bi \subset \Delta_\delta^{\circ n}\), which lifts the projection \(L_\delta^{\circ n} \to L_\bi \subset L_\delta^{\circ n}\) to an element of \(\End_{C_{n\delta}}(\Delta_\delta^{\circ n})_0\).

\begin{Lemma}\label{DelLift} We have:
\begin{enumerate}
\item Every element of \(T \subseteq \End_{C_{n\delta}}(L_\delta^{\circ n})\) may be lifted to an element of \(\End_{C_{n\delta}}(\Delta_\delta^{\circ n})_0\), and this lift is unique.
\item The elements \(\hat{r}_1, \ldots, \hat{r}_{n-1}\) satisfy the Coxeter relations of \(\mathfrak{S}_n\).
\end{enumerate}
\end{Lemma}
\begin{proof}
By the above paragraph, we have the endomorphisms \(\{\widehat{r}_1, \ldots, \widehat{r}_{n-1}\}\), and \(\{e_\bi \mid \bi \in (I')^n\}\) in \(\End_{C_{n\delta}}(\Delta_\delta^{\circ n})_0\), which lift the generators of \(T\). Thus every element of \(T\) may be lifted to an element of \(\End_{C_{n\delta}}(\Delta_\delta^{\circ n})_0\). But \(T\) has a basis which lifts to give \(\ell^n n!\) linearly independent elements in \(\End_{C_{n\delta}}(\Delta_\delta^{\circ n})_0\), so by Lemma \ref{Endcircdim}, this constitutes a basis for \(\End_{C_{n\delta}}(\Delta_\delta^{\circ n})_0\). It follows that lifts of elements of \(T\) to \(\End_{C_{n\delta}}(\Delta_\delta^{\circ n})_0\) must be unique. Part (ii) follows from (i) and the fact that \(r_1, \ldots, r_{n-1}\) satisfy Coxeter relations.
\end{proof}

Let \(\sigma, \sigma' \in R_{2\delta}\) be the following products of \(\psi\)'s, displayed diagrammatically:
\begin{align*}
\sigma:=
\begin{braid}\tikzset{baseline=0mm}
\draw(0,2) node[above]{$1$}--(6,-2);
\draw(1.5,2) node[above]{$2$}--(7.5,-2);
\draw(3,2) node[above]{$\cdots$};
\draw(4.5,2) node[above]{$d$}--(10.5,-2);
\draw(6,2) node[above]{$d\hspace{-0.9mm}+\hspace{-0.9mm}1$}--(0,-2);
\draw(7.5,2) node[above]{$d\hspace{-0.9mm}+\hspace{-0.9mm}2$}--(1.5,-2);
\draw(9,2) node[above]{$\cdots$};
\draw(10.5,2) node[above]{$2d$}--(4.5,-2);
\end{braid},
\hspace{10mm}
\sigma':=
\begin{braid}\tikzset{baseline=0mm}
\draw(0,2) node[above]{$1$}--(0,-2);
\draw(1.5,2) node[above]{$2$}--(7.5,-2);
\draw(3,2) node[above]{$\cdots$};
\draw(4.5,2) node[above]{$d$}--(10.5,-2);
\draw(6,2) node[above]{$d\hspace{-0.9mm}+\hspace{-0.9mm}1$}--(3,0)--(6,-2);
\draw(7.5,2) node[above]{$d\hspace{-0.9mm}+\hspace{-0.9mm}2$}--(1.5,-2);
\draw(9,2) node[above]{$\cdots$};
\draw(10.5,2) node[above]{$2d$}--(4.5,-2);
\end{braid}.
\end{align*}
The labels in this case only indicate strand position and are not meant to color the strands.

In order to understand the multiplicative structure of \(\End_{C_{n\delta}}(\Delta^{\circ n}_\delta)\), we will need to describe the maps \(\hat{r}_t\) more explicitly and examine commutation relations between these maps and the others detailed in \S\ref{EndSec}. The following two lemmas are steps in this direction. Their proofs are straightforward but rather lengthy exercises in manipulating KLR diagrams. For this reason we defer the proofs until \S\ref{diagproofs}. 

The generators  \(v_i\in\Delta_{\delta,i}\) introduced in \S\ref{zigsec}, yield generators 
\begin{align*}
v_\bi = v_{i_1, \ldots, i_n}:= 1 \otimes v_{i_1} \otimes \cdots \otimes v_{i_n} \in \Delta_{\bi} = \Ind_{\de, \ldots, \de}( \Delta_{\delta, i_1} \boxtimes \cdots \boxtimes \Delta_{\delta, i_n}), 
\end{align*}
for \(\bi = (i_1, \ldots, i_n) \in (I')^n\). The elements $a\otimes b$ for $a,b\in R_\de$ are interpreted as elements of $R_{2\de}$ via the parabolic embedding $R_\de\otimes R_\de\into R_{2\de}$. With this notation we have:

\begin{Lemma}\label{sigmaprime}
Let \(i,j \in I'\). In   \(\Delta_{\delta,i} \circ \Delta_{\delta,j}\),  we have
\begin{align*}
\sigma'v_{i,j} = 
\begin{cases}
\xi_i [y_d \otimes 1 + 1 \otimes (y_d-2y_1)] v_{i,i}
&
\textup{if }i=j;\\
\xi_i \varepsilon_{ij} (\psi_{j,i} \otimes \psi_{i,j})v_{i,j}
&
\textup{if }{\tt c}_{i,j}=-1;\\
0
&
\textup{otherwise}.
\end{cases}
\end{align*}
\end{Lemma}

\begin{Lemma}\label{psisigma}
Let \(i,j,m \in I'\) with \({\tt c}_{i,j}=-1\). In \(\Delta_{\delta,m} \circ \Delta_{\delta,i}\), we have 
\begin{align*}
(\psi_{j,i} \otimes 1)\sigma v_{m,i} = [\sigma(1 \otimes \psi_{j,i}) + \delta_{j,m} \xi_j(1 \otimes \psi_{j,i}) - \delta_{i,m} \xi_i(\psi_{j,i} \otimes 1)]v_{m,i}.
\end{align*}
\end{Lemma}

Now we briefly describe the construction of the map \(r^{i,j}\), presented in \cite{KKK,KM}. It is recommended that the interested reader consult that paper for a thorough treatment.
If \(x\) is an indeterminate in degree 2, let \(\iota: R_\delta \to\kk[x] \otimes R_\delta\) be the algebra homomorphism defined by \(\iota(1_{\bi}) = 1_{\bi}\), \(\iota(\psi_r)=\psi_r\), and \(\iota(y_r)=y_r +x\). Let \(L_{\delta,i,x}:=\kk[x] \otimes L_{\delta,i}\) be the \(\kk[x] \otimes R_\delta\)-module with action twisted by \(\iota\). We may perform the same construction with another indeterminate \(x'\), and consider the \(\kk[x,x'] \otimes R_{2\delta}\)-modules \(L_{\delta,i,x} \circ  L_{\delta,j,x'}\) and \(L_{\delta,j,x'} \circ  L_{\delta,i,x}\). There is a nonzero homomorphism \(r_{x,x'}^{i,j}: L_{\delta,i,x} \circ  L_{\delta,j,x'} \to L_{\delta,j,x'} \circ  L_{\delta,i,x}\) defined in terms of certain intertwining elements of \(R_{2\delta}\). Then \(r^{i,j}\) is equal to
\begin{align}\label{rij}
r^{i,j} :=[(x-x')^{-s} r_{x,x'}^{i,j}]_{x=x'=0},
\end{align}
where \(s\) is maximal such that \(r^{i,j}_{x,x'}(L_{\delta,i,x} \circ L_{\delta,j,x'}) \subset(x-x')^sL_{\delta,j,x'} \circ L_{\delta,i,x}\).

For any  \(i \in I'\), let \(\bar v_i \in L_{\delta,i}\) be the image of \(v_i\) in the quotient \(\Delta_{\delta,i} \twoheadrightarrow L_{\delta,i}\). Writing $\bar v_{i,j}$ for $1\otimes \bar v_i \otimes \bar v_j\in L_{\de,i}\circ L_{\de,j}$, it can be seen as in \cite[Proposition 8.2.1]{KM} that
\begin{align}\label{rijxx}
r_{x,x'}^{i,j}(\bar v_{i,j}) = (x-x')^{\kappa}\sigma \bar v_{j,i} + (x-x')^{\kappa -1}\sigma' \bar v_{j,i},
\end{align}
where \(\kappa = \sum_{a=1}^d \sum_{b=1}^d \delta_{\bb^i_a,\bb^j_b}\).

\begin{Lemma}\label{L150617}  
For \(t \in [1,n-1]\), the homomorphism \(\hat{r}_t\in\End_{C_{n\de}}( \Delta_\delta^{\circ n})\) satisfies 
\begin{align*}
\hat{r}_t(v_{\bi})=(1 \hspace{-0.3mm} \otimes \hspace{-0.3mm} \cdots\hspace{-0.3mm}   \otimes\hspace{-0.3mm} 1 \hspace{-0.3mm}\otimes\hspace{-0.3mm}  (\sigma +\delta_{i_t,i_{t+1}}\xi_{i_t}) \hspace{-0.3mm} \otimes\hspace{-0.3mm} 1 \hspace{-0.3mm}\otimes \cdots \hspace{-0.3mm} \otimes \hspace{-0.3mm} 1)v_{s_{t}\bi},
\end{align*}
where \((\sigma +\delta_{i_t,i_{t+1}}\xi_{i_t})\) is inserted in the \((t,t+1)\)th slots, for all \(\bi \in (I')^n\).
\end{Lemma}

\begin{proof}
Let \(i,j \in I'\). All \(y's\) and \(\psi\)'s of positive degree act as zero on \(L_{\delta,i}\) and \(L_{\delta,j}\) since these modules are concentrated in degree zero. So by Lemma \ref{sigmaprime}, we have
\( \sigma' \bar v_{j,i}= \delta_{i,j}\xi_i(x-x')\bar v_{j,i}\). Thus, (\ref{rij}) and (\ref{rijxx}) give \(r^{i,j}(\bar v_{i,j}) = (\sigma + \xi_i \delta_{i,j})\bar v_{j,i}\).

It may be seen via Theorem \ref{Cbasis} and word/degree considerations that \((1_{\bb^i\bb^j}\Delta_{\delta,j} \circ \Delta_{\delta,i})_0\) has basis \(\{v_{i,i}, \sigma v_{i,i}\}\) if \(i=j\), and \(\{\sigma v_{j,i}\}\) if \(i \neq j\). Thus Lemma \ref{DelLift}(i) implies that \(\hat{r}_1(v_{i,j}) = (\sigma+\delta_{i,j}\xi_i)v_{j,i}\). The result for general \(n\) follows from this case.
\end{proof}

\subsection{\boldmath Commutation relations in \(\End_{C_{n\delta}}(\Delta_\delta^{\circ n})\)} Now we examine commutation relations between elements of \(\End_{C_{n\delta}}(\Delta_\delta^{\circ n})\). We will use the following generator of \( \Delta_\delta^{\circ n}\):
$$
v_{\delta, \ldots, \delta} :=1 \otimes v_\delta \otimes \cdots \otimes v_\delta = \sum_{\bi \in (I')^n} v_\bi.
$$

\begin{Lemma}\label{scommute}
The following relations hold in \(\End_{C_{n \delta}}(\Delta_\delta^{\circ n})\):
\begin{align}\label{End0}
\hat{r}_t\circ e_{\bi} =   e_{s_t\bi} \circ \hat{r}_t,
\end{align}
\begin{align}\label{End1}
(\hat{r}_t \circ a^{i,j}_u- a^{i,j}_{s_t(u)} \circ \hat{r}_t) \circ e_{\bi} = 0,
\end{align}
\begin{align}\label{End2}
(\hat{r}_t \circ c_u- c_{s_t(u)} \circ \hat{r}_t) \circ e_{\bi} = 0,
\end{align}
\begin{align}\label{End3}
(\hat{r}_t  \circ z_u - z_{s_t(u)} \circ  \hat{r}_t) \circ e_{\bi}=
\begin{cases}
(\delta_{u,t}-\delta_{u,t+1})(c_t + c_{t+1}) \circ e_{\bi} & \textup{if }i_t=i_{t+1};\\
(\delta_{u,t}-\delta_{u,t+1}) a_t^{i_{t+1},i_t} \circ a_{t+1}^{i_t,i_{t+1}} \circ e_{\bi}
&
\textup{if }{\tt c}_{i_t, i_{t+1}}=-1;\\
0 & \textup{otherwise},
\end{cases}
\end{align}
for all \(t \in [1,n-1]\), \(u \in [1,n]\), and \(\bi \in (I')^n\).
\end{Lemma}
\begin{proof}
It is enough to check these relations in the case \(n=2\). We note that (\ref{End0}) holds by construction of the map \(\hat{r}_t\).  For \(i,j,m \in I'\) such that \({\tt c}_{i,j} = -1\), we have
\begin{align*}
\hat{r}_1 \circ a^{i,j}_1 \circ e_{j,m}(v_{\delta,\delta}) &= (\psi_{j,i} \otimes 1)(\sigma + \delta_{i,m}\xi_i)v_{m,i}
= (\sigma + \delta_{j,m} \xi_j)(1 \otimes \psi_{j,i})v_{m,i}\\
&=a^{i,j}_2 \circ \hat{r}_1 \circ e_{j,m}(v_{\delta,\delta}),
\end{align*}
where Lemma~\ref{L150617} has been applied for the first equality and Lemma \ref{psisigma} has been applied for the second equality. Thus (\ref{End1}) holds when \(u=1\). Since \(\hat{r}_1^2=1\), the claim also holds for \(u=2\), completing the proof of (\ref{End1}).

The relation (\ref{End2}) follows from (\ref{End1}) when \({\tt C} \neq {\tt A}_1^{(1)}\) since \(c_t\) may be expressed in terms of \(a^{ij}_t\)'s. When \({\tt C} = {\tt A}_1^{(1)}\), we have
\begin{align*}
\hat{r}_1 \circ c_1 \circ e_{11}(v_{\delta,\delta}) &= [(y_1 - y_2) \otimes 1](\sigma +1)v_{1,1}\\
 c_2 \circ \hat{r}_1 \circ e_{11}(v_{\delta,\delta}) &=(\sigma +1)[1 \otimes (y_1 - y_2)]v_{1,1}.
\end{align*}
The equality of these expressions is easily verified by direct application of KLR relations.

For relation (\ref{End3}), we have
\begin{align*}
(\hat{r}_1 \circ z_1\circ e_{ji} - z_2 \circ \hat{r}_1 &\circ e_{ji})(v_{\delta,\delta})=[(y_1 \otimes 1)(\sigma + \delta_{i,j}\xi_j) - (\sigma + \delta_{i,j}\xi_j)(1 \otimes y_1)]v_{i,j}\\
&=[\sigma(1 \otimes y_1) - \sigma' + \delta_{i,j}\xi_j (y_1 \otimes 1) -(\sigma + \delta_{i,j}\xi_j)(1 \otimes y_1)]v_{i,j}\\
&=[- \sigma' + \delta_{i,j}\xi_j (y_1 \otimes 1) - \delta_{i,j}\xi_j(1 \otimes y_1)]v_{i,j},
\end{align*}
after applying KLR relation (\ref{KLRypsi}) to write \((y_1 \otimes 1) \sigma 1_{\bb^i \bb^j} = (\sigma(1 \otimes y_1)-\sigma')1_{\bb^i \bb^j}\). 

Therefore, when \(i=j\), we have by Lemma \ref{sigmaprime} that
\begin{align*}
(\hat{r}_1 \circ z_1\circ e_{ji} - &z_2 \circ \hat{r}_1 \circ e_{ji})(v_{\delta,\delta})\\
&=\xi_i [-(y_d \otimes 1) - (1 \otimes (y_d-2y_1)) +  (y_1 \otimes 1) - (1 \otimes y_1) ] v_{i,i}\\
&= \xi_i[((y_1-y_d) \otimes 1) + (1 \otimes (y_1-y_d))]v_{i,i}\\
&= (c_1 + c_2) \circ e_{ii} (v_{\delta,\delta})
\end{align*}
On the other hand, if \({\tt c}_{i,j} = -1\), then Lemma \ref{sigmaprime} gives us
\begin{align*}
(\hat{r}_1 \circ z_1\circ e_{ji} - &z_2 \circ \hat{r}_1 \circ e_{ji})(v_{\delta,\delta})
=-\xi_i \varepsilon_{ij} (\psi_{j,i} \otimes \psi_{i,j})v_{i,j}
= \xi_i \varepsilon_{ji} (\psi_{j,i} \otimes \psi_{i,j})v_{i,j}\\
&=[\mu_{ij}(1 \otimes \psi_{i,j})] [\mu_{ji}(\psi_{j,i} \otimes 1)]v_{i,j}
= a_1^{i,j} \circ a_2^{j,i} \circ e_{ji} (v_{\delta,\delta}),
\end{align*}
applying Lemma \ref{ximu} for the third equality. The case \({\tt c}_{i,j} = 0\) follows immediately from Lemma \ref{sigmaprime}. 
Thus relation (\ref{End3}) holds for \(u=1\), and the case \(u=2\) follows from (\ref{End0}), (\ref{End1}) and (\ref{End2}). 
\end{proof}

Now we define some convenient notation for elements of \(\End_{C_{n \de}}(\Delta_\delta^{\circ n})\). For \(w =s_{t_1} \cdots s_{t_m} \in \mathfrak{S}_n\), define \(\hat{r}_w := \hat{r}_{t_1}\circ  \cdots\circ \hat{r}_{t_m} \in \End_{C_{n\delta}}(\Delta_\delta^{\circ n})_0\). By Lemma \ref{DelLift}(ii) this definition is independent of reduced expression for \(w\). 
For \(\bt \in \ZZ_{\geq 0}^n\) and \(\bu \in \{0,1\}^n\), we set
\begin{align*}
z^{\bt} := z_1^{\circ t_1} \circ \cdots \circ z_n^{\circ t_n},
\hspace{10mm}
c^{\bu} := c_1^{ \circ u_1} \circ \cdots \circ c_n^{ \circ u_n}.
\end{align*}

For \(i,j \in I'\), we say \(i\) and \(j\) are {\em connected} if \(i=j\) or \({\tt c}_{i,j} = -1\). We say \(\bi, \bj \in (I')^n\) are {\em connected} if \(i_r\) and \(j_r\) are connected for all \(r \in [1,n]\).
For connected \(i,j \in I'\) and \(r \in [1,n]\), let \(\hat{a}_r^{i,j} \in \End_{C_{n\de}}(\Delta_\delta^{\circ n})\) be defined
\begin{align*}
\hat{a}_r^{i,j} := \begin{cases}
a^{i,j}_r & \textup{if } {\tt c}_{i,j} = -1;\\
\textup{id} & \textup{if } i=j.
\end{cases}
\end{align*}
For connected \(\bi, \bj \in (I')^n\), set 
\begin{align*}
a^{\bi, \bj} := \hat{a}^{i_1, j_1}_1 \circ \cdots \circ \hat{a}^{i_n,j_n}_n.
\end{align*}

\begin{Lemma}\label{EndBasis}
The algebra \(\End_{C_{n\delta}}(\Delta_\delta^{\circ n})\) has basis
\begin{align}\label{basim}
\{z^{\bt} \circ c^{\bu} \circ a^{\bi, w\bj} \circ\hat{r}_w \circ e_{\bj}\},
\end{align}
ranging over \(w \in \mathfrak{S}_{n}\), \(\bt \in \ZZ_{\geq 0}^n\), \(\bi,\bj \in (I')^n\) such that \(\bi\) and \(w\bj\) are connected, and \(\bu \in \{0,1\}^n\) such that \(u_r = 0 \) if \(i_r \neq (w \bj)_r\).
\end{Lemma}

\begin{proof} First, we argue that the elements of (\ref{basim}) are linearly independent. This argument proceeds along similar lines to the proof of \cite[Theorem 4.2.1]{KM}. 

Let \(\mathscr{D}^{nd}_{d,\ldots,d}\) be a set of minimal left coset representatives in \(\mathfrak{S}_{nd}/\mathfrak{S}_d \times \cdots \times \mathfrak{S}_d\). Note that, by the KLR basis theorem \cite[Theorem 2.5]{KL1},  \(\Delta_\delta^{\circ n}\) has \(\kk\)-basis \(\{\psi_w v\}\), where \(w \in \mathscr{D}^{nd}_{d,\ldots,d}\), and \(v\) ranges over basis elements for \(\Delta_\delta^{\boxtimes n}\).

For \(w \in \mathfrak{S}_n\), define the block permutation \(\textup{bl}(w) \in \mathfrak{S}_{nd}\) by 
\begin{align*}
\textup{bl}(w)(a) = w(\lceil a/d\rceil)d  + (a-1 \textup{ mod }d) - d +1.
\end{align*}
Diagrammatically speaking, \(\textup{bl}(w)\) is achieved by replacing each strand of the diagram of \(w\) by \(d\) parallel strands. E.g., \(\textup{bl}(s_1) = \sigma \in \mathfrak{S}_{2d}\).

Let \(\bj \in (I')^n\), and assume \(w = s_{t_1} \cdots s_{t_m}\). Then, by the definition of \(\hat{r}_w\), we have
\begin{align*}
\hat{r}_w(v_\bj) &= (\psi_{\textup{bl}(s_{t_m})} + \kappa_m) \cdots (\psi_{\textup{bl}(s_{t_1})} + \kappa_1)v_{w \bj},
\end{align*}
for some constants \(\kappa_1, \ldots, \kappa_m \in \kk\). Recombining terms, this may be written in turn as
\begin{align*}
\psi_{\textup{bl}(w^{-1})}v_{w\bj} + (*),
\end{align*}
where \((*)\) is a \(\kk\)-linear combination of terms of the form \(\psi_{w'}v\), where \(v \in \Delta_\delta^{\boxtimes n}\), and \(w' \in \mathscr{D}^{nd}_{d,\ldots,d}\) is such that \(\ell(w') < \ell(\textup{bl}(w^{-1}))\).

Then, considering the action of an element of (\ref{basim}) on \(\Delta_\delta^{\circ n}\), we have
\begin{equation*}
z^{\bt} \circ c^{\bu} \circ a^{\bi, w\bj} \circ\hat{r}_w \circ e_{\bj}(v_{\delta, \ldots, \delta}) =\hspace{70mm}
\end{equation*}
\begin{equation*}
\hspace{15mm}\psi_{\textup{bl}(w^{-1})}(y_1^{t_1}(y_1-y_d)^{u_1}\psi_{(w\bj)_1,i_1} \otimes \cdots \otimes y_n^{t_n}(y_1-y_d)^{u_n}\psi_{(w\bj)_n,i_n})   v_{\bi}+(*),
\end{equation*}
where  \((*)\) is again a \(\kk\)-linear combination of terms of the form \(\psi_{w'}v\), where \(v \in \Delta_\delta^{\boxtimes n}\), and \(w' \in \mathscr{D}^{nd}_{d,\ldots,d}\) is such that \(\ell(w') < \ell(\textup{bl}(w^{-1}))\). Thus, by Lemma \ref{MinDelHom} and induction on the word length of \(w\), it can be shown that the elements (\ref{basim}) form a linearly independent set of endomorphisms. Now, comparing graded dimension of the set (\ref{basim}) with Lemma \ref{Endcircdim} proves the result.
\end{proof}

Theorem \ref{EndBasis} has the immediate corollary:

\begin{Corollary}\label{EndGens}
The algebra \(\End_{C_{n \de}}(\Delta_\delta^{\circ n})\) is generated by the homomorphisms 
\begin{align*}
\{e_\bi \mid \bi \in (I')^n\}\ \cup\ 
\{c_r, z_r,a^{i,j}_r \mid r \in [1,n],\, i,j \in I' \textup{ with }{\tt c}_{i,j} = -1\}\ \cup\ 
\{\hat{r}_t \mid t \in [1,n-1]\}.
\end{align*}
\end{Corollary}


\subsection{Proof of the Main Theorem}\label{affzig}
Now we prove the main result of the paper (which appears as Theorem A in the introductory section):
\begin{Theorem}\label{mainthm}
The map \(\varphi: \textup{\(\Zig\)}^\aff_n({\tt C}') \to \End_{C_{n\delta}}(\Delta^{\circ n}_\delta)\), defined on generators by
\begin{align*}
e_\bi \mapsto e_\bi, 
\hspace{10mm}
a_t^{i,j} \mapsto a_t^{i,j},
\hspace{10mm}
c_t \mapsto c_t,
\hspace{10mm}
s_u \mapsto \hat{r}_u,
\hspace{10mm}
z_t \mapsto z_t,
\end{align*}
for all \(t \in [1,n]\), \(u \in [1,n-1]\), \(\bi \in (I')^n\), and \(i,j \in I'\) such that \({\tt c}_{i,j}=-1\), is an isomorphism of graded \(\kk\)-algebras.
\end{Theorem}
\begin{proof}
By Corollary \ref{ExpEndZig}, Lemma \ref{DelLift}(ii), and Lemma \ref{scommute}, the images of the generators obey the defining relations of \(\Zig_n^\aff\) as presented in Lemma \ref{zigpresdef}. Hence \(\varphi\) defines a graded \(\kk\)-algebra homomorphism. Moreover, \(\varphi\) surjects onto the generators of \(\End_{C_{n\delta}}(\Delta_\delta^{\circ n})\) by Corollary \ref{EndGens}, so it follows that \(\varphi\) is an isomorphism by comparison of the graded dimensions in Lemma \ref{affzigdim} and Lemma \ref{Endcircdim}.
\end{proof}

\begin{Corollary}\label{maincor}
If $\kk$ is a field of characteristic \(p=0\) or \(p>n\), then \(B_{n \delta}\) is Morita equivalent to \(\textup{\(\Zig\)}_n^\aff({\tt C}')\).
\end{Corollary}
\begin{proof}
In this situation the module \(\Delta_\delta^{\circ n}\) is a projective generator for \(B_{n\delta}\), see \cite[Lemma 6.22]{KMStrat}, so \(B_{n \delta}\) is Morita equivalent to \(\End_{C_{n\delta}}(\Delta^{\circ n}_\delta)^\textup{op} \cong (\Zig_n^\aff)^\textup{op}\). Then the result follows by Lemma \ref{affzigop}.
\end{proof}

\section{Appendix}\label{diagproofs}

This section is devoted to proving Lemmas \ref{sigmaprime} and \ref{psisigma}, which are crucial in determining the commutation relations among generating endomorphisms of \(\Delta_\delta^{\circ n}\). In all cases, the approach to proving these lemmas is similar:
\begin{enumerate}
\item Every element of \(\Delta_{\delta,i} \circ \Delta_{\delta,j}\) should be written as a linear combination of terms of the form \(\psi_w(x_1 \otimes x_2)v_{i,j}\), where \(x_1,x_2 \in R_\delta\), and \(w\) is a minimal left coset representative for \(\mathfrak{S}_{2d} / \mathfrak{S}_d \times \mathfrak{S}_d\). Diagrammatically speaking, this is a matter of moving beads, and crossings of strands which originate from the same side, to the top of the diagram by applying KLR relations.
\item Once all terms are rewritten as in (i), use Lemmas \ref{Cdelfacts} through \ref{Cypsi} to simplify the expressions \((x_1 \otimes x_2)v_{i,j}\), rewriting these elements of \(\Delta_{\delta,i} \boxtimes \Delta_{\delta,j}\) in the form of the basis in Theorem \ref{Cbasis}.
\end{enumerate}

We have written a Sage program which performs steps (i) and (ii), and have used this algorithm to verify Lemmas \ref{sigmaprime} and \ref{psisigma} in the exceptional cases of type \({\tt E}_\ell^{(1)}\). This program is available upon request. In the following proofs we assume \({\tt C}\) is of type \({\tt A}_\ell^{(1)}\) or \({\tt D}_\ell^{(1)}\).\\

\noindent{\bf Lemma \ref{sigmaprime}.}
{\em
Let \(i,j \in I'\), and recall that \(v_{i,j}\) is a generator for \(\Delta_{\delta,i} \circ \Delta_{\delta,j}\). Then we have
\begin{align*}
\sigma'v_{i,j}= 
\begin{cases}
\xi_i [y_d \otimes 1 + 1 \otimes (y_d-2y_1)] v_{i,i}
&
\textup{if }i=j;\\
\xi_i \varepsilon_{ij} (\psi_{j,i} \otimes \psi_{i,j})v_{i,j}
&
\textup{if }{\tt c}_{i,j}=-1;\\
0
&
\textup{otherwise}.
\end{cases}
\end{align*}
}
\begin{proof}
\underline{Case \(i=j\), \({\tt C} = {\tt A}^{(1)}_\ell\)}. If \(\ell=1\), the result is easily checked. Assume \(\ell \geq 2\). We depict \(\sigma' v_{i,i}\) diagrammatically, where \(v_{i,i}\) is conceived to be at the top of the diagram:
\begin{align*}
\begin{braid}\tikzset{baseline=0mm}
\draw(0,2) node[above]{$0$}--(0,-2);
\draw(1,2) node[above]{$1$}--(8,-2);
\draw(2,2) node[above]{$\cdots$};
\draw(3,2) node[above]{$i\hspace{-0.8mm}-\hspace{-0.8mm}1$}--(10,-2);
\draw(4,2) node[above]{$\ell$}--(11,-2);
\draw(5,2) node[above]{$\cdots$};
\draw(6,2) node[above]{$i$}--(13,-2);
\draw(7,2) node[above]{$0$}--(3.5,0)--(7,-2);
\draw(8,2) node[above]{$1$}--(1,-2);
\draw(9,2) node[above]{$\cdots$};
\draw(10,2) node[above]{$i\hspace{-0.8mm}-\hspace{-0.8mm}1$}--(3,-2);
\draw(11,2) node[above]{$\ell$}--(4,-2);
\draw(12,2) node[above]{$\cdots$};
\draw(13,2) node[above]{$i$}--(6,-2);
\end{braid}.
\end{align*}
We now move crossings up, when possible, to act on the individual factors \(\Delta_{\delta,i}\), and use Lemmas \ref{wordfacts} and \ref{Cpsi} to recognize when these terms are zero. Applying the braid relation to the
\(
\begin{braid}\tikzset{baseline=0mm}
\draw(0,0.5) node[above]{$1$}--(2,-0.5);
\draw(1,0.5) node[above]{$0$}--(0,0)--(1,-0.5);
\draw(2,0.5) node[above]{$1$}--(0,-0.5);
\end{braid}
\)
braid, we see that the 
\(
\begin{braid}\tikzset{baseline=0mm}
\draw(0,0.5) node[above]{$1$}--(2,-0.5);
\draw(1,0.5) node[above]{$0$}--(2,0)--(1,-0.5);
\draw(2,0.5) node[above]{$1$}--(0,-0.5);
\end{braid}
\)
term allows for the \((0,1)\)-crossing to move up to act on \(\Delta_{\delta,i}\) as zero, leaving only the remainder term \(\varepsilon_{01}\begin{braid}\tikzset{baseline=0mm}
\draw(0,0.5) node[above]{$1$}--(0,-0.5);
\draw(1,0.5) node[above]{$0$}--(1,0)--(1,-0.5);
\draw(2,0.5) node[above]{$1$}--(2,-0.5);
\end{braid}
\).
This behavior will occur frequently enough that we will merely say that the \((i,i+1,i)\)-braid `opens'. Indeed, the \((1,0,1)\)- through \((i-1,i-2,i-1)\)-braids open in succession, giving:
\begin{align*}
\begin{braid}\tikzset{baseline=0mm}
\draw(0,2) node[above]{$0$}--(0,-2);
\draw(1,2) node[above]{$1$}--(1,-2);
\draw(2,2) node[above]{$\cdots$};
\draw(3,2) node[above]{$i\hspace{-0.8mm}-\hspace{-0.8mm}1$}--(3,-2);
\draw(4,2) node[above]{$\ell$}--(12,-2);
\draw(5,2) node[above]{$\cdots$};
\draw(6,2) node[above]{$i\hspace{-0.8mm}+\hspace{-0.8mm}1$}--(14,-2);
\draw(7,2) node[above]{$i$}--(13,0)--(15,-2);
\draw(8,2) node[above]{$0$}--(4,0)--(8,-2);
\draw(9,2) node[above]{$1$}--(13,1)--(11,0)--(13,-1)--(9,-2);
\draw(10,2) node[above]{$\cdots$};
\draw(11,2) node[above]{$i\hspace{-0.8mm}-\hspace{-0.8mm}1$}--(13.5,1)--(11.5,0)--(13.5,-1)--(11,-2);
\draw(12,2) node[above]{$\ell$}--(4,-2);
\draw(13,2) node[above]{$\cdots$};
\draw(14,2) node[above]{$i\hspace{-0.8mm}+\hspace{-0.8mm}1$}--(6,-2);
\draw(15,2) node[above]{$i$}--(13,0)--(7,-2);
\draw(7.5,-2) node[below]{$\varepsilon_{01} \cdots \varepsilon_{i-2,i-1}$};
\end{braid}
=
\begin{braid}\tikzset{baseline=0mm}
\draw(0,2) node[above]{$0$}--(0,-2);
\draw(1,2) node[above]{$1$}--(1,-2);
\draw(2,2) node[above]{$\cdots$};
\draw(3,2) node[above]{$i\hspace{-0.8mm}-\hspace{-0.8mm}1$}--(3,-2);
\draw(4,2) node[above]{$\ell$}--(4,-2);
\draw(5,2) node[above]{$\cdots$};
\draw(6,2) node[above]{$i\hspace{-0.8mm}+\hspace{-0.8mm}1$}--(6,-2);
\draw(7,2) node[above]{$i$}--(15,-2);
\draw(8,2) node[above]{$0$}--(7,0)--(8,-2);
\draw(9,2) node[above]{$1$}--(7.5,0)--(9,-2);
\draw(10,2) node[above]{$\cdots$};
\draw(11,2) node[above]{$i\hspace{-0.8mm}-\hspace{-0.8mm}1$}--(8.5,0)--(11,-2);
\draw(12,2) node[above]{$\ell$}--(9,0)--(12,-2);
\draw(13,2) node[above]{$\cdots$};
\draw(14,2) node[above]{$i\hspace{-0.8mm}+\hspace{-0.8mm}1$}--(10,0)--(14,-2);
\draw(15,2) node[above]{$i$}--(7,-2);
\draw(7.5,-2) node[below]{$\varepsilon_{01} \cdots \varepsilon_{i-2,i-1}\varepsilon_{i+2,i+1} \cdots \varepsilon_{\ell,\ell-1}\varepsilon_{0,\ell}$};
\end{braid},
\end{align*}
after the \((\ell,0,\ell)\)-braid opens, followed by the \((\ell-1,\ell,\ell-1)\)- through \((i+2,i+1,i+2)\)-braids in succession. Now, applying the \((i,i+1,i)\)-braid relation, this is equal to
\begin{align}\label{twodiags}
\begin{braid}\tikzset{baseline=0mm}
\draw(0,2) node[above]{$0$}--(0,-2);
\draw(1,2) node[above]{$1$}--(1,-2);
\draw(2,2) node[above]{$\cdots$};
\draw(3,2) node[above]{$i\hspace{-0.8mm}-\hspace{-0.8mm}1$}--(3,-2);
\draw(4,2) node[above]{$\ell$}--(4,-2);
\draw(5,2) node[above]{$\cdots$};
\draw(6,2) node[above]{$i\hspace{-0.8mm}+\hspace{-0.8mm}1$}--(6,-2);
\draw(7,2) node[above]{$i$}--(15,-2);
\draw(8,2) node[above]{$0$}--(7,0)--(8,-2);
\draw(9,2) node[above]{$1$}--(8,0)--(9,-2);
\draw(10,2) node[above]{$\cdots$};
\draw(11,2) node[above]{$i\hspace{-0.8mm}-\hspace{-0.8mm}1$}--(10,0)--(11,-2);
\draw(12,2) node[above]{$\ell$}--(12,-2);
\draw(13,2) node[above]{$\cdots$};
\draw(14,2) node[above]{$i\hspace{-0.8mm}+\hspace{-0.8mm}1$}--(14,-2);
\draw(15,2) node[above]{$i$}--(7,-2);
\draw(7.5,-2) node[below]{$\varepsilon_{01} \cdots \varepsilon_{i-2,i-1}\varepsilon_{i+2,i+1} \cdots \varepsilon_{\ell,\ell-1}\varepsilon_{0,\ell}$};
\end{braid}
+
\begin{braid}\tikzset{baseline=0mm}
\draw(0,2) node[above]{$0$}--(0,-2);
\draw(1,2) node[above]{$1$}--(1,-2);
\draw(2,2) node[above]{$\cdots$};
\draw(3,2) node[above]{$i\hspace{-0.8mm}-\hspace{-0.8mm}1$}--(3,-2);
\draw(4,2) node[above]{$\ell$}--(4,-2);
\draw(5,2) node[above]{$\cdots$};
\draw(6,2) node[above]{$i\hspace{-0.8mm}+\hspace{-0.8mm}1$}--(6,-2);
\draw(7,2) node[above]{$i$}--(11,0)--(7,-2);
\draw(8,2) node[above]{$0$}--(7,0)--(8,-2);
\draw(9,2) node[above]{$1$}--(8,0)--(9,-2);
\draw(10,2) node[above]{$\cdots$};
\draw(11,2) node[above]{$i\hspace{-0.8mm}-\hspace{-0.8mm}1$}--(10,0)--(11,-2);
\draw(12,2) node[above]{$\ell$}--(12,-2);
\draw(13,2) node[above]{$\cdots$};
\draw(14,2) node[above]{$i\hspace{-0.8mm}+\hspace{-0.8mm}1$}--(14,-2);
\draw(15,2) node[above]{$i$}--(15,-2);
\draw(7.5,-2) node[below]{$\varepsilon_{01} \cdots \varepsilon_{i-2,i-1}\varepsilon_{i+1,i}\varepsilon_{i+2,i+1} \cdots \varepsilon_{\ell,\ell-1}\varepsilon_{0,\ell}$};
\end{braid}
\end{align}
In the left term in (\ref{twodiags}), the \((i,i-1,i)\)-braid opens, introducing an \((i+1,i)\)-double crossing, which opens to give
\begin{align*}
-\xi_i[1 \otimes (y_{d-1}-y_d)]v_{i,i} = -\xi_i [1 \otimes (y_1-y_d)]v_{i,i}.
\end{align*}
In the right term in (\ref{twodiags}), the \((i,i-1)\)-double crossing opens to give
\begin{align*}
\xi_i[y_d \otimes 1 - 1 \otimes y_i]v_{i,i}  = \xi_i[y_d\otimes 1 - 1 \otimes y_1]v_{i,i},
\end{align*}
proving the claim.

\underline{Case \(i=j\), \({\tt C} = {\tt D}^{(1)}_\ell\), \(1\leq i \leq \ell-2\)}. We depict \(\sigma' v_{i,i}\) diagrammatically:
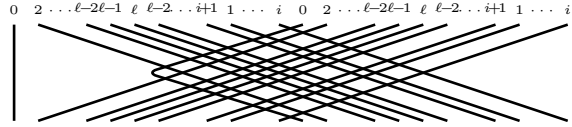
\begin{align*}
 \begin{braid}\tikzset{baseline=0mm}
\draw(0,2) node[above]{$0$}--(0,-2);
\draw(1,2) node[above]{$2$}--(13,-2);
\draw(2,2) node[above]{$\cdots$};
\draw(3,2) node[above]{$\ell\hspace{-1mm}-\hspace{-1mm}2$}--(15,-2);
\draw(4,2) node[above]{$\ell\hspace{-1mm}-\hspace{-1mm}1$}--(16,-2);
\draw(5,2) node[above]{$\ell$}--(17,-2);
\draw(6,2) node[above]{$\ell\hspace{-1mm}-\hspace{-1mm}2$}--(18,-2);
\draw(7,2) node[above]{$\cdots$};
\draw(8,2) node[above]{$i\hspace{-1mm}+\hspace{-1mm}1$}--(20,-2);
\draw(9,2) node[above]{$1$}--(21,-2);
\draw(10,2) node[above]{$\cdots$};
\draw(11,2) node[above]{$i$}--(23,-2);
\draw(12,2) node[above]{$0$}--(5.5,0)--(12,-2);
\draw(13,2) node[above]{$2$}--(1,-2);
\draw(14,2) node[above]{$\cdots$};
\draw(15,2) node[above]{$\ell\hspace{-1mm}-\hspace{-1mm}2$}--(3,-2);
\draw(16,2) node[above]{$\ell\hspace{-1mm}-\hspace{-1mm}1$}--(4,-2);
\draw(17,2) node[above]{$\ell$}--(5,-2);
\draw(18,2) node[above]{$\ell\hspace{-1mm}-\hspace{-1mm}2$}--(6,-2);
\draw(19,2) node[above]{$\cdots$};
\draw(20,2) node[above]{$i\hspace{-1mm}+\hspace{-1mm}1$}--(8,-2);
\draw(21,2) node[above]{$1$}--(9,-2);
\draw(22,2) node[above]{$\cdots$};
\draw(23,2) node[above]{$i$}--(11,-2);
\end{braid}.
\end{align*}
We begin dragging the \(0\)-strand to the right to simplify the diagram. The \((2,0,2)\)-braid opens, followed by the \((3,2,3)\)- through \((\ell-1,\ell-2,\ell-1)\)-braids in succession. Then the \((\ell-2,\ell,\ell-2)\)-braid opens, followed by the \((\ell-3,\ell-2,\ell-3)\)- through \((i+1, i+2, i+1)\)-braids in succession, giving (excluding straight strands on the left):
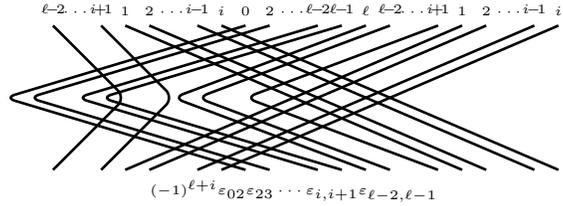
\begin{align*}
\begin{braid}\tikzset{baseline=0mm}
\draw(6,3) node[above]{$\ell\hspace{-1.1mm}-\hspace{-1mm}2$}--(9,0)--(6,-3);
\draw(7,3) node[above]{$\cdots$};
\draw(8,3) node[above]{$i\hspace{-1mm}+\hspace{-1mm}1$}--(11,0)--(8,-3);
\draw(9,3) node[above]{$1$}--(23,-3);
\draw(10,3) node[above]{$2$}--(24,-3);
\draw(11,3) node[above]{$\cdots$};
\draw(12,3) node[above]{$i\hspace{-1mm}-\hspace{-1mm}1$}--(26,-3);
\draw(13,3) node[above]{$i$}--(27,-3);
\draw(14,3) node[above]{$0$}--(4,0)--(14,-3);
\draw(15,3) node[above]{$2$}--(5,0)--(15,-3);
\draw(16,3) node[above]{$\cdots$};
\draw(17,3) node[above]{$\ell\hspace{-1mm}-\hspace{-1mm}2$}--(7,0)--(17,-3);
\draw(18,3) node[above]{$\ell\hspace{-1mm}-\hspace{-1mm}1$}--(8,0)--(18,-3);
\draw(19,3) node[above]{$\ell$}--(11,0)--(19,-3);
\draw(20,3) node[above]{$\ell\hspace{-1.1mm}-\hspace{-1mm}2$}--(12,0)--(20,-3);
\draw(21,3) node[above]{$\cdots$};
\draw(22,3) node[above]{$i\hspace{-1mm}+\hspace{-1mm}1$}--(14,0)--(22,-3);
\draw(23,3) node[above]{$1$}--(9,-3);
\draw(24,3) node[above]{$2$}--(10,-3);
\draw(25,3) node[above]{$\cdots$};
\draw(26,3) node[above]{$i\hspace{-1mm}-\hspace{-1mm}1$}--(12,-3);
\draw(27,3) node[above]{$i$}--(13,-3);
\draw(16,-3) node[below]{$(-1)^{\ell+i}\varepsilon_{02} \varepsilon_{23} \cdots \varepsilon_{i,i+1} \varepsilon_{\ell-2,\ell-1} $};
\end{braid}.
\end{align*}
Now the \((\ell-2,\ell-1)\)-double crossing opens, introducing a \((\ell-2, \ell-3,\ell-2)\)-braid which opens, followed by a \((\ell-2,\ell-3)\)-double crossing which opens. This sequence repeats until the \((i+2,i+1,i+2)\)-braid opens, followed by \((i+2,i+1)\)-double crossing which opens. Finally, the \((i+1,i,i+1)\)-braid opens, giving:
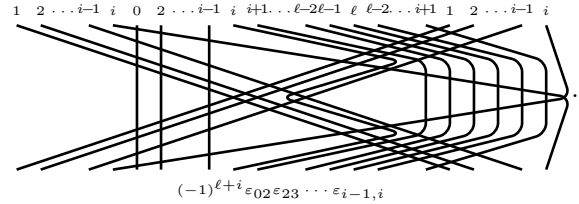
\begin{align*}
\begin{braid}\tikzset{baseline=0mm}
\draw(9,3) node[above]{$1$}--(27,-3);
\draw(10,3) node[above]{$2$}--(28,-3);
\draw(11,3) node[above]{$\cdots$};
\draw(12,3) node[above]{$i\hspace{-1mm}-\hspace{-1mm}1$}--(30,-3);
\draw(13,3) node[above]{$i$}--(32,0)--(31,-3);
\draw(14,3) node[above]{$0$}--(14,-3);
\draw(15,3) node[above]{$2$}--(15,-3);
\draw(16,3) node[above]{$\cdots$};
\draw(17,3) node[above]{$i\hspace{-1mm}-\hspace{-1mm}1$}--(17,-3);
\draw(18,3) node[above]{$i$}--(25,1.5)--(20,0)--(25,-1.5)--(18,-3);
\draw(19,3) node[above]{$i\hspace{-1mm}+\hspace{-1mm}1$}--(26,1.5)--(26,-1.5)--(19,-3);
\draw(20,3) node[above]{$\cdots$};
\draw(21,3) node[above]{$\ell\hspace{-1mm}-\hspace{-1mm}2$}--(27,1.5)--(27,-1.5)--(21,-3);
\draw(22,3) node[above]{$\ell\hspace{-1mm}-\hspace{-1mm}1$}--(28,1.5)--(28,-1.5)--(22,-3);
\draw(23,3) node[above]{$\ell$}--(29,1.5)--(29,-1.5)--(23,-3);
\draw(24,3) node[above]{$\ell\hspace{-1.1mm}-\hspace{-1mm}2$}--(30,1.5)--(30,-1.5)--(24,-3);
\draw(25,3) node[above]{$\cdots$};
\draw(26,3) node[above]{$i\hspace{-1mm}+\hspace{-1mm}1$}--(31,1.5)--(31,-1.5)--(26,-3);
\draw(27,3) node[above]{$1$}--(9,-3);
\draw(28,3) node[above]{$2$}--(10,-3);
\draw(29,3) node[above]{$\cdots$};
\draw(30,3) node[above]{$i\hspace{-1mm}-\hspace{-1mm}1$}--(12,-3);
\draw(31,3) node[above]{$i$}--(32,0)--(13,-3);
\draw(20,-3) node[below]{$(-1)^{\ell+i}\varepsilon_{02} \varepsilon_{23} \cdots \varepsilon_{i-1,i}$};
\end{braid}.
\end{align*}
Now the central \((i-1,i,i-1)\)- through \((1,2,1)\)-braids open in succession, and then \((i-2,i-1)\)- through \((1,2)\)-double crossings open in succession. Then the \((2,0,2)\)-braid opens, followed by the \((3,2,3)\)- through \((i-1,i-2,i-1)\)-braids opening in succession, giving (omitting straight strands on the left):
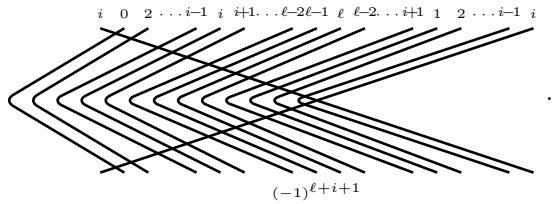
\begin{align*}
\begin{braid}\tikzset{baseline=0mm}
\draw(13,3) node[above]{$i$}--(31,-3);
\draw(14,3) node[above]{$0$}--(9,0)--(14,-3);
\draw(15,3) node[above]{$2$}--(10,0)--(15,-3);
\draw(16,3) node[above]{$\cdots$};
\draw(17,3) node[above]{$i\hspace{-1mm}-\hspace{-1mm}1$}--(11,0)--(17,-3);
\draw(18,3) node[above]{$i$}--(12,0)--(18,-3);
\draw(19,3) node[above]{$i\hspace{-1mm}+\hspace{-1mm}1$}--(13,0)--(19,-3);
\draw(20,3) node[above]{$\cdots$};
\draw(21,3) node[above]{$\ell\hspace{-1mm}-\hspace{-1mm}2$}--(14,0)--(21,-3);
\draw(22,3) node[above]{$\ell\hspace{-1mm}-\hspace{-1mm}1$}--(15,0)--(22,-3);
\draw(23,3) node[above]{$\ell$}--(16,0)--(23,-3);
\draw(24,3) node[above]{$\ell\hspace{-1.1mm}-\hspace{-1mm}2$}--(17,0)--(24,-3);
\draw(25,3) node[above]{$\cdots$};
\draw(26,3) node[above]{$i\hspace{-1mm}+\hspace{-1mm}1$}--(18,0)--(26,-3);
\draw(27,3) node[above]{$1$}--(19,0)--(27,-3);
\draw(28,3) node[above]{$2$}--(20,0)--(28,-3);
\draw(29,3) node[above]{$\cdots$};
\draw(30,3) node[above]{$i\hspace{-1mm}-\hspace{-1mm}1$}--(21,0)--(30,-3);
\draw(31,3) node[above]{$i$}--(13,-3);
\draw(22,-3) node[below]{$(-1)^{\ell+i+1}$};
\end{braid}.
\end{align*}
Now, applying the braid relation to the \((i,i-1,i)\)-braid gives
\begin{align*}
\begin{braid}\tikzset{baseline=0mm}
\draw(13,2) node[above]{$i$}--(27,0)--(31,-2);
\draw(14,2) node[above]{$0$}--(14,-2);
\draw(15,2) node[above]{$2$}--(15,-2);
\draw(16,2) node[above]{$\cdots$};
\draw(17,2) node[above]{$i\hspace{-1mm}-\hspace{-1mm}1$}--(17,-2);
\draw(18,2) node[above]{$i$}--(18,-2);
\draw(19,2) node[above]{$i\hspace{-1mm}+\hspace{-1mm}1$}--(19,-2);
\draw(20,2) node[above]{$\cdots$};
\draw(21,2) node[above]{$\ell\hspace{-1mm}-\hspace{-1mm}2$}--(21,-2);
\draw(22,2) node[above]{$\ell\hspace{-1mm}-\hspace{-1mm}1$}--(22,-2);
\draw(23,2) node[above]{$\ell$}--(23,-2);
\draw(24,2) node[above]{$\ell\hspace{-1.1mm}-\hspace{-1mm}2$}--(24,-2);
\draw(25,2) node[above]{$\cdots$};
\draw(26,2) node[above]{$i\hspace{-1mm}+\hspace{-1mm}1$}--(26,-2);
\draw(27,2) node[above]{$1$}--(28,0)--(27,-2);
\draw(28,2) node[above]{$2$}--(29,0)--(28,-2);
\draw(29,2) node[above]{$\cdots$};
\draw(30,2) node[above]{$i\hspace{-1mm}-\hspace{-1mm}1$}--(31,0)--(30,-2);
\draw(31,2) node[above]{$i$}--(27,0)--(13,-2);
\draw(22,-2) node[below]{$(-1)^{\ell+i+1}$};
\end{braid}
+
\begin{braid}\tikzset{baseline=0mm}
\draw(13,2) node[above]{$i$}--(31,0)--(13,-2);
\draw(14,2) node[above]{$0$}--(14,-2);
\draw(15,2) node[above]{$2$}--(15,-2);
\draw(16,2) node[above]{$\cdots$};
\draw(17,2) node[above]{$i\hspace{-1mm}-\hspace{-1mm}1$}--(17,-2);
\draw(18,2) node[above]{$i$}--(18,-2);
\draw(19,2) node[above]{$i\hspace{-1mm}+\hspace{-1mm}1$}--(19,-2);
\draw(20,2) node[above]{$\cdots$};
\draw(21,2) node[above]{$\ell\hspace{-1mm}-\hspace{-1mm}2$}--(21,-2);
\draw(22,2) node[above]{$\ell\hspace{-1mm}-\hspace{-1mm}1$}--(22,-2);
\draw(23,2) node[above]{$\ell$}--(23,-2);
\draw(24,2) node[above]{$\ell\hspace{-1.1mm}-\hspace{-1mm}2$}--(24,-2);
\draw(25,2) node[above]{$\cdots$};
\draw(26,2) node[above]{$i\hspace{-1mm}+\hspace{-1mm}1$}--(26,-2);
\draw(27,2) node[above]{$1$}--(27,-2);
\draw(28,2) node[above]{$2$}--(28,-2);
\draw(29,2) node[above]{$\cdots$};
\draw(30,2) node[above]{$i\hspace{-1mm}-\hspace{-1mm}1$}--(30,-2);
\draw(31,2) node[above]{$i$}--(31,-2);
\draw(22,-2) node[below]{$(-1)^{\ell+i+1}\varepsilon_{i-1,i}$};
\end{braid}
\end{align*}
In the term on the left, the \((i,i+1,i)\)-braid opens, introducing \((i-1,i)\)- and \((i,i+1)\)-double crossings which open, finally introducing an \((i,i-1,i)\)-braid which opens, giving
\begin{align*}
(-1)^{\ell+i}[1 \otimes (y_{d-1}-y_d)]v_{i,i}= (-1)^{\ell+i}[1 \otimes (y_{1}-y_d)]v_{i,i}
\end{align*}
In the term on the right, the \((i,i-1)\)-double crossing opens, followed by an \((i,i+1)\)-crossing opening, finally introducing an \((i,i-1,i)\)-braid which opens, giving
\begin{align*}
(-1)^{\ell+i+1}[y_d \otimes 1 - 1 \otimes y_1]v_{i,i},
\end{align*}
proving the statement.

\underline{Case \(i=j\), \({\tt C} = {\tt D}^{(1)}_\ell\), \(i=\ell,\ell-1\)}. We'll check the \(i=\ell\) case, the other case being similar. We depict \(\sigma' v_{i,i}\) diagrammatically:
\begin{align*}
\begin{braid}\tikzset{baseline=0mm}
\draw(0,2) node[above]{$0$}--(0,-2);
\draw(1,2) node[above]{$2$}--(11,-2);
\draw(2,2) node[above]{$\cdots$};
\draw(3,2) node[above]{$\ell\hspace{-1mm}-\hspace{-1mm}2$}--(13,-2);
\draw(4,2) node[above]{$\ell\hspace{-1mm}-\hspace{-1mm}1$}--(14,-2);
\draw(5,2) node[above]{$1$}--(15,-2);
\draw(6,2) node[above]{$2$}--(16,-2);
\draw(7,2) node[above]{$\cdots$};
\draw(8,2) node[above]{$\ell\hspace{-1mm}-\hspace{-1mm}2$}--(18,-2);
\draw(9,2) node[above]{$\ell$}--(19,-2);
\draw(10,2) node[above]{$0$}--(5,0)--(10,-2);
\draw(11,2) node[above]{$2$}--(1,-2);
\draw(12,2) node[above]{$\cdots$};
\draw(13,2) node[above]{$\ell\hspace{-1mm}-\hspace{-1mm}2$}--(3,-2);
\draw(14,2) node[above]{$\ell\hspace{-1mm}-\hspace{-1mm}1$}--(4,-2);
\draw(15,2) node[above]{$1$}--(5,-2);
\draw(16,2) node[above]{$2$}--(6,-2);
\draw(17,2) node[above]{$\cdots$};
\draw(18,2) node[above]{$\ell\hspace{-1mm}-\hspace{-1mm}2$}--(8,-2);
\draw(19,2) node[above]{$\ell$}--(9,-2);
\end{braid}.
\end{align*}
As in the last case, we begin by pulling the \(0\)-strand to the right. The \((2,0,2)\)-braid opens, then the \((3,2,3)\)- through \((\ell-1,\ell-2,\ell-1)\)-braids open in succession, giving (omitting straight strands on the left):
\begin{align*}
\begin{braid}\tikzset{baseline=0mm}
\draw(5,2) node[above]{$1$}--(15,-2);
\draw(6,2) node[above]{$2$}--(16,-2);
\draw(7,2) node[above]{$\cdots$};
\draw(8,2) node[above]{$\ell\hspace{-1mm}-\hspace{-1mm}2$}--(18,-2);
\draw(9,2) node[above]{$\ell$}--(19,-2);
\draw(10,2) node[above]{$0$}--(5,0)--(10,-2);
\draw(11,2) node[above]{$2$}--(6,0)--(11,-2);
\draw(12,2) node[above]{$\cdots$};
\draw(13,2) node[above]{$\ell\hspace{-1mm}-\hspace{-1mm}2$}--(8,0)--(13,-2);
\draw(14,2) node[above]{$\ell\hspace{-1mm}-\hspace{-1mm}1$}--(9,0)--(14,-2);
\draw(15,2) node[above]{$1$}--(5,-2);
\draw(16,2) node[above]{$2$}--(6,-2);
\draw(17,2) node[above]{$\cdots$};
\draw(18,2) node[above]{$\ell\hspace{-1mm}-\hspace{-1mm}2$}--(8,-2);
\draw(19,2) node[above]{$\ell$}--(9,-2);
\draw(12,-2) node[below]{$\varepsilon_{02} \varepsilon_{23} \cdots \varepsilon_{\ell-2,\ell-1}$};
\end{braid}
=
\begin{braid}\tikzset{baseline=0mm}
\draw(5,2) node[above]{$1$}--(5,-2);
\draw(6,2) node[above]{$2$}--(9,0)--(6,-2);
\draw(7,2) node[above]{$\cdots$};
\draw(8,2) node[above]{$\ell\hspace{-1mm}-\hspace{-1mm}2$}--(10,0)--(8,-2);
\draw(9,2) node[above]{$\ell$}--(19,-2);
\draw(10,2) node[above]{$0$}--(5,0)--(10,-2);
\draw(11,2) node[above]{$2$}--(6,0)--(11,-2);
\draw(12,2) node[above]{$\cdots$};
\draw(13,2) node[above]{$\ell\hspace{-1mm}-\hspace{-1mm}2$}--(8,0)--(13,-2);
\draw(14,2) node[above]{$\ell\hspace{-1mm}-\hspace{-1mm}1$}--(9,0)--(14,-2);
\draw(15,2) node[above]{$1$}--(11,0)--(15,-2);
\draw(16,2) node[above]{$2$}--(12,0)--(16,-2);
\draw(17,2) node[above]{$\cdots$};
\draw(18,2) node[above]{$\ell\hspace{-1mm}-\hspace{-1mm}2$}--(13,0)--(18,-2);
\draw(19,2) node[above]{$\ell$}--(9,-2);
\draw(12,-2) node[below]{$-\varepsilon_{02} \varepsilon_{\ell-2,\ell-1}$};
\end{braid},
\end{align*}
after the \((1,2,1)\)-braid opens, followed by the \((2,1,2)\)- through \((\ell-2,\ell-\ell-3,\ell-2)\)-braids. Now the \((2,3)\)- through \((\ell-2,\ell-1)\)-braids open in succession. Then \((2,0,2)\)-braid opens, followed by the \((3,2,3)\)- through \((\ell-2, \ell-3,\ell-2)\)- braids in succession, giving (omitting straight strands on the left):
\begin{align*}
\begin{braid}\tikzset{baseline=0mm}
\draw(9,2) node[above]{$\ell$}--(19,-2);
\draw(10,2) node[above]{$0$}--(5,0)--(10,-2);
\draw(11,2) node[above]{$2$}--(6,0)--(11,-2);
\draw(12,2) node[above]{$\cdots$};
\draw(13,2) node[above]{$\ell\hspace{-1mm}-\hspace{-1mm}2$}--(8,0)--(13,-2);
\draw(14,2) node[above]{$\ell\hspace{-1mm}-\hspace{-1mm}1$}--(9,0)--(14,-2);
\draw(15,2) node[above]{$1$}--(10,0)--(15,-2);
\draw(16,2) node[above]{$2$}--(11,0)--(16,-2);
\draw(17,2) node[above]{$\cdots$};
\draw(18,2) node[above]{$\ell\hspace{-1mm}-\hspace{-1mm}2$}--(13,0)--(18,-2);
\draw(19,2) node[above]{$\ell$}--(9,-2);
\draw(14,-2) node[below]{$(-1)$};
\end{braid}
=
\begin{braid}\tikzset{baseline=0mm}
\draw(9,2) node[above]{$\ell$}--(19,-2);
\draw(10,2) node[above]{$0$}--(10,-2);
\draw(11,2) node[above]{$2$}--(11,-2);
\draw(12,2) node[above]{$\cdots$};
\draw(13,2) node[above]{$\ell\hspace{-1mm}-\hspace{-1mm}2$}--(13,-2);
\draw(14,2) node[above]{$\ell\hspace{-1mm}-\hspace{-1mm}1$}--(15,0)--(14,-2);
\draw(15,2) node[above]{$1$}--(16,0)--(15,-2);
\draw(16,2) node[above]{$2$}--(17,0)--(16,-2);
\draw(17,2) node[above]{$\cdots$};
\draw(18,2) node[above]{$\ell\hspace{-1mm}-\hspace{-1mm}2$}--(19,0)--(18,-2);
\draw(19,2) node[above]{$\ell$}--(9,-2);
\draw(14,-2) node[below]{$(-1)$};
\end{braid}
+
\begin{braid}\tikzset{baseline=0mm}
\draw(9,2) node[above]{$\ell$}--(13.6,0)--(9,-2);
\draw(10,2) node[above]{$0$}--(10,-2);
\draw(11,2) node[above]{$2$}--(11,-2);
\draw(12,2) node[above]{$\cdots$};
\draw(13,2) node[above]{$\ell\hspace{-1mm}-\hspace{-1mm}2$}--(13,-2);
\draw(14,2) node[above]{$\ell\hspace{-1mm}-\hspace{-1mm}1$}--(14,-2);
\draw(15,2) node[above]{$1$}--(15,-2);
\draw(16,2) node[above]{$2$}--(16,-2);
\draw(17,2) node[above]{$\cdots$};
\draw(18,2) node[above]{$\ell\hspace{-1mm}-\hspace{-1mm}2$}--(18,-2);
\draw(19,2) node[above]{$\ell$}--(19,-2);
\draw(14,-2) node[below]{$-\varepsilon_{\ell-2,\ell}$};
\end{braid},
\end{align*}
after applying the braid relation to the \((\ell,\ell-2,\ell)\)-braid. In the left term, the \((\ell,\ell-2,\ell)\)-braid opens, then the \((\ell-2,\ell)\)-double crossing opens, giving
\begin{align*}
-(1 \otimes (y_{d-1} - y_d))v_{i,i} = -(1 \otimes (y_1 - y_d))v_{i,i}.
\end{align*}
In the right term, the \((\ell,\ell-2)\)-double crossing opens, giving
\begin{align*}
(y_d \otimes 1 - 1 \otimes y_{\ell-2})v_{i,i} = (y_d \otimes 1 - 1 \otimes y_1)v_{i,i},
\end{align*}
proving the claim in this case.

\underline{Case \({\tt c}_{i,j}=-1\), \({\tt C}= {\tt A}_{\ell}^{(1)}\)}. We have \(j=i+1\) or \(j=i-1\). We will prove the statement in the former case; the latter is similar. We write \(\sigma' v_{i,j}\) diagrammatically:
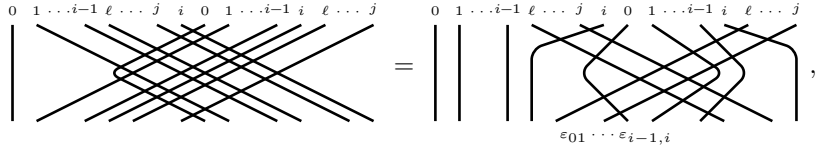
\begin{align*}
\begin{braid}\tikzset{baseline=0mm}
\draw(0,2) node[above]{$0$}--(0,-2);
\draw(1,2) node[above]{$1$}--(9,-2);
\draw(2,2) node[above]{$\cdots$};
\draw(3,2) node[above]{$i\hspace{-0.8mm}-\hspace{-0.8mm}1$}--(11,-2);
\draw(4,2) node[above]{$\ell$}--(12,-2);
\draw(5,2) node[above]{$\cdots$};
\draw(6,2) node[above]{$j$}--(14,-2);
\draw(7,2) node[above]{$i$}--(15,-2);
\draw(8,2) node[above]{$0$}--(4,0)--(8,-2);
\draw(9,2) node[above]{$1$}--(1,-2);
\draw(10,2) node[above]{$\cdots$};
\draw(11,2) node[above]{$i\hspace{-0.8mm}-\hspace{-0.8mm}1$}--(3,-2);
\draw(12,2) node[above]{$i$}--(4,-2);
\draw(13,2) node[above]{$\ell$}--(5,-2);
\draw(14,2) node[above]{$\cdots$};
\draw(15,2) node[above]{$j$}--(7,-2);
\end{braid}
=
 \begin{braid}\tikzset{baseline=0mm}
\draw(0,2) node[above]{$0$}--(0,-2);
\draw(1,2) node[above]{$1$}--(1,-2);
\draw(2,2) node[above]{$\cdots$};
\draw(3,2) node[above]{$i\hspace{-0.8mm}-\hspace{-0.8mm}1$}--(3,-2);
\draw(4,2) node[above]{$\ell$}--(12,-2);
\draw(5,2) node[above]{$\cdots$};
\draw(6,2) node[above]{$j$}--(14,-2);
\draw(7,2) node[above]{$i$}--(4,1)--(4,-2);
\draw(8,2) node[above]{$0$}--(6,0)--(8,-2);
\draw(9,2) node[above]{$1$}--(12,0)--(9,-2);
\draw(10,2) node[above]{$\cdots$};
\draw(11,2) node[above]{$i\hspace{-0.8mm}-\hspace{-0.8mm}1$}--(13,0)--(11,-2);
\draw(12,2) node[above]{$i$}--(15,1)--(15,-2);
\draw(13,2) node[above]{$\ell$}--(5,-2);
\draw(14,2) node[above]{$\cdots$};
\draw(15,2) node[above]{$j$}--(7,-2);
\draw(7.5,-2) node[below]{$\varepsilon_{01} \cdots \varepsilon_{i-1,i}$};
\end{braid},
\end{align*}
after the \((1,0,1)\)- through \((i,i-1,i)\)-braids open in succession. Now the \((\ell,0,\ell)\)-braid opens, followed by the \((\ell-1,\ell,\ell-1)\)- through \((j,j+1,j)\)-braids in succession, giving
\(\xi_i \varepsilon_{ij} (\psi_{j,i} \otimes \psi_{i,j})v_{i,j}\), as desired.

\underline{Case \({\tt c}_{i,j}=-1\), \({\tt C}= {\tt D}_{\ell}^{(1)}\), \(1 \leq i,j\leq \ell-2\)}. We have \(j=i+1\) or \(j=i-1\). We will prove the statement in the former case; the latter is similar. We write \(\sigma' v_{i,j}\) diagrammatically:
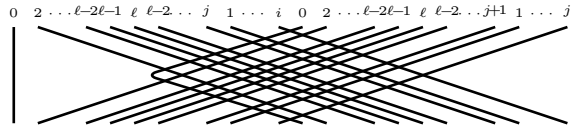
\begin{align*}
 \begin{braid}\tikzset{baseline=0mm}
\draw(0,2) node[above]{$0$}--(0,-2);
\draw(1,2) node[above]{$2$}--(13,-2);
\draw(2,2) node[above]{$\cdots$};
\draw(3,2) node[above]{$\ell\hspace{-1mm}-\hspace{-1mm}2$}--(15,-2);
\draw(4,2) node[above]{$\ell\hspace{-1mm}-\hspace{-1mm}1$}--(16,-2);
\draw(5,2) node[above]{$\ell$}--(17,-2);
\draw(6,2) node[above]{$\ell\hspace{-1mm}-\hspace{-1mm}2$}--(18,-2);
\draw(7,2) node[above]{$\cdots$};
\draw(8,2) node[above]{$j$}--(20,-2);
\draw(9,2) node[above]{$1$}--(21,-2);
\draw(10,2) node[above]{$\cdots$};
\draw(11,2) node[above]{$i$}--(23,-2);
\draw(12,2) node[above]{$0$}--(5.5,0)--(12,-2);
\draw(13,2) node[above]{$2$}--(1,-2);
\draw(14,2) node[above]{$\cdots$};
\draw(15,2) node[above]{$\ell\hspace{-1mm}-\hspace{-1mm}2$}--(3,-2);
\draw(16,2) node[above]{$\ell\hspace{-1mm}-\hspace{-1mm}1$}--(4,-2);
\draw(17,2) node[above]{$\ell$}--(5,-2);
\draw(18,2) node[above]{$\ell\hspace{-1mm}-\hspace{-1mm}2$}--(6,-2);
\draw(19,2) node[above]{$\cdots$};
\draw(20,2) node[above]{$j\hspace{-1mm}+\hspace{-1mm}1$}--(8,-2);
\draw(21,2) node[above]{$1$}--(9,-2);
\draw(22,2) node[above]{$\cdots$};
\draw(23,2) node[above]{$j$}--(11,-2);
\end{braid}.
\end{align*}
Dragging the \(0\)-strand to the right, the \((2,0,2)\)-braid opens, then the \((3,2,3)\)- through \((\ell-1,\ell-2,\ell-1)\)-braids open in succession. Then \((\ell,\ell-2,\ell)\)- and \((\ell-2,\ell,\ell-2)\)-braids open, followed by \((\ell-3,\ell-2,\ell-2)\)- through \((j+1,j+2,j+1)\)-braids opening in succession. This gives (omitting straight strands on the left):
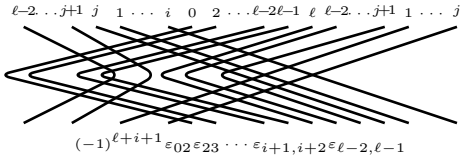
\begin{align*}
 \begin{braid}\tikzset{baseline=0mm}
\draw(5,2) node[above]{$\ell\hspace{-1mm}-\hspace{-1mm}2$}--(9,0)--(5,-2);
\draw(6,2) node[above]{$\cdots$};
\draw(7,2) node[above]{$j\hspace{-1mm}+\hspace{-1mm}1$}--(10.5,0)--(7,-2);
\draw(8,2) node[above]{$j$}--(20,-2);
\draw(9,2) node[above]{$1$}--(21,-2);
\draw(10,2) node[above]{$\cdots$};
\draw(11,2) node[above]{$i$}--(23,-2);
\draw(12,2) node[above]{$0$}--(4,0)--(12,-2);
\draw(13,2) node[above]{$2$}--(5,0)--(13,-2);
\draw(14,2) node[above]{$\cdots$};
\draw(15,2) node[above]{$\ell\hspace{-1mm}-\hspace{-1mm}2$}--(7,0)--(15,-2);
\draw(16,2) node[above]{$\ell\hspace{-1mm}-\hspace{-1mm}1$}--(8,0)--(16,-2);
\draw(17,2) node[above]{$\ell$}--(10.5,0)--(17,-2);
\draw(18,2) node[above]{$\ell\hspace{-1mm}-\hspace{-1mm}2$}--(11.5,0)--(18,-2);
\draw(19,2) node[above]{$\cdots$};
\draw(20,2) node[above]{$j\hspace{-1mm}+\hspace{-1mm}1$}--(13,0)--(19,-2);
\draw(21,2) node[above]{$1$}--(9,-2);
\draw(22,2) node[above]{$\cdots$};
\draw(23,2) node[above]{$j$}--(11,-2);
\draw(14,-2) node[below]{$(-1)^{\ell+i+1}\varepsilon_{02} \varepsilon_{23} \cdots \varepsilon_{i+1,i+2} \varepsilon_{\ell-2,\ell-1}$};
\end{braid}.
\end{align*}
Now the \((\ell-2,\ell-1)\)-double crossing opens. The \((\ell-2,\ell-3,\ell-2)\)-braid opens, which introduces an \((\ell-2,\ell-3)\)-double crossing which opens. This sequence repeats, until the \((j+2,j+1,j+1)\)-braid opens, introducing a \((j+2,j+1)\)-double crossing which opens. Finally, a \((j+1,j,j+1)\)-braid opens, giving (omitting straight strands on the left):
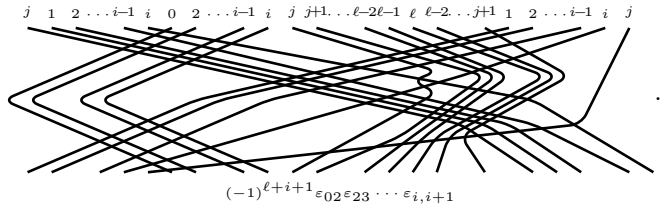
\begin{align*}
 \begin{braid}\tikzset{baseline=0mm}
\draw(5,3) node[above]{$j$}--(23,-1)--(26,-3);
\draw(6,3) node[above]{$1$}--(24,-1)--(27,-3);
\draw(7,3) node[above]{$2$}--(25,-1)--(28,-3);
\draw(8,3) node[above]{$\cdots$};
\draw(9,3) node[above]{$i\hspace{-1mm}-\hspace{-1mm}1$}--(26,-1)--(30,-3);
\draw(10,3) node[above]{$i$}--(26,0)--(31,-3);
\draw(11,3) node[above]{$0$}--(4,0)--(11,-3);
\draw(12,3) node[above]{$2$}--(5,0)--(12,-3);
\draw(13,3) node[above]{$\cdots$};
\draw(14,3) node[above]{$i\hspace{-1mm}-\hspace{-1mm}1$}--(7,0)--(14,-3);
\draw(15,3) node[above]{$i$}--(8,0)--(15,-3);
\draw(16,3) node[above]{$j$}--(22,1.2)--(21,0.8)--(22,0.3)--(16,-3);
\draw(17,3) node[above]{$j\hspace{-1mm}+\hspace{-1mm}1$}--(24,1)--(20.5,-1.5)--(17,-3);
\draw(18,3) node[above]{$\cdots$};
\draw(19,3) node[above]{$\ell\hspace{-1mm}-\hspace{-1mm}2$}--(24.5,1)--(21,-1.5)--(19,-3);
\draw(20,3) node[above]{$\ell\hspace{-1mm}-\hspace{-1mm}1$}--(25,1)--(21.5,-1.5)--(20,-3);
\draw(21,3) node[above]{$\ell$}--(26.5,1)--(22,-1.5)--(21,-3);
\draw(22,3) node[above]{$\ell\hspace{-1mm}-\hspace{-1mm}2$}--(27,1)--(22.5,-1.5)--(22,-3);
\draw(23,3) node[above]{$\cdots$};
\draw(24,3) node[above]{$j\hspace{-1mm}+\hspace{-1mm}1$}--(27.5,1)--(23,-1.5)--(24,-3);
\draw(25,3) node[above]{$1$}--(12,0)--(5,-3);
\draw(26,3) node[above]{$2$}--(13,0)--(6,-3);
\draw(27,3) node[above]{$\cdots$};
\draw(28,3) node[above]{$i\hspace{-1mm}-\hspace{-1mm}1$}--(15,0)--(8,-3);
\draw(29,3) node[above]{$i$}--(23,1)--(9,-3);
\draw(30,3) node[above]{$j$}--(28,-1)--(10,-3);
\draw(18,-3) node[below]{$(-1)^{\ell+i+1}\varepsilon_{02} \varepsilon_{23} \cdots \varepsilon_{i,i+1}$};
\end{braid}.
\end{align*}
Now the \((i,j,i)\)-braid opens, and then the \((i-1,i,i-1)\)-through \((1,2,1)\)-braids open in succession. Finally, the \((j,j+1,j)\)-braid opens, giving:
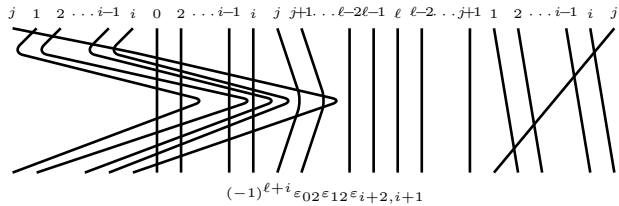
\begin{align*}
 \begin{braid}\tikzset{baseline=0mm}
\draw(5,3) node[above]{$j$}--(10,2)--(18.7,0)--(10,-3);
\draw(6,3) node[above]{$1$}--(5,2)--(13,0)--(5,-3);
\draw(7,3) node[above]{$2$}--(6,2)--(15,0)--(6,-3);
\draw(8,3) node[above]{$\cdots$};
\draw(9,3) node[above]{$i\hspace{-1mm}-\hspace{-1mm}1$}--(8,2)--(16,0)--(8,-3);
\draw(10,3) node[above]{$i$}--(9,2)--(16.7,0)--(9,-3);
\draw(11,3) node[above]{$0$}--(11,-3);
\draw(12,3) node[above]{$2$}--(12,-3);
\draw(13,3) node[above]{$\cdots$};
\draw(14,3) node[above]{$i\hspace{-1mm}-\hspace{-1mm}1$}--(14,-3);
\draw(15,3) node[above]{$i$}--(15,-3);
\draw(16,3) node[above]{$j$}--(17,0)--(16,-3);
\draw(17,3) node[above]{$j\hspace{-1mm}+\hspace{-1mm}1$}--(18,0)--(17,-3);
\draw(18,3) node[above]{$\cdots$};
\draw(19,3) node[above]{$\ell\hspace{-1mm}-\hspace{-1mm}2$}--(19,-3);
\draw(20,3) node[above]{$\ell\hspace{-1mm}-\hspace{-1mm}1$}--(20,-3);
\draw(21,3) node[above]{$\ell$}--(21,-3);
\draw(22,3) node[above]{$\ell\hspace{-1mm}-\hspace{-1mm}2$}--(22,-3);
\draw(23,3) node[above]{$\cdots$};
\draw(24,3) node[above]{$j\hspace{-1mm}+\hspace{-1mm}1$}--(24,-3);
\draw(25,3) node[above]{$1$}--(26,-3);
\draw(26,3) node[above]{$2$}--(27,-3);
\draw(27,3) node[above]{$\cdots$};
\draw(28,3) node[above]{$i\hspace{-1mm}-\hspace{-1mm}1$}--(29,-3);
\draw(29,3) node[above]{$i$}--(30,-3);
\draw(30,3) node[above]{$j$}--(25,-3);
\draw(18,-3) node[below]{$(-1)^{\ell+i}\varepsilon_{02} \varepsilon_{12} \varepsilon_{i+2,i+1}$};
\end{braid}.
\end{align*}
Now the \((j,j+1)\)-double crossing opens, followed by the \((i-1,i)\)- through \((1,2)\)-double crossings in succession, giving (omitting strands on the right):
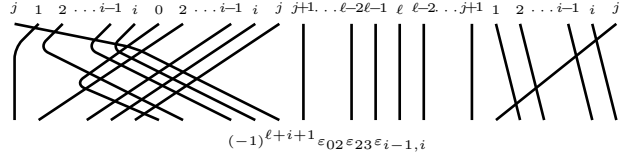
\begin{align*}
 \begin{braid}\tikzset{baseline=0mm}
\draw(5,2) node[above]{$j$}--(10,1)--(16,-2);
\draw(6,2) node[above]{$1$}--(5,1)--(5,-2);
\draw(7,2) node[above]{$2$}--(6,1)--(12,-2);
\draw(8,2) node[above]{$\cdots$};
\draw(9,2) node[above]{$i\hspace{-1mm}-\hspace{-1mm}1$}--(8,1)--(14,-2);
\draw(10,2) node[above]{$i$}--(9,1)--(15,-2);
\draw(11,2) node[above]{$0$}--(7.5,-0.5)--(11,-2);
\draw(12,2) node[above]{$2$}--(6,-2);
\draw(13,2) node[above]{$\cdots$};
\draw(14,2) node[above]{$i\hspace{-1mm}-\hspace{-1mm}1$}--(8,-2);
\draw(15,2) node[above]{$i$}--(9,-2);
\draw(16,2) node[above]{$j$}--(10,-2);
\draw(17,2) node[above]{$j\hspace{-1mm}+\hspace{-1mm}1$}--(17,-2);
\draw(18,2) node[above]{$\cdots$};
\draw(19,2) node[above]{$\ell\hspace{-1mm}-\hspace{-1mm}2$}--(19,-2);
\draw(20,2) node[above]{$\ell\hspace{-1mm}-\hspace{-1mm}1$}--(20,-2);
\draw(21,2) node[above]{$\ell$}--(21,-2);
\draw(22,2) node[above]{$\ell\hspace{-1mm}-\hspace{-1mm}2$}--(22,-2);
\draw(23,2) node[above]{$\cdots$};
\draw(24,2) node[above]{$j\hspace{-1mm}+\hspace{-1mm}1$}--(24,-2);
\draw(25,2) node[above]{$1$}--(26,-2);
\draw(26,2) node[above]{$2$}--(27,-2);
\draw(27,2) node[above]{$\cdots$};
\draw(28,2) node[above]{$i\hspace{-1mm}-\hspace{-1mm}1$}--(29,-2);
\draw(29,2) node[above]{$i$}--(30,-2);
\draw(30,2) node[above]{$j$}--(25,-2);
\draw(18,-2) node[below]{$(-1)^{\ell+i+1}\varepsilon_{02} \varepsilon_{23} \varepsilon_{i-1,i}$};
\end{braid}.
\end{align*}
Now the \((2,0,2)\)-braid opens, followed by the \((3,2,3)\)- through \((j,i,j)\)-braids in succession, giving \((-1)^{\ell+i+1}\varepsilon_{i,i+1}(\psi_{j,i} \otimes \psi_{i,j})v_{i,j}\), as desired.

\underline{Case \({\tt c}_{i,j}=-1\), \({\tt C}= {\tt D}_{\ell}^{(1)}\), \(\ell-2 \leq i,j \leq \ell\)}. We will check the case \(i=\ell-2\), \(j=\ell\). The other cases are similar. We write \(\sigma' v_{i,j}\) diagrammatically:
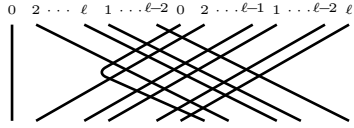
\begin{align*}
 \begin{braid}\tikzset{baseline=0mm}
\draw(0,2) node[above]{$0$}--(0,-2);
\draw(1,2) node[above]{$2$}--(9,-2);
\draw(2,2) node[above]{$\cdots$};
\draw(3,2) node[above]{$\ell$}--(11,-2);
\draw(4,2) node[above]{$1$}--(12,-2);
\draw(5,2) node[above]{$\cdots$};
\draw(6,2) node[above]{$\ell\hspace{-1mm}-\hspace{-1mm}2$}--(14,-2);
\draw(7,2) node[above]{$0$}--(3.5,0)--(8,-2);
\draw(8,2) node[above]{$2$}--(1,-2);
\draw(9,2) node[above]{$\cdots$};
\draw(10,2) node[above]{$\ell\hspace{-1mm}-\hspace{-1mm}1$}--(3,-2);
\draw(11,2) node[above]{$1$}--(4,-2);
\draw(12,2) node[above]{$\cdots$};
\draw(13,2) node[above]{$\ell\hspace{-1mm}-\hspace{-1mm}2$}--(6,-2);
\draw(14,2) node[above]{$\ell$}--(7,-2);
\end{braid}.
\end{align*}
Dragging the \(0\)-strand to the right, the \((2,0,2)\)-braid opens, then the \((3,2,3)\)- through \((\ell-1,\ell-2,\ell-1)\)-braids open in succession. The \((\ell-3,\ell-2,\ell-2)\)-braid opens, and then the \((\ell-4,\ell-3,\ell-4)\)- through \((1,2,1)\)-braids open in succession, giving (omitting straight strands on the left):
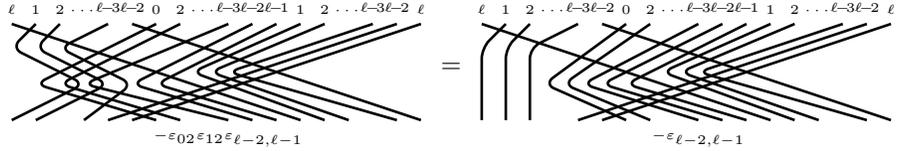
\begin{align*}
 \begin{braid}\tikzset{baseline=0mm}
\draw(0,2) node[above]{$\ell$}--(12,-2);
\draw(1,2) node[above]{$1$}--(0,1)--(3,-0.5)--(0,-2);
\draw(2,2) node[above]{$2$}--(1,1)--(4,-0.5)--(1,-2);
\draw(3,2) node[above]{$\cdots$};
\draw(4,2) node[above]{$\ell\hspace{-1mm}-\hspace{-1mm}3$}--(2,1)--(5,-0.5)--(3,-2);
\draw(5,2) node[above]{$\ell\hspace{-1mm}-\hspace{-1mm}2$}--(17,-2);
\draw(6,2) node[above]{$0$}--(1,-0.5)--(6,-2);
\draw(7,2) node[above]{$2$}--(2,-0.5)--(7,-2);
\draw(8,2) node[above]{$\cdots$};
\draw(9,2) node[above]{$\ell\hspace{-1mm}-\hspace{-1mm}3$}--(3,-0.5)--(9,-2);
\draw(10,2) node[above]{$\ell\hspace{-1mm}-\hspace{-1mm}2$}--(5,-0.5)--(10,-2);
\draw(11,2) node[above]{$\ell\hspace{-1mm}-\hspace{-1mm}1$}--(6,-0.5)--(11,-2);
\draw(12,2) node[above]{$1$}--(7.5,0)--(13,-2);
\draw(13,2) node[above]{$2$}--(8.25,0)--(14,-2);
\draw(14,2) node[above]{$\cdots$};
\draw(15,2) node[above]{$\ell\hspace{-1mm}-\hspace{-1mm}3$}--(9,0)--(16,-2);
\draw(16,2) node[above]{$\ell\hspace{-1mm}-\hspace{-1mm}2$}--(4,-2);
\draw(17,2) node[above]{$\ell$}--(5,-2);
\draw(9,-2) node[below]{$-\varepsilon_{02}\varepsilon_{12}\varepsilon_{\ell-2,\ell-1}$};
\end{braid}
=
 \begin{braid}\tikzset{baseline=0mm}
\draw(0,2) node[above]{$\ell$}--(12,-2);
\draw(1,2) node[above]{$1$}--(0,1)--(0,-2);
\draw(2,2) node[above]{$2$}--(1,1)--(1,-2);
\draw(3,2) node[above]{$\cdots$};
\draw(4,2) node[above]{$\ell\hspace{-1mm}-\hspace{-1mm}3$}--(2,1)--(2,-2);
\draw(5,2) node[above]{$\ell\hspace{-1mm}-\hspace{-1mm}2$}--(17,-2);
\draw(6,2) node[above]{$0$}--(2.6,-0.5)--(6,-2);
\draw(7,2) node[above]{$2$}--(3.4,-0.5)--(7,-2);
\draw(8,2) node[above]{$\cdots$};
\draw(9,2) node[above]{$\ell\hspace{-1mm}-\hspace{-1mm}3$}--(4.2,-0.5)--(9,-2);
\draw(10,2) node[above]{$\ell\hspace{-1mm}-\hspace{-1mm}2$}--(5,-0.5)--(10,-2);
\draw(11,2) node[above]{$\ell\hspace{-1mm}-\hspace{-1mm}1$}--(6,-0.5)--(11,-2);
\draw(12,2) node[above]{$1$}--(7.5,0)--(13,-2);
\draw(13,2) node[above]{$2$}--(8.25,0)--(14,-2);
\draw(14,2) node[above]{$\cdots$};
\draw(15,2) node[above]{$\ell\hspace{-1mm}-\hspace{-1mm}3$}--(9,0)--(16,-2);
\draw(16,2) node[above]{$\ell\hspace{-1mm}-\hspace{-1mm}2$}--(4,-2);
\draw(17,2) node[above]{$\ell$}--(5,-2);
\draw(9,-2) node[below]{$-\varepsilon_{\ell-2,\ell-1}$};
\end{braid},
\end{align*}
after the \((\ell-4,\ell-3)\)- through \((1,2)\)-double crossings open in succession, followed by the \((2,0,2)\)- and \((3,2,3)\)- through \((\ell-3,\ell-4,\ell-3)\)-braids opening in succession. Now the \((\ell-2,\ell-3,\ell-2)\)-braid opens, introducing an \((\ell-2, \ell-1)\)-double crossing which opens, followed by \((\ell-2, \ell-3,\ell-2)\)- and \((\ell,\ell-2,\ell)\)-braids opening, which gives \(-\varepsilon_{ij}(\psi_{j,i} \otimes \psi_{i,j})v_{i,j}\), as desired.

\underline{Case \({\tt c}_{i,j}=0\), all types}. By the usual manipulations of KLR elements (cf. \cite[\S 2.6]{BKW}), we may write \(\sigma' v_{i,j}\) as a sum of terms of the form \(1_{\bb^j \bb^i}\psi_w (x_i \otimes x_j)\), where \(x_i \in \Delta_{\delta,i}\) and \(x_j \in \Delta_{\delta,j}\), and \(w \triangleleft \sigma'\) (where we consider \(\sigma'\) as an element of \(\mathfrak{S}_{2d})\) is a minimal left coset representative for \(\mathfrak{S}_{2d} / \mathfrak{S}_d \times \mathfrak{S}_d\). Since \((\bb^j \bb^i)_1 = (\bb^j \bb^i)_{d+1} = 0\) and \(i_1 = 0\) for every word \(\bi\) of \(\Delta_{\delta,i}\) and \(\Delta_{\delta,j}\), it follows that \(w=\textup{id}\). But \(1_j \Delta_{\delta,i} = 0\) by Lemma \ref{MinDelHom}, so  \(\sigma' v_{i,j} = 0\).
\end{proof}

\noindent{\bf Lemma \ref{psisigma}.}
{\em
Let \(i,j,m \in I'\) with \({\tt c}_{i,j}=-1\). Then we have}
\begin{align*}
(\psi_{j,i} \otimes 1)\sigma v_{m,i} = [\sigma(1 \otimes \psi_{j,i}) + \delta_{j,m} \xi_j(1 \otimes \psi_{j,i}) - \delta_{i,m} \xi_i(\psi_{j,i} \otimes 1)]v_{m,i}.
\end{align*}

\begin{proof}
\underline{Case \(m=j\), \({\tt C}={\tt A}^{(1)}_\ell\)}. Since \(i\) and \(j\) are neighbors, either \(j=i-1\) or \(j=i+1\). We will prove the claim in the former case; the latter is similar. We depict \((\psi_{j,i}\otimes 1) \sigma v_{j,i}\) diagrammatically: 
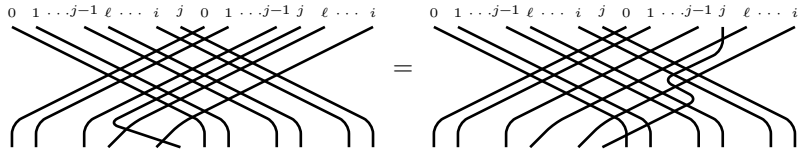
\begin{align*}
\begin{braid}\tikzset{baseline=0mm}
\draw(0,2) node[above]{$0$}--(8,-2)--(8,-3);
\draw(1,2) node[above]{$1$}--(9,-2)--(9,-3);
\draw(2,2) node[above]{$\cdots$};
\draw(3,2) node[above]{$j\hspace{-0.8mm}-\hspace{-0.8mm}1$}--(11,-2)--(11,-3);
\draw(4,2) node[above]{$\ell$}--(12,-2)--(12,-3);
\draw(5,2) node[above]{$\cdots$};
\draw(6,2) node[above]{$i$}--(14,-2)--(14,-3);
\draw(7,2) node[above]{$j$}--(15,-2)--(15,-3);
\draw(8,2) node[above]{$0$}--(0,-2)--(0,-3);
\draw(9,2) node[above]{$1$}--(1,-2)--(1,-3);
\draw(10,2) node[above]{$\cdots$};
\draw(11,2) node[above]{$j\hspace{-0.8mm}-\hspace{-0.8mm}1$}--(3,-2)--(3,-3);
\draw(12,2) node[above]{$j$}--(4,-2)--(7,-3);
\draw(13,2) node[above]{$\ell$}--(5,-2)--(4,-3);
\draw(14,2) node[above]{$\cdots$};
\draw(15,2) node[above]{$i$}--(7,-2)--(6,-3);
\end{braid}
=
\begin{braid}\tikzset{baseline=0mm}
\draw(0,2) node[above]{$0$}--(8,-2)--(8,-3);
\draw(1,2) node[above]{$1$}--(9,-2)--(9,-3);
\draw(2,2) node[above]{$\cdots$};
\draw(3,2) node[above]{$j\hspace{-0.8mm}-\hspace{-0.8mm}1$}--(11,-2)--(11,-3);
\draw(4,2) node[above]{$\ell$}--(12,-2)--(12,-3);
\draw(5,2) node[above]{$\cdots$};
\draw(6,2) node[above]{$i$}--(14,-2)--(14,-3);
\draw(7,2) node[above]{$j$}--(15,-2)--(15,-3);
\draw(8,2) node[above]{$0$}--(0,-2)--(0,-3);
\draw(9,2) node[above]{$1$}--(1,-2)--(1,-3);
\draw(10,2) node[above]{$\cdots$};
\draw(11,2) node[above]{$j\hspace{-0.8mm}-\hspace{-0.8mm}1$}--(3,-2)--(3,-3);
\draw(12,2) node[above]{$j$}--(12,1)--(9.5,-0.2)--(11,-1)--(7,-3);
\draw(13,2) node[above]{$\ell$}--(5,-2)--(4,-3);
\draw(14,2) node[above]{$\cdots$};
\draw(15,2) node[above]{$i$}--(7,-2)--(6,-3);
\end{braid}
\end{align*}
Applying the braid relation to the \((i,j,i)\)-braid, we have \(\sigma(1 \otimes \psi_{j,i}) v_{j,i}\), plus the error term:
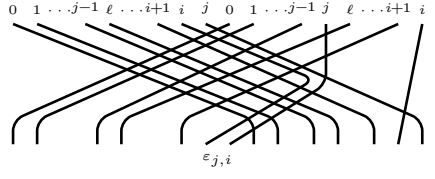
\begin{align*}
\begin{braid}\tikzset{baseline=0mm}
\draw(0,2) node[above]{$0$}--(10,-2)--(10,-3);
\draw(1,2) node[above]{$1$}--(11,-2)--(11,-3);
\draw(2,2) node[above]{$\cdots$};
\draw(3,2) node[above]{$j\hspace{-0.8mm}-\hspace{-0.8mm}1$}--(12.5,-2)--(12.5,-3);
\draw(4,2) node[above]{$\ell$}--(13.5,-2)--(13.5,-3);
\draw(5,2) node[above]{$\cdots$};
\draw(6,2) node[above]{$i\hspace{-0.8mm}+\hspace{-0.8mm}1$}--(15,-2)--(15,-3);
\draw(7,2) node[above]{$i$}--(12.5,-0.3)--(8,-3);
\draw(8,2) node[above]{$j$}--(17,-2)--(17,-3);
\draw(9,2) node[above]{$0$}--(0,-2)--(0,-3);
\draw(10,2) node[above]{$1$}--(1,-2)--(1,-3);
\draw(11,2) node[above]{$\cdots$};
\draw(12,2) node[above]{$j\hspace{-0.8mm}-\hspace{-0.8mm}1$}--(3.5,-2)--(3.5,-3);
\draw(13,2) node[above]{$j$}--(13,-0.5)--(9,-3);
\draw(14,2) node[above]{$\ell$}--(4.5,-2)--(4.5,-3);
\draw(15,2) node[above]{$\cdots$};
\draw(16,2) node[above]{$i\hspace{-0.8mm}+\hspace{-0.8mm}1$}--(7,-2)--(7,-3);
\draw(17,2) node[above]{$i$}--(16,-3);
\draw(8.5,-3) node[below]{$\varepsilon_{j,i}$};
\end{braid}.
\end{align*}
the \((i,i+1,i)\)- through \((\ell, \ell-1, \ell)\)-braids open in succession, giving
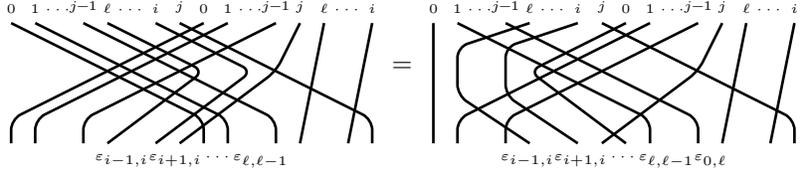
\begin{align*}
\begin{braid}\tikzset{baseline=0mm}
\draw(0,2) node[above]{$0$}--(8,-2)--(8,-3);
\draw(1,2) node[above]{$1$}--(9,-2)--(9,-3);
\draw(2,2) node[above]{$\cdots$};
\draw(3,2) node[above]{$j\hspace{-0.8mm}-\hspace{-0.8mm}1$}--(11,-2)--(11,-3);
\draw(4,2) node[above]{$\ell$}--(8,0)--(4,-3);
\draw(5,2) node[above]{$\cdots$};
\draw(6,2) node[above]{$i$}--(10,0)--(6,-3);
\draw(7,2) node[above]{$j$}--(15,-2)--(15,-3);
\draw(8,2) node[above]{$0$}--(0,-2)--(0,-3);
\draw(9,2) node[above]{$1$}--(1,-2)--(1,-3);
\draw(10,2) node[above]{$\cdots$};
\draw(11,2) node[above]{$j\hspace{-0.8mm}-\hspace{-0.8mm}1$}--(3,-2)--(3,-3);
\draw(12,2) node[above]{$j$}--(11,0)--(7,-3);
\draw(13,2) node[above]{$\ell$}--(12,-3);
\draw(14,2) node[above]{$\cdots$};
\draw(15,2) node[above]{$i$}--(14,-3);
\draw(7.5,-3) node[below]{$\varepsilon_{i-1,i}\varepsilon_{i+1,i}\cdots \varepsilon_{\ell,\ell-1}$};
\end{braid}
=
\begin{braid}\tikzset{baseline=0mm}
\draw(0,2) node[above]{$0$}--(0,-3);
\draw(1,2) node[above]{$1$}--(9,-2)--(9,-3);
\draw(2,2) node[above]{$\cdots$};
\draw(3,2) node[above]{$j\hspace{-0.8mm}-\hspace{-0.8mm}1$}--(11,-2)--(11,-3);
\draw(4,2) node[above]{$\ell$}--(1,1)--(1,-1)--(4,-3);
\draw(5,2) node[above]{$\cdots$};
\draw(6,2) node[above]{$i$}--(3,1)--(3,-1)--(6,-3);
\draw(7,2) node[above]{$j$}--(15,-2)--(15,-3);
\draw(8,2) node[above]{$0$}--(4,0)--(8,-3);
\draw(9,2) node[above]{$1$}--(1,-2)--(1,-3);
\draw(10,2) node[above]{$\cdots$};
\draw(11,2) node[above]{$j\hspace{-0.8mm}-\hspace{-0.8mm}1$}--(3,-2)--(3,-3);
\draw(12,2) node[above]{$j$}--(11,0)--(7,-3);
\draw(13,2) node[above]{$\ell$}--(12,-3);
\draw(14,2) node[above]{$\cdots$};
\draw(15,2) node[above]{$i$}--(14,-3);
\draw(7.5,-3) node[below]{$\varepsilon_{i-1,i}\varepsilon_{i+1,i}\cdots \varepsilon_{\ell,\ell-1}\varepsilon_{0,\ell}$};
\end{braid},
\end{align*}
after the \((0,\ell,0)\)-braid opens. Now, the \((1,0,1)\)- through \((j,j-1,j)\)-braids open in succession, giving \( \xi_j(1 \otimes \psi_{j,i})v_{j,i}\), as desired. 

\underline{Case \(m=i\), \({\tt C}={\tt A}^{(1)}_\ell\)}. We show that the claim holds in the case \(i=j+1\); the case \(i=j-1\) is similar. We depict \((\psi_{j,i} \otimes 1)\sigma v_{i,i}\) diagrammatically: 
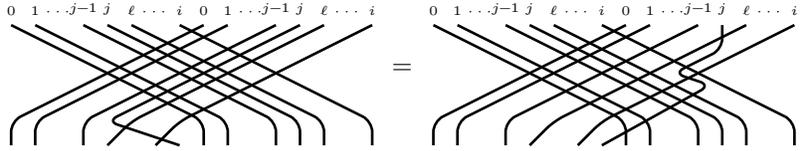
\begin{align*}
\begin{braid}\tikzset{baseline=0mm}
\draw(0,2) node[above]{$0$}--(8,-2)--(8,-3);
\draw(1,2) node[above]{$1$}--(9,-2)--(9,-3);
\draw(2,2) node[above]{$\cdots$};
\draw(3,2) node[above]{$j\hspace{-0.8mm}-\hspace{-0.8mm}1$}--(11,-2)--(11,-3);
\draw(4,2) node[above]{$j$}--(12,-2)--(12,-3);
\draw(5,2) node[above]{$\ell$}--(13,-2)--(13,-3);
\draw(6,2) node[above]{$\cdots$};
\draw(7,2) node[above]{$i$}--(15,-2)--(15,-3);
\draw(8,2) node[above]{$0$}--(0,-2)--(0,-3);
\draw(9,2) node[above]{$1$}--(1,-2)--(1,-3);
\draw(10,2) node[above]{$\cdots$};
\draw(11,2) node[above]{$j\hspace{-0.8mm}-\hspace{-0.8mm}1$}--(3,-2)--(3,-3);
\draw(12,2) node[above]{$j$}--(4,-2)--(7,-3);
\draw(13,2) node[above]{$\ell$}--(5,-2)--(4,-3);
\draw(14,2) node[above]{$\cdots$};
\draw(15,2) node[above]{$i$}--(7,-2)--(6,-3);
\end{braid}
=
\begin{braid}\tikzset{baseline=0mm}
\draw(0,2) node[above]{$0$}--(8,-2)--(8,-3);
\draw(1,2) node[above]{$1$}--(9,-2)--(9,-3);
\draw(2,2) node[above]{$\cdots$};
\draw(3,2) node[above]{$j\hspace{-0.8mm}-\hspace{-0.8mm}1$}--(11,-2)--(11,-3);
\draw(4,2) node[above]{$j$}--(12,-2)--(12,-3);
\draw(5,2) node[above]{$\ell$}--(13,-2)--(13,-3);
\draw(6,2) node[above]{$\cdots$};
\draw(7,2) node[above]{$i$}--(15,-2)--(15,-3);
\draw(8,2) node[above]{$0$}--(0,-2)--(0,-3);
\draw(9,2) node[above]{$1$}--(1,-2)--(1,-3);
\draw(10,2) node[above]{$\cdots$};
\draw(11,2) node[above]{$j\hspace{-0.8mm}-\hspace{-0.8mm}1$}--(3,-2)--(3,-3);
\draw(12,2) node[above]{$j$}--(12,1)--(10,0)--(11.5,-0.5)--(7,-3);
\draw(13,2) node[above]{$\ell$}--(5,-2)--(4,-3);
\draw(14,2) node[above]{$\cdots$};
\draw(15,2) node[above]{$i$}--(7,-2)--(6,-3);
\end{braid}.
\end{align*}
Applying the braid relation to the \((i,j,i)\)-braid, we get \(\sigma (1 \otimes \psi_{j,i})v_{i,i}\) plus the error term:
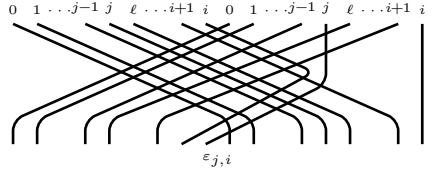
\begin{align*}
\begin{braid}\tikzset{baseline=0mm}
\draw(0,2) node[above]{$0$}--(9,-2)--(9,-3);
\draw(1,2) node[above]{$1$}--(10,-2)--(10,-3);
\draw(2,2) node[above]{$\cdots$};
\draw(3,2) node[above]{$j\hspace{-0.8mm}-\hspace{-0.8mm}1$}--(12,-2)--(12,-3);
\draw(4,2) node[above]{$j$}--(13,-2)--(13,-3);
\draw(5,2) node[above]{$\ell$}--(14,-2)--(14,-3);
\draw(6,2) node[above]{$\cdots$};
\draw(7,2) node[above]{$i\hspace{-0.8mm}+\hspace{-0.8mm}1$}--(16,-2)--(16,-3);
\draw(8,2) node[above]{$i$}--(12.5,0)--(7,-3);
\draw(9,2) node[above]{$0$}--(0,-2)--(0,-3);
\draw(10,2) node[above]{$1$}--(1,-2)--(1,-3);
\draw(11,2) node[above]{$\cdots$};
\draw(12,2) node[above]{$j\hspace{-0.8mm}-\hspace{-0.8mm}1$}--(3,-2)--(3,-3);
\draw(13,2) node[above]{$j$}--(13,-0.5)--(8,-3);
\draw(14,2) node[above]{$\ell$}--(4,-2)--(4,-3);
\draw(15,2) node[above]{$\cdots$};
\draw(16,2) node[above]{$i\hspace{-0.8mm}+\hspace{-0.8mm}1$}--(6,-2)--(6,-3);
\draw(17,2) node[above]{$i$}--(17,-3);
\draw(8.5,-3) node[below]{$\varepsilon_{j,i}$};
\end{braid}.
\end{align*}
For this term, the \((i+1,i,i+1)\)- through \((\ell, \ell-1,\ell)\)-braids open, giving
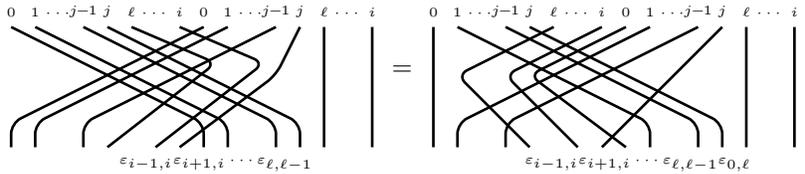
\begin{align*}
\begin{braid}\tikzset{baseline=0mm}
\draw(0,2) node[above]{$0$}--(8,-2)--(8,-3);
\draw(1,2) node[above]{$1$}--(9,-2)--(9,-3);
\draw(2,2) node[above]{$\cdots$};
\draw(3,2) node[above]{$j\hspace{-0.8mm}-\hspace{-0.8mm}1$}--(11,-2)--(11,-3);
\draw(4,2) node[above]{$j$}--(12,-2)--(12,-3);
\draw(5,2) node[above]{$\ell$}--(8.5,0.5)--(4,-3);
\draw(6,2) node[above]{$\cdots$};
\draw(7,2) node[above]{$i$}--(10.5,0.5)--(6,-3);
\draw(8,2) node[above]{$0$}--(0,-2)--(0,-3);
\draw(9,2) node[above]{$1$}--(1,-2)--(1,-3);
\draw(10,2) node[above]{$\cdots$};
\draw(11,2) node[above]{$j\hspace{-0.8mm}-\hspace{-0.8mm}1$}--(3,-2)--(3,-3);
\draw(12,2) node[above]{$j$}--(11,0)--(7,-3);
\draw(13,2) node[above]{$\ell$}--(13,-2)--(13,-3);
\draw(14,2) node[above]{$\cdots$};
\draw(15,2) node[above]{$i$}--(15,-2)--(15,-3);
\draw(8.5,-3) node[below]{$\varepsilon_{i-1,i}\varepsilon_{i+1,i} \cdots \varepsilon_{\ell,\ell-1}$};
\end{braid}
=
\begin{braid}\tikzset{baseline=0mm}
\draw(0,2) node[above]{$0$}--(0,-3);
\draw(1,2) node[above]{$1$}--(9,-2)--(9,-3);
\draw(2,2) node[above]{$\cdots$};
\draw(3,2) node[above]{$j\hspace{-0.8mm}-\hspace{-0.8mm}1$}--(11,-2)--(11,-3);
\draw(4,2) node[above]{$j$}--(12,-2)--(12,-3);
\draw(5,2) node[above]{$\ell$}--(1,0)--(4,-3);
\draw(6,2) node[above]{$\cdots$};
\draw(7,2) node[above]{$i$}--(3,0)--(6,-3);
\draw(8,2) node[above]{$0$}--(4,0)--(8,-3);
\draw(9,2) node[above]{$1$}--(1,-2)--(1,-3);
\draw(10,2) node[above]{$\cdots$};
\draw(11,2) node[above]{$j\hspace{-0.8mm}-\hspace{-0.8mm}1$}--(3,-2)--(3,-3);
\draw(12,2) node[above]{$j$}--(7,-3);
\draw(13,2) node[above]{$\ell$}--(13,-2)--(13,-3);
\draw(14,2) node[above]{$\cdots$};
\draw(15,2) node[above]{$i$}--(15,-2)--(15,-3);
\draw(8.5,-3) node[below]{$\varepsilon_{i-1,i}\varepsilon_{i+1,i} \cdots \varepsilon_{\ell,\ell-1}\varepsilon_{0,\ell}$};
\end{braid},
\end{align*}
after the \((0,\ell,0)\)-braid opens. Now the \((1,0,1)\)- through \((j,j-1,j)\)-braids open in succession, giving \(-\xi_i (\psi_{j,i} \otimes 1)v_{i,i}\), as desired.

\underline{Case \(m=j\), \({\tt C}={\tt D}^{(1)}_\ell\), \(1 \leq i,j \leq \ell-2\)}. We check that (\ref{psisigma}) holds in the case \( j=i+1\). The case \(j=i-1\) is similar. We depict \((\psi_{j,i} \otimes 1)\sigma v_{j,i}\) diagrammatically, with \(v_{j,i}\) at the top of the diagram: 
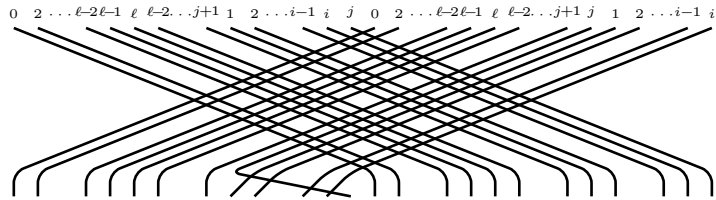
\begin{align*}
\begin{braid}\tikzset{baseline=0mm}
\draw(0,2) node[above]{$0$}--(15,-4)--(15,-5);
\draw(1,2) node[above]{$2$}--(16,-4)--(16,-5);
\draw(2,2) node[above]{$\cdots$};
\draw(3,2) node[above]{$\ell\hspace{-1.1mm}-\hspace{-1.1mm}2$}--(18,-4)--(18,-5);
\draw(4,2) node[above]{$\ell\hspace{-1.1mm}-\hspace{-1.1mm}1$}--(19,-4)--(19,-5);
\draw(5,2) node[above]{$\ell$}--(20,-4)--(20,-5);
\draw(6,2) node[above]{$\ell\hspace{-1.1mm}-\hspace{-1.1mm}2$}--(21,-4)--(21,-5);
\draw(7,2) node[above]{$\cdots$};
\draw(8,2) node[above]{$j\hspace{-0.8mm}+\hspace{-0.8mm}1$}--(23,-4)--(23,-5);
\draw(9,2) node[above]{$1$}--(24,-4)--(24,-5);
\draw(10,2) node[above]{$2$}--(25,-4)--(25,-5);
\draw(11,2) node[above]{$\cdots$};
\draw(12,2) node[above]{$i\hspace{-0.8mm}-\hspace{-0.8mm}1$}--(27,-4)--(27,-5);
\draw(13,2) node[above]{$i$}--(28,-4)--(28,-5);
\draw(14,2) node[above]{$j$}--(29,-4)--(29,-5);
\draw(15,2) node[above]{$0$}--(0,-4)--(0,-5);
\draw(16,2) node[above]{$2$}--(1,-4)--(1,-5);
\draw(17,2) node[above]{$\cdots$};
\draw(18,2) node[above]{$\ell\hspace{-1.1mm}-\hspace{-1.1mm}2$}--(3,-4)--(3,-5);
\draw(19,2) node[above]{$\ell\hspace{-1.1mm}-\hspace{-1.1mm}1$}--(4,-4)--(4,-5);
\draw(20,2) node[above]{$\ell$}--(5,-4)--(5,-5);
\draw(21,2) node[above]{$\ell\hspace{-1.1mm}-\hspace{-1.1mm}2$}--(6,-4)--(6,-5);
\draw(22,2) node[above]{$\cdots$};
\draw(23,2) node[above]{$j\hspace{-0.8mm}+\hspace{-0.8mm}1$}--(8,-4)--(8,-5);
\draw(24,2) node[above]{$j$}--(9,-4)--(14,-5);
\draw(25,2) node[above]{$1$}--(10,-4)--(9,-5);
\draw(26,2) node[above]{$2$}--(11,-4)--(10,-5);
\draw(27,2) node[above]{$\cdots$};
\draw(28,2) node[above]{$i\hspace{-0.8mm}-\hspace{-0.8mm}1$}--(13,-4)--(12,-5);
\draw(29,2) node[above]{$i$}--(14,-4)--(13,-5);
\end{braid}.
\end{align*}
The \(j\)-strand moves past the first \((i,i)\)-crossing, as the open term in the \((i,j,i)\)-braid relation is zero. This gives
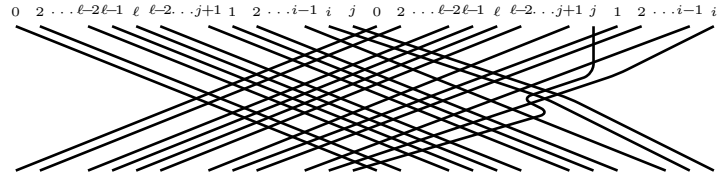
\begin{align*}
\begin{braid}\tikzset{baseline=0mm}
\draw(0,2) node[above]{$0$}--(15,-4);
\draw(1,2) node[above]{$2$}--(16,-4);
\draw(2,2) node[above]{$\cdots$};
\draw(3,2) node[above]{$\ell\hspace{-1.1mm}-\hspace{-1.1mm}2$}--(18,-4);
\draw(4,2) node[above]{$\ell\hspace{-1.1mm}-\hspace{-1.1mm}1$}--(19,-4);
\draw(5,2) node[above]{$\ell$}--(20,-4);
\draw(6,2) node[above]{$\ell\hspace{-1.1mm}-\hspace{-1.1mm}2$}--(21,-4);
\draw(7,2) node[above]{$\cdots$};
\draw(8,2) node[above]{$j\hspace{-0.8mm}+\hspace{-0.8mm}1$}--(23,-4);
\draw(9,2) node[above]{$1$}--(24,-4);
\draw(10,2) node[above]{$2$}--(25,-4);
\draw(11,2) node[above]{$\cdots$};
\draw(12,2) node[above]{$i\hspace{-0.8mm}-\hspace{-0.8mm}1$}--(27,-4);
\draw(13,2) node[above]{$i$}--(22,-1)--(28,-4);
\draw(14,2) node[above]{$j$}--(23,-1)--(29,-4);
\draw(15,2) node[above]{$0$}--(0,-4);
\draw(16,2) node[above]{$2$}--(1,-4);
\draw(17,2) node[above]{$\cdots$};
\draw(18,2) node[above]{$\ell\hspace{-1.1mm}-\hspace{-1.1mm}2$}--(3,-4);
\draw(19,2) node[above]{$\ell\hspace{-1.1mm}-\hspace{-1.1mm}1$}--(4,-4);
\draw(20,2) node[above]{$\ell$}--(5,-4);
\draw(21,2) node[above]{$\ell\hspace{-1.1mm}-\hspace{-1.1mm}2$}--(6,-4);
\draw(22,2) node[above]{$\cdots$};
\draw(23,2) node[above]{$j\hspace{-0.8mm}+\hspace{-0.8mm}1$}--(8,-4);
\draw(24,2) node[above]{$j$}--(24,0)--(21,-1)--(22.2,-1.6)--(14,-4);
\draw(25,2) node[above]{$1$}--(9,-4);
\draw(26,2) node[above]{$2$}--(10,-4);
\draw(27,2) node[above]{$\cdots$};
\draw(28,2) node[above]{$i\hspace{-0.8mm}-\hspace{-0.8mm}1$}--(12,-4);
\draw(29,2) node[above]{$i$}--(25,0)--(13,-4);
\end{braid}.
\end{align*}
Applying the braid relation to the \((i,j,i)\)-braid, we have \(\sigma(1 \otimes \psi_{j,i})v_{j,j}\), plus a remainder term. Now we simplify the remainder term. The \((i,j,i)\)-braid opens, followed by the \((i-1,i,i-1)\)- through \((1,2,1)\)-braids opening in succession. This gives
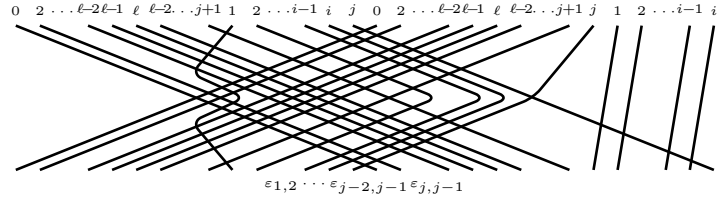
\begin{align*}
\begin{braid}\tikzset{baseline=0mm}
\draw(0,2) node[above]{$0$}--(15,-4);
\draw(1,2) node[above]{$2$}--(16,-4);
\draw(2,2) node[above]{$\cdots$};
\draw(3,2) node[above]{$\ell\hspace{-1.1mm}-\hspace{-1.1mm}2$}--(18,-4);
\draw(4,2) node[above]{$\ell\hspace{-1.1mm}-\hspace{-1.1mm}1$}--(19,-4);
\draw(5,2) node[above]{$\ell$}--(20,-4);
\draw(6,2) node[above]{$\ell\hspace{-1.1mm}-\hspace{-1.1mm}2$}--(21,-4);
\draw(7,2) node[above]{$\cdots$};
\draw(8,2) node[above]{$j\hspace{-0.8mm}+\hspace{-0.8mm}1$}--(23,-4);
\draw(9,2) node[above]{$1$}--(7.3,0)--(9.5,-1)--(7.3,-2)--(9,-4);
\draw(10,2) node[above]{$2$}--(17.5,-1)--(10,-4);
\draw(11,2) node[above]{$\cdots$};
\draw(12,2) node[above]{$i\hspace{-0.8mm}-\hspace{-0.8mm}1$}--(19.5,-1)--(12,-4);
\draw(13,2) node[above]{$i$}--(20.5,-1)--(13,-4);
\draw(14,2) node[above]{$j$}--(29,-4);
\draw(15,2) node[above]{$0$}--(0,-4);
\draw(16,2) node[above]{$2$}--(1,-4);
\draw(17,2) node[above]{$\cdots$};
\draw(18,2) node[above]{$\ell\hspace{-1.1mm}-\hspace{-1.1mm}2$}--(3,-4);
\draw(19,2) node[above]{$\ell\hspace{-1.1mm}-\hspace{-1.1mm}1$}--(4,-4);
\draw(20,2) node[above]{$\ell$}--(5,-4);
\draw(21,2) node[above]{$\ell\hspace{-1.1mm}-\hspace{-1.1mm}2$}--(6,-4);
\draw(22,2) node[above]{$\cdots$};
\draw(23,2) node[above]{$j\hspace{-0.8mm}+\hspace{-0.8mm}1$}--(8,-4);
\draw(24,2) node[above]{$j$}--(21.5,-1)--(14,-4);
\draw(25,2) node[above]{$1$}--(24,-4);
\draw(26,2) node[above]{$2$}--(25,-4);
\draw(27,2) node[above]{$\cdots$};
\draw(28,2) node[above]{$i\hspace{-0.8mm}-\hspace{-0.8mm}1$}--(27,-4);
\draw(29,2) node[above]{$i$}--(28,-4);
\draw(14.5,-4) node[below]{$\varepsilon_{1,2}\cdots \varepsilon_{j-2,j-1} \varepsilon_{j,j-1}$};
\end{braid}.
\end{align*}
Now the \((2,1,2)\)- and \((0,2,0)\)-braids open, followed by the \((3,2,3)\)- through \((\ell-1, \ell-2, \ell-1)\)-braids and the \((\ell, \ell-2, \ell)\)-braid, giving
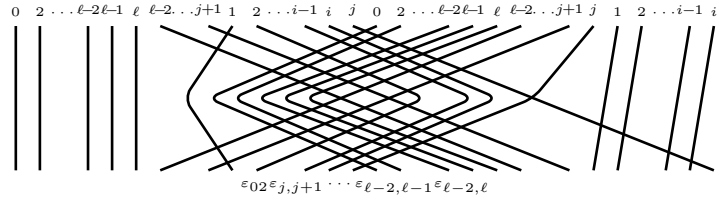
\begin{align*}
\begin{braid}\tikzset{baseline=0mm}
\draw(0,2) node[above]{$0$}--(0,-4);
\draw(1,2) node[above]{$2$}--(1,-4);
\draw(2,2) node[above]{$\cdots$};
\draw(3,2) node[above]{$\ell\hspace{-1.1mm}-\hspace{-1.1mm}2$}--(3,-4);
\draw(4,2) node[above]{$\ell\hspace{-1.1mm}-\hspace{-1.1mm}1$}--(4,-4);
\draw(5,2) node[above]{$\ell$}--(5,-4);
\draw(6,2) node[above]{$\ell\hspace{-1.1mm}-\hspace{-1.1mm}2$}--(21,-4);
\draw(7,2) node[above]{$\cdots$};
\draw(8,2) node[above]{$j\hspace{-0.8mm}+\hspace{-0.8mm}1$}--(23,-4);
\draw(9,2) node[above]{$1$}--(7,-1)--(9,-4);
\draw(10,2) node[above]{$2$}--(17,-1)--(10,-4);
\draw(11,2) node[above]{$\cdots$};
\draw(12,2) node[above]{$i\hspace{-0.8mm}-\hspace{-0.8mm}1$}--(19,-1)--(12,-4);
\draw(13,2) node[above]{$i$}--(20,-1)--(13,-4);
\draw(14,2) node[above]{$j$}--(29,-4);
\draw(15,2) node[above]{$0$}--(8,-1)--(15,-4);
\draw(16,2) node[above]{$2$}--(9,-1)--(16,-4);
\draw(17,2) node[above]{$\cdots$};
\draw(18,2) node[above]{$\ell\hspace{-1.1mm}-\hspace{-1.1mm}2$}--(10,-1)--(18,-4);
\draw(19,2) node[above]{$\ell\hspace{-1.1mm}-\hspace{-1.1mm}1$}--(11,-1)--(19,-4);
\draw(20,2) node[above]{$\ell$}--(12,-1)--(20,-4);
\draw(21,2) node[above]{$\ell\hspace{-1.1mm}-\hspace{-1.1mm}2$}--(6,-4);
\draw(22,2) node[above]{$\cdots$};
\draw(23,2) node[above]{$j\hspace{-0.8mm}+\hspace{-0.8mm}1$}--(8,-4);
\draw(24,2) node[above]{$j$}--(21.5,-1)--(14,-4);
\draw(25,2) node[above]{$1$}--(24,-4);
\draw(26,2) node[above]{$2$}--(25,-4);
\draw(27,2) node[above]{$\cdots$};
\draw(28,2) node[above]{$i\hspace{-0.8mm}-\hspace{-0.8mm}1$}--(27,-4);
\draw(29,2) node[above]{$i$}--(28,-4);
\draw(14.5,-4) node[below]{$\varepsilon_{02}\varepsilon_{j,j+1}\cdots\varepsilon_{\ell-2,\ell-1} \varepsilon_{\ell-2,\ell}
$};
\end{braid}.
\end{align*}
Now the \((\ell-2,\ell,\ell-2)\)-braid opens, followed by the \((\ell-3,\ell-2,\ell-3)\)- through \((j,j+1,j)\)-braids opening in succession, giving
 
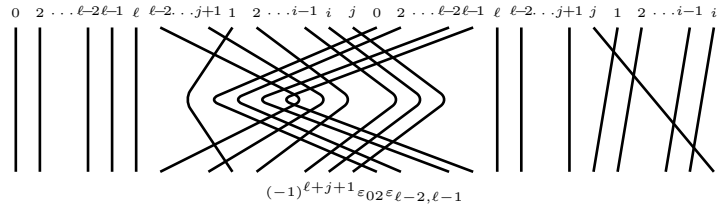
\begin{align*}
\begin{braid}\tikzset{baseline=0mm}
\draw(0,2) node[above]{$0$}--(0,-4);
\draw(1,2) node[above]{$2$}--(1,-4);
\draw(2,2) node[above]{$\cdots$};
\draw(3,2) node[above]{$\ell\hspace{-1.1mm}-\hspace{-1.1mm}2$}--(3,-4);
\draw(4,2) node[above]{$\ell\hspace{-1.1mm}-\hspace{-1.1mm}1$}--(4,-4);
\draw(5,2) node[above]{$\ell$}--(5,-4);
\draw(6,2) node[above]{$\ell\hspace{-1.1mm}-\hspace{-1.1mm}2$}--(12,-1)--(6,-4);
\draw(7,2) node[above]{$\cdots$};
\draw(8,2) node[above]{$j\hspace{-0.8mm}+\hspace{-0.8mm}1$}--(13,-1)--(8,-4);
\draw(9,2) node[above]{$1$}--(7,-1)--(9,-4);
\draw(10,2) node[above]{$2$}--(14,-1)--(10,-4);
\draw(11,2) node[above]{$\cdots$};
\draw(12,2) node[above]{$i\hspace{-0.8mm}-\hspace{-0.8mm}1$}--(16,-1)--(12,-4);
\draw(13,2) node[above]{$i$}--(17,-1)--(13,-4);
\draw(14,2) node[above]{$j$}--(18,-1)--(14,-4);
\draw(15,2) node[above]{$0$}--(8,-1)--(15,-4);
\draw(16,2) node[above]{$2$}--(9,-1)--(16,-4);
\draw(17,2) node[above]{$\cdots$};
\draw(18,2) node[above]{$\ell\hspace{-1.1mm}-\hspace{-1.1mm}2$}--(10,-1)--(18,-4);
\draw(19,2) node[above]{$\ell\hspace{-1.1mm}-\hspace{-1.1mm}1$}--(11,-1)--(19,-4);
\draw(20,2) node[above]{$\ell$}--(20,-4);
\draw(21,2) node[above]{$\ell\hspace{-1.1mm}-\hspace{-1.1mm}2$}--(21,-4);
\draw(22,2) node[above]{$\cdots$};
\draw(23,2) node[above]{$j\hspace{-0.8mm}+\hspace{-0.8mm}1$}--(23,-4);
\draw(24,2) node[above]{$j$}--(29,-4);
\draw(25,2) node[above]{$1$}--(24,-4);
\draw(26,2) node[above]{$2$}--(25,-4);
\draw(27,2) node[above]{$\cdots$};
\draw(28,2) node[above]{$i\hspace{-0.8mm}-\hspace{-0.8mm}1$}--(27,-4);
\draw(29,2) node[above]{$i$}--(28,-4);
\draw(14.5,-4) node[below]{$(-1)^{\ell +j+1}\varepsilon_{02}
\varepsilon_{\ell-2,\ell-1} 
$};
\end{braid}.
\end{align*}
Now the \((\ell-2,\ell-1)\)-double crossing opens. Then the \((\ell-2,\ell-3,\ell-2)\)-braid opens, followed by the \((\ell-3,\ell-2)\)-double crossing. This pattern repeats until the \((j+2,j+1,j+2)\)-braid opens, followed by the \((j+1,j+2)\)-double crossing. Then the \((j+1,j,j+1)\)-braid opens, which gives (omitting strands outside the central area)
 
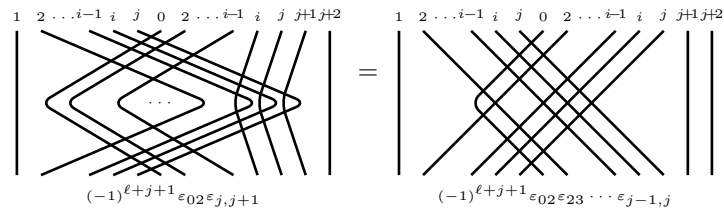
\begin{align*}
\begin{braid}\tikzset{baseline=0mm}
\draw(9,2) node[above]{$1$}--(9,-4);
\draw(10,2) node[above]{$2$}--(17,-1)--(10,-4);
\draw(11,2) node[above]{$\cdots$};
\draw(12,2) node[above]{$i\hspace{-0.8mm}-\hspace{-0.8mm}1$}--(19,-1)--(12,-4);
\draw(13,2) node[above]{$i$}--(20,-1)--(13,-4);
\draw(14,2) node[above]{$j$}--(21,-1)--(14,-4);
\draw(15,2) node[above]{$0$}--(10,-1)--(15,-4);
\draw(16,2) node[above]{$2$}--(11,-1)--(16,-4);
\draw(17,2) node[above]{$\cdots$};
\draw(18,2) node[above]{$i\hspace{-1.1mm}-\hspace{-1.1mm}1$}--(13,-1)--(18,-4);
\draw(19,2) node[above]{$i$}--(18,-1)--(19,-4);
\draw(20,2) node[above]{$j$}--(19,-1)--(20,-4);
\draw(21,2) node[above]{$j\hspace{-1.1mm}+\hspace{-1.1mm}1$}--(20,-1)--(21,-4);
\draw(22,2) node[above]{$j\hspace{-1.1mm}+\hspace{-1.1mm}2$}--(22,-4);
\draw(15.5,-4) node[below]{$(-1)^{\ell +j+1}\varepsilon_{02}
\varepsilon_{j,j+1}$};
\draw(15,-1) node{$\cdots$};
\end{braid}
=
\begin{braid}\tikzset{baseline=0mm}
\draw(9,2) node[above]{$1$}--(9,-4);
\draw(10,2) node[above]{$2$}--(16,-4);
\draw(11,2) node[above]{$\cdots$};
\draw(12,2) node[above]{$i\hspace{-0.8mm}-\hspace{-0.8mm}1$}--(18,-4);
\draw(13,2) node[above]{$i$}--(19,-4);
\draw(14,2) node[above]{$j$}--(20,-4);
\draw(15,2) node[above]{$0$}--(12,-1)--(15,-4);
\draw(16,2) node[above]{$2$}--(10,-4);
\draw(17,2) node[above]{$\cdots$};
\draw(18,2) node[above]{$i\hspace{-1.1mm}-\hspace{-1.1mm}1$}--(12,-4);
\draw(19,2) node[above]{$i$}--(13,-4);
\draw(20,2) node[above]{$j$}--(14,-4);
\draw(21,2) node[above]{$j\hspace{-1.1mm}+\hspace{-1.1mm}1$}--(21,-4);
\draw(22,2) node[above]{$j\hspace{-1.1mm}+\hspace{-1.1mm}2$}--(22,-4);
\draw(15.5,-4) node[below]{$(-1)^{\ell +j+1}\varepsilon_{02}
\varepsilon_{23}\cdots\varepsilon_{j-1,j} 
$};
\end{braid}
\end{align*}
after the \((2,3)\)- through \((j,j+1)\)-double crossings open. Now the \((2,0,2)\)-crossing opens, followed by the \((3,2,3)\)- through \((j,j-1,j)\)-braids, giving \((-1)^{\ell+j+1}(1 \otimes \psi_{j,i})v_{j,i}\), as desired.

\underline{Case \(m=j\), \({\tt C}={\tt D}^{(1)}_\ell\), \(\ell-2 \leq i,j \leq \ell\)}. We check that the claim holds in the case \(i=\ell-2, j=\ell\). The other cases are similar. We depict \((\psi_{j,i} \otimes 1)\sigma v_{j,i}\) diagrammatically: 
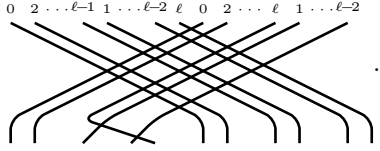
\begin{align*}
\begin{braid}\tikzset{baseline=0mm}
\draw(0,2) node[above]{$0$}--(8,-2)--(8,-3);
\draw(1,2) node[above]{$2$}--(9,-2)--(9,-3);
\draw(2,2) node[above]{$\cdots$};
\draw(3,2) node[above]{$\ell\hspace{-1mm}-\hspace{-1mm}1$}--(11,-2)--(11,-3);
\draw(4,2) node[above]{$1$}--(12,-2)--(12,-3);
\draw(5,2) node[above]{$\cdots$};
\draw(6,2) node[above]{$\ell\hspace{-1mm}-\hspace{-1mm}2$}--(14,-2)--(14,-3);
\draw(7,2) node[above]{$\ell$}--(15,-2)--(15,-3);
\draw(8,2) node[above]{$0$}--(0,-2)--(0,-3);
\draw(9,2) node[above]{$2$}--(1,-2)--(1,-3);
\draw(10,2) node[above]{$\cdots$};
\draw(11,2) node[above]{$\ell$}--(3,-2)--(6,-3);
\draw(12,2) node[above]{$1$}--(4,-2)--(3,-3);
\draw(13,2) node[above]{$\cdots$};
\draw(14,2) node[above]{$\ell\hspace{-1mm}-\hspace{-1mm}2$}--(6,-2)--(5,-3);
\end{braid}.
\end{align*}
The \(\ell\)-strand moves past the first \((\ell-2,\ell-2)\)-crossing, and applying the braid relation to the next \((\ell-2,\ell,\ell-2)\)-braid gives \((1 \otimes \psi_{\ell,\ell-2})v_{\ell,\ell-2} \) plus an error term:
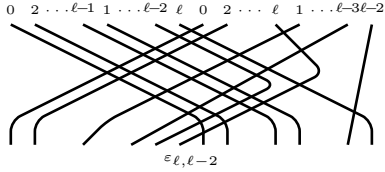
\begin{align*}
\begin{braid}\tikzset{baseline=0mm}
\draw(0,2) node[above]{$0$}--(8,-2)--(8,-3);
\draw(1,2) node[above]{$2$}--(9,-2)--(9,-3);
\draw(2,2) node[above]{$\cdots$};
\draw(3,2) node[above]{$\ell\hspace{-1mm}-\hspace{-1mm}1$}--(11,-2)--(11,-3);
\draw(4,2) node[above]{$1$}--(12,-2)--(12,-3);
\draw(5,2) node[above]{$\cdots$};
\draw(6,2) node[above]{$\ell\hspace{-1mm}-\hspace{-1mm}2$}--(11,-0.5)--(6,-3);
\draw(7,2) node[above]{$\ell$}--(15,-2)--(15,-3);
\draw(8,2) node[above]{$0$}--(0,-2)--(0,-3);
\draw(9,2) node[above]{$2$}--(1,-2)--(1,-3);
\draw(10,2) node[above]{$\cdots$};
\draw(11,2) node[above]{$\ell$}--(13,0)--(7,-3);
\draw(12,2) node[above]{$1$}--(4,-2)--(3,-3);
\draw(13,2) node[above]{$\cdots$};
\draw(14,2) node[above]{$\ell\hspace{-1mm}-\hspace{-1mm}3$}--(5,-3);
\draw(15,2) node[above]{$\ell\hspace{-1mm}-\hspace{-1mm}2$}--(14,-3);
\draw(7.5,-3) node[below]{$\varepsilon_{\ell,\ell-2}$}; 
\end{braid}.
\end{align*}
Now the \((\ell-3,\ell-2,\ell-3)\)- through \((1,2,1)\)-braids open in succession. Then the \((2,1,2)\)-braid and \((0,2,0)\)-braids open, followed by the \((3,2,3)\)- through \((\ell-1,\ell-2,\ell-1)\)-braids opening in succession, giving:
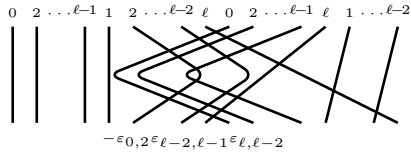
\begin{align*}
\begin{braid}\tikzset{baseline=0mm}
\draw(0,2) node[above]{$0$}--(0,-2);
\draw(1,2) node[above]{$2$}--(1,-2);
\draw(2,2) node[above]{$\cdots$};
\draw(3,2) node[above]{$\ell\hspace{-1mm}-\hspace{-1mm}1$}--(3,-2);
\draw(4,2) node[above]{$1$}--(4,-2);
\draw(5,2) node[above]{$2$}--(8,0)--(5,-2);
\draw(6,2) node[above]{$\cdots$};
\draw(7,2) node[above]{$\ell\hspace{-1mm}-\hspace{-1mm}2$}--(10,0)--(7,-2);
\draw(8,2) node[above]{$\ell$}--(16,-2);
\draw(9,2) node[above]{$0$}--(4,0)--(9,-2);
\draw(10,2) node[above]{$2$}--(5,0)--(10,-2);
\draw(11,2) node[above]{$\cdots$};
\draw(12,2) node[above]{$\ell\hspace{-1mm}-\hspace{-1mm}1$}--(7,0)--(12,-2);
\draw(13,2) node[above]{$\ell$}--(8,-2);
\draw(14,2) node[above]{$1$}--(13,-2);
\draw(15,2) node[above]{$\cdots$};
\draw(16,2) node[above]{$\ell\hspace{-1mm}-\hspace{-1mm}2$}--(15,-2);
\draw(7.5,-2) node[below]{$-\varepsilon_{0,2}\varepsilon_{\ell-2,\ell-1} \varepsilon_{\ell,\ell-2}$}; 
\end{braid}.
\end{align*}
Now the \((2,3)\)- through \((\ell-2,\ell-1)\)-double crossings open, introducing a \((2,0,2)\)-braid, which opens. Then the \((3,2,3)\)- through \((\ell-2,\ell-3,\ell-2)\)-braids open, followed by a \((\ell, \ell-2,\ell)\)-braid which opens, giving \(\xi_\ell(1 \otimes \psi_{\ell,\ell-2})v_{\ell, \ell-2}\), as desired.

\underline{Case \(m=i\), \({\tt C}={\tt D}^{(1)}_\ell\), \(1 \leq i,j \leq \ell-2\)}. We show that the claim holds in the case \(i=j+1\); the case \(i=j-1\) is similar. We depict \((\psi_{j,i} \otimes 1) \sigma v_{i,i}\) diagrammatically:
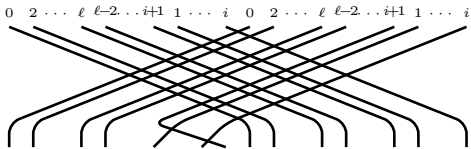
\begin{align*}
\begin{braid}\tikzset{baseline=0mm}
\draw(0,2) node[above]{$0$}--(10,-2)--(10,-3);
\draw(1,2) node[above]{$2$}--(11,-2)--(11,-3);
\draw(2,2) node[above]{$\cdots$};
\draw(3,2) node[above]{$\ell$}--(13,-2)--(13,-3);
\draw(4,2) node[above]{$\ell\hspace{-1mm}-\hspace{-1mm}2$}--(14,-2)--(14,-3);
\draw(5,2) node[above]{$\cdots$};
\draw(6,2) node[above]{$i\hspace{-1mm}+\hspace{-1mm}1$}--(16,-2)--(16,-3);
\draw(7,2) node[above]{$1$}--(17,-2)--(17,-3);
\draw(8,2) node[above]{$\cdots$};
\draw(9,2) node[above]{$i$}--(19,-2)--(19,-3);
\draw(10,2) node[above]{$0$}--(0,-2)--(0,-3);
\draw(11,2) node[above]{$2$}--(1,-2)--(1,-3);
\draw(12,2) node[above]{$\cdots$};
\draw(13,2) node[above]{$\ell$}--(3,-2)--(3,-3);
\draw(14,2) node[above]{$\ell\hspace{-1mm}-\hspace{-1mm}2$}--(4,-2)--(4,-3);
\draw(15,2) node[above]{$\cdots$};
\draw(16,2) node[above]{$i\hspace{-1mm}+\hspace{-1mm}1$}--(6,-2)--(9,-3);
\draw(17,2) node[above]{$1$}--(7,-2)--(6,-3);
\draw(18,2) node[above]{$\cdots$};
\draw(19,2) node[above]{$i$}--(9,-2)--(8,-3);
\end{braid}.
\end{align*}
Now the \((i+1)\)-strand moves up to the right past the first \((i,i)\)-crossing. Applying the braid relation to the next \((i,i+1,i)\)-braid gives \((1 \otimes \psi_{j,i})\sigma v_{i,i} \), plus a remainder term:
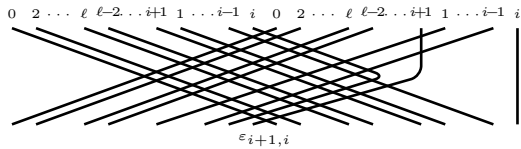
\begin{align*}
\begin{braid}\tikzset{baseline=0mm}
\draw(0,2) node[above]{$0$}--(11,-2);
\draw(1,2) node[above]{$2$}--(12,-2);
\draw(2,2) node[above]{$\cdots$};
\draw(3,2) node[above]{$\ell$}--(14,-2);
\draw(4,2) node[above]{$\ell\hspace{-1mm}-\hspace{-1mm}2$}--(15,-2);
\draw(5,2) node[above]{$\cdots$};
\draw(6,2) node[above]{$i\hspace{-1mm}+\hspace{-1mm}1$}--(17,-2);
\draw(7,2) node[above]{$1$}--(18,-2);
\draw(8,2) node[above]{$\cdots$};
\draw(9,2) node[above]{$i\hspace{-1mm}-\hspace{-1mm}1$}--(20,-2);
\draw(10,2) node[above]{$i$}--(15.5,0)--(9,-2);
\draw(11,2) node[above]{$0$}--(0,-2);
\draw(12,2) node[above]{$2$}--(1,-2);
\draw(13,2) node[above]{$\cdots$};
\draw(14,2) node[above]{$\ell$}--(3,-2);
\draw(15,2) node[above]{$\ell\hspace{-1mm}-\hspace{-1mm}2$}--(4,-2);
\draw(16,2) node[above]{$\cdots$};
\draw(17,2) node[above]{$i\hspace{-1mm}+\hspace{-1mm}1$}--(17,0)--(10,-2);
\draw(18,2) node[above]{$1$}--(6,-2);
\draw(19,2) node[above]{$\cdots$};
\draw(20,2) node[above]{$i\hspace{-1mm}-\hspace{-1mm}1$}--(8,-2);
\draw(21,2) node[above]{$i$}--(21,-2);
\draw(10.5,-2) node[below]{$\varepsilon_{i+1,i}$}; 
\end{braid}.
\end{align*}
Dragging the \(i\)-strand to the left, the \((i-1,i,i-1)\)- through \((1,2,1)\)-braids open in succession, followed by the \((2,1,2)\)- and \((0,2,0)\)-braids. Then the \((3,2,3)\)- through \((\ell-1,\ell-2,\ell-1)\)-braids open in succession, followed by the \((\ell,\ell-2,\ell)\)-braid, giving (omitting straight strands outside the central area): 
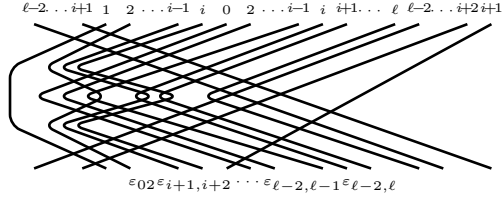
\begin{align*}
\begin{braid}\tikzset{baseline=0mm}
\draw(0,3) node[above]{$\ell\hspace{-1mm}-\hspace{-1mm}2$}--(17,-3);
\draw(1,3) node[above]{$\cdots$};
\draw(2,3) node[above]{$i\hspace{-1mm}+\hspace{-1mm}1$}--(19,-3);
\draw(3,3) node[above]{$1$}--(-1,1.2)--(-1,-1.2)--(3,-3);
\draw(4,3) node[above]{$2$}--(0.4,1.2)--(3,0)--(0.4,-1.2)--(4,-3);
\draw(5,3) node[above]{$\cdots$};
\draw(6,3) node[above]{$i\hspace{-1mm}-\hspace{-1mm}1$}--(1,1.2)--(5,0)--(1,-1.2)--(6,-3);
\draw(7,3) node[above]{$i$}--(1.6,1.2)--(6,0)--(1.6,-1.2)--(7,-3);
\draw(8,3) node[above]{$0$}--(0,0)--(9,-3);
\draw(9,3) node[above]{$2$}--(1,0)--(10,-3);
\draw(10,3) node[above]{$\cdots$};
\draw(11,3) node[above]{$i\hspace{-1mm}-\hspace{-1mm}1$}--(2,0)--(12,-3);
\draw(12,3) node[above]{$i$}--(4,0)--(13,-3);
\draw(13,3) node[above]{$i\hspace{-1mm}+\hspace{-1mm}1$}--(5,0)--(14,-3);
\draw(14,3) node[above]{$\cdots$};
\draw(15,3) node[above]{$\ell$}--(7,0)--(16,-3);
\draw(16,3) node[above]{$\ell\hspace{-1mm}-\hspace{-1mm}2$}--(0,-3);
\draw(17,3) node[above]{$\cdots$};
\draw(18,3) node[above]{$i\hspace{-1mm}+\hspace{-1mm}2$}--(2,-3);
\draw(19,3) node[above]{$i\hspace{-1mm}+\hspace{-1mm}1$}--(8,-3);
\draw(9.5,-3) node[below]{$\varepsilon_{02}\varepsilon_{i+1,i+2}\cdots \varepsilon_{\ell-2,\ell-1}\varepsilon_{\ell-2,\ell}$}; 
\end{braid}.
\end{align*}
Now the \((2,3)\)- through \((i,i+1)\)-double crossings open, introducing a \((2,0,2)\)-braid which opens, followed by \((3,2,3)\)- through \((i,i-1,i)\)-braids which open, giving:
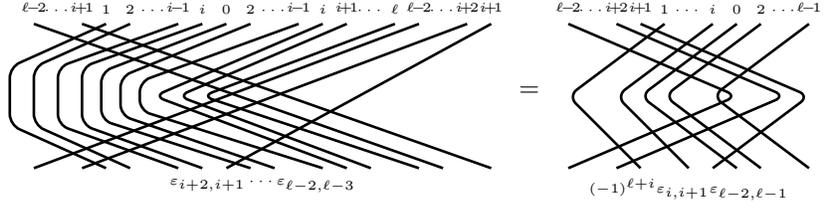
\begin{align*}
\begin{braid}\tikzset{baseline=0mm}
\draw(0,3) node[above]{$\ell\hspace{-1mm}-\hspace{-1mm}2$}--(17,-3);
\draw(1,3) node[above]{$\cdots$};
\draw(2,3) node[above]{$i\hspace{-1mm}+\hspace{-1mm}1$}--(19,-3);
\draw(3,3) node[above]{$1$}--(-1,1.2)--(-1,-1.2)--(3,-3);
\draw(4,3) node[above]{$2$}--(0,1.2)--(0,-1.2)--(4,-3);
\draw(5,3) node[above]{$\cdots$};
\draw(6,3) node[above]{$i\hspace{-1mm}-\hspace{-1mm}1$}--(1,1.2)--(1,-1.2)--(6,-3);
\draw(7,3) node[above]{$i$}--(2,1.2)--(2,-1.2)--(7,-3);
\draw(8,3) node[above]{$0$}--(2.8,1)--(2.8,-1)--(9,-3);
\draw(9,3) node[above]{$2$}--(3.6,0.8)--(3.6,-0.8)--(10,-3);
\draw(10,3) node[above]{$\cdots$};
\draw(11,3) node[above]{$i\hspace{-1mm}-\hspace{-1mm}1$}--(4.4,0.6)--(4.4,-0.6)--(12,-3);
\draw(12,3) node[above]{$i$}--(5,0)--(13,-3);
\draw(13,3) node[above]{$i\hspace{-1mm}+\hspace{-1mm}1$}--(6,0)--(14,-3);
\draw(14,3) node[above]{$\cdots$};
\draw(15,3) node[above]{$\ell$}--(7,0)--(16,-3);
\draw(16,3) node[above]{$\ell\hspace{-1mm}-\hspace{-1mm}2$}--(0,-3);
\draw(17,3) node[above]{$\cdots$};
\draw(18,3) node[above]{$i\hspace{-1mm}+\hspace{-1mm}2$}--(2,-3);
\draw(19,3) node[above]{$i\hspace{-1mm}+\hspace{-1mm}1$}--(8,-3);
\draw(9.5,-3) node[below]{$\varepsilon_{i+2,i+1} \cdots \varepsilon_{\ell-2,\ell-3}$}; 
\end{braid}
=
\begin{braid}\tikzset{baseline=0mm}
\draw(-1,3) node[above]{$\ell\hspace{-1mm}-\hspace{-1mm}2$}--(6,0)--(-1,-3);
\draw(0,3) node[above]{$\cdots$};
\draw(1,3) node[above]{$i\hspace{-1mm}+\hspace{-1mm}2$}--(8,0)--(1,-3);
\draw(2,3) node[above]{$i\hspace{-1mm}+\hspace{-1mm}1$}--(9,0)--(5,-3);
\draw(3,3) node[above]{$1$}--(-1,0)--(2,-3);
\draw(4,3) node[above]{$\cdots$};
\draw(5,3) node[above]{$i$}--(1,0)--(4,-3);
\draw(6,3) node[above]{$0$}--(2,0)--(6,-3);
\draw(7,3) node[above]{$2$}--(3,0)--(7,-3);
\draw(8,3) node[above]{$\cdots$};
\draw(9,3) node[above]{$\ell\hspace{-1mm}-\hspace{-1mm}1$}--(5,0)--(9,-3);
\draw(4,-3) node[below]{$(-1)^{\ell+i}\varepsilon_{i,i+1} \varepsilon_{\ell-2,\ell-1}$}; 
\end{braid},
\end{align*}
after the \((\ell-2,\ell,\ell-2)\)-braid opens, followed by the \((\ell-3,\ell-2,\ell-3)\)-through \((i+1,i+2,i+1)\)-braids in succession. Now the \((\ell-2,\ell)\)-double crossing opens. The \((\ell-2,\ell-3,\ell-2)\)-braid opens, followed by an \((\ell-2,\ell-3)\)-double crossing which opens. This sequence repeats until the \((i+2,i+1,i+2)\)-braid opens, followed by an \((i+2,i+1)\)-double crossing which opens. Finally, the \((i+1,i,i+1)\)-braid opens, giving \(-\xi_i(\psi_{j,i} \otimes 1)v_{i,i}\), as desired. 

\underline{Case \(m=i\), \({\tt C}={\tt D}^{(1)}_\ell\), \(\ell-2 \leq i,j \leq \ell\)}. We show that the claim holds in the case \(i=\ell-2, j=\ell\); the other cases are similar. We depict \((\psi_{j,i} \otimes 1) \sigma v_{i,i}\) diagrammatically:
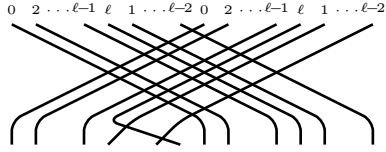
\begin{align*}
\begin{braid}\tikzset{baseline=0mm}
\draw(0,2) node[above]{$0$}--(8,-2)--(8,-3);
\draw(1,2) node[above]{$2$}--(9,-2)--(9,-3);
\draw(2,2) node[above]{$\cdots$};
\draw(3,2) node[above]{$\ell\hspace{-1mm}-\hspace{-1mm}1$}--(11,-2)--(11,-3);
\draw(4,2) node[above]{$\ell$}--(12,-2)--(12,-3);
\draw(5,2) node[above]{$1$}--(13,-2)--(13,-3);
\draw(6,2) node[above]{$\cdots$};
\draw(7,2) node[above]{$\ell\hspace{-1mm}-\hspace{-1mm}2$}--(15,-2)--(15,-3);
\draw(8,2) node[above]{$0$}--(0,-2)--(0,-3);
\draw(9,2) node[above]{$2$}--(1,-2)--(1,-3);
\draw(10,2) node[above]{$\cdots$};
\draw(11,2) node[above]{$\ell\hspace{-1mm}-\hspace{-1mm}1$}--(3,-2)--(3,-3);
\draw(12,2) node[above]{$\ell$}--(4,-2)--(7,-3);
\draw(13,2) node[above]{$1$}--(5,-2)--(4,-3);
\draw(14,2) node[above]{$\cdots$};
\draw(15,2) node[above]{$\ell\hspace{-1mm}-\hspace{-1mm}2$}--(7,-2)--(6,-3);
\end{braid}.
\end{align*}
The \(\ell\)-strand moves up past the first \((\ell-2,\ell-2)\)-crossing. Applying the braid relation to the next \((\ell-2,\ell,\ell-2)\)-braid gives \(\sigma(1 \otimes \psi_{\ell,\ell-2})v_{\ell-2, \ell-2} \), plus an error term:
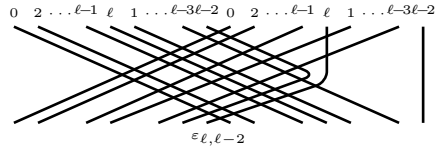
\begin{align*}
\begin{braid}\tikzset{baseline=0mm}
\draw(0,2) node[above]{$0$}--(9,-2);
\draw(1,2) node[above]{$2$}--(10,-2);
\draw(2,2) node[above]{$\cdots$};
\draw(3,2) node[above]{$\ell\hspace{-1mm}-\hspace{-1mm}1$}--(12,-2);
\draw(4,2) node[above]{$\ell$}--(13,-2);
\draw(5,2) node[above]{$1$}--(14,-2);
\draw(6,2) node[above]{$\cdots$};
\draw(7,2) node[above]{$\ell\hspace{-1mm}-\hspace{-1mm}3$}--(16,-2);
\draw(8,2) node[above]{$\ell\hspace{-1mm}-\hspace{-1mm}2$}--(12.5,0)--(7,-2);
\draw(9,2) node[above]{$0$}--(0,-2);
\draw(10,2) node[above]{$2$}--(1,-2);
\draw(11,2) node[above]{$\cdots$};
\draw(12,2) node[above]{$\ell\hspace{-1mm}-\hspace{-1mm}1$}--(3,-2);
\draw(13,2) node[above]{$\ell$}--(13,-0.3)--(8,-2);
\draw(14,2) node[above]{$1$}--(4,-2);
\draw(15,2) node[above]{$\cdots$};
\draw(16,2) node[above]{$\ell\hspace{-1mm}-\hspace{-1mm}3$}--(6,-2);
\draw(17,2) node[above]{$\ell\hspace{-1mm}-\hspace{-1mm}2$}--(17,-2);
\draw(8.5,-2) node[below]{$\varepsilon_{\ell,\ell-2}$}; 
\end{braid}.
\end{align*}
Now we simplify this error term. The \((\ell-3,\ell-2,\ell-3)\)- through \((1,2,1)\)-braids open in succession, giving (omitting straight strands to the right):
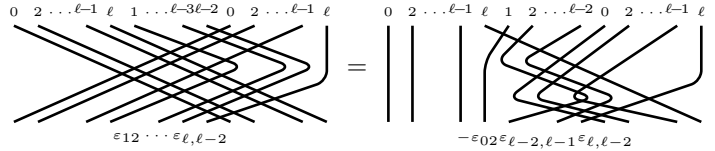
\begin{align*}
\begin{braid}\tikzset{baseline=0mm}
\draw(0,2) node[above]{$0$}--(9,-2);
\draw(1,2) node[above]{$2$}--(10,-2);
\draw(2,2) node[above]{$\cdots$};
\draw(3,2) node[above]{$\ell\hspace{-1mm}-\hspace{-1mm}1$}--(12,-2);
\draw(4,2) node[above]{$\ell$}--(13,-2);
\draw(5,2) node[above]{$1$}--(9.5,0.3)--(4,-2);
\draw(6,2) node[above]{$\cdots$};
\draw(7,2) node[above]{$\ell\hspace{-1mm}-\hspace{-1mm}3$}--(11.5,0.3)--(6,-2);
\draw(8,2) node[above]{$\ell\hspace{-1mm}-\hspace{-1mm}2$}--(12.5,0.3)--(7,-2);
\draw(9,2) node[above]{$0$}--(0,-2);
\draw(10,2) node[above]{$2$}--(1,-2);
\draw(11,2) node[above]{$\cdots$};
\draw(12,2) node[above]{$\ell\hspace{-1mm}-\hspace{-1mm}1$}--(3,-2);
\draw(13,2) node[above]{$\ell$}--(13,-0.3)--(8,-2);
\draw(6.5,-2) node[below]{$\varepsilon_{12} \cdots \varepsilon_{\ell,\ell-2}$}; 
\end{braid}
=
\begin{braid}\tikzset{baseline=0mm}
\draw(0,2) node[above]{$0$}--(0,-2);
\draw(1,2) node[above]{$2$}--(1,-2);
\draw(2,2) node[above]{$\cdots$};
\draw(3,2) node[above]{$\ell\hspace{-1mm}-\hspace{-1mm}1$}--(3,-2);
\draw(4,2) node[above]{$\ell$}--(13,-2);
\draw(5,2) node[above]{$1$}--(4,0.5)--(4,-2);
\draw(6,2) node[above]{$2$}--(4.5,0.5)--(8.5,-1)--(5,-2);
\draw(7,2) node[above]{$\cdots$};
\draw(8,2) node[above]{$\ell\hspace{-1mm}-\hspace{-1mm}2$}--(5.5,0.5)--(9.5,-1)--(7,-2);
\draw(9,2) node[above]{$0$}--(5,-1)--(9,-2);
\draw(10,2) node[above]{$2$}--(6,-1)--(10,-2);
\draw(11,2) node[above]{$\cdots$};
\draw(12,2) node[above]{$\ell\hspace{-1mm}-\hspace{-1mm}1$}--(7.5,-1)--(12,-2);
\draw(13,2) node[above]{$\ell$}--(13,-0.3)--(8,-2);
\draw(6.5,-2) node[below]{$-\varepsilon_{02}\varepsilon_{\ell-2,\ell-1}\varepsilon_{\ell,\ell-2}$}; 
\end{braid},
\end{align*}
after the \((2,1,2)\)- and \((0,2,0)\)-braids open, followed by the \((3,2,3)\)- through \((\ell-1,\ell-2,\ell-1)\)-braids opening in succession. Now, the \((2,3)\)- through \((\ell-2,\ell-1)\)-braids open in succession. Then the \((2,0,2)\)- braid opens, followed by the \((3,2,3)\)- through \((\ell-2,\ell-3,\ell-2)\)-braids opening in succession. Finally the \((\ell,\ell-2,\ell)\)-braid opens, giving \((\psi_{\ell,\ell-2} \otimes 1)v_{\ell-2, \ell-2}\), as desired.

\underline{Case \(j \neq m \neq i\), all types}. We may write 
\begin{align*}
(\psi_{j,i} \otimes 1)\sigma v_{m,i} = \sigma(1 \otimes \psi_{j,i})v_{m,i} + (*),
\end{align*}
where \((*)\) is a linear combination of terms of the form \(1_{\bb^j\bb^m}\psi_w (x_1 \otimes x_2)\), where \(x_1 \in \Delta_{\delta,m}\), \(x_2 \in \Delta_{\delta,i}\), and \(w\triangleleft \sigma\) is a minimal left coset representative for \(\mathfrak{S}_{2d}/\mathfrak{S}_{d} \times \mathfrak{S}_d\). As in the similar case in Lemma \ref{sigmaprime}, it follows that \(\psi_w = 1\). Thus \(x_1\) is a vector of word \(\bb^j\) and \(x_2\) is a vector of word \(\bb^m\). Hence by Lemma \ref{MinDelHom}, it follows that \((*)\) is zero unless \(m\) neighbors both \(j\) and \(i\). But since \(i\) neighbors \(j\) by assumption, this cannot be the case.
\end{proof}

\end{document}